\documentclass{amsart}
\title[Characterizations of Steiner and linear isometries operads]{Characterizations of equivariant Steiner and linear isometries operads}

\author{Jonathan Rubin}
\address{University of California Los Angeles,
Los Angeles, CA 90095}
\email{jrubin@math.ucla.edu}

\subjclass[2010]{Primary: 55P91}

\date{\today}
\usepackage{amsmath,amssymb,amsthm,stackrel,tikz,enumitem,mathrsfs,stmaryrd,physics}

\renewcommand{\c}[1]{\mathcal{#1}}
\renewcommand{\t}[1]{\textnormal{#1}}
\renewcommand{\b}[1]{\mathbf{#1}}
\newcommand{\bb}[1]{\mathbb{#1}}
\newcommand{\vp}{\varphi}
\newcommand{\la}{\langle}
\newcommand{\ra}{\rangle}

\newcommand{\ol}[1]{\overline{#1}}
\newcommand{\ub}[1]{\underline{\mathbf{#1}}}
\newcommand{\s}[1]{\mathscr{#1}}

\newcommand{\cpa}{
	\begin{tikzpicture}[scale=0.15,baseline=0.2mm]
		\node(0) at (0,0) {$\cdot$};
		\node(1) at (0,1) {$\cdot$};
		
	\end{tikzpicture}
}

\newcommand{\cpb}{
	\begin{tikzpicture}[scale=0.15,baseline=0.2mm]
		\node(0) at (0,0) {$\cdot$};
		\node(1) at (0,1) {$\cdot$};
		
		\draw(0,0) -- (0,1);
	\end{tikzpicture}
}

\newcommand{\cppa}{
	\begin{tikzpicture}[scale=0.15,baseline=0.2mm]
		\node(0) at (0,0) {$\cdot$};
		\node(1) at (0,1) {$\cdot$};
		\node(2) at (0,2) {$\cdot$};
		
	\end{tikzpicture}
}

\newcommand{\cppb}{
	\begin{tikzpicture}[scale=0.15,baseline=0.2mm]
		\node(0) at (0,0) {$\cdot$};
		\node(1) at (0,1) {$\cdot$};
		\node(2) at (0,2) {$\cdot$};
		
		\draw (0,0) -- (0,1);
	\end{tikzpicture}
}

\newcommand{\cppc}{
	\begin{tikzpicture}[scale=0.15,baseline=0.2mm]
		\node(0) at (0,0) {$\cdot$};
		\node(1) at (0,1) {$\cdot$};
		\node(2) at (0,2) {$\cdot$};
		
		\draw (0,0) -- (0,1);
		\draw plot [smooth, tension=1.5] coordinates {(0,0) (0.5,1) (0,2)};
	\end{tikzpicture}
}

\newcommand{\cppd}{
	\begin{tikzpicture}[scale=0.15,baseline=0.2mm]
		\node(0) at (0,0) {$\cdot$};
		\node(1) at (0,1) {$\cdot$};
		\node(2) at (0,2) {$\cdot$};
		
		\draw (0,1) -- (0,2);
	\end{tikzpicture}
}

\newcommand{\cppe}{
	\begin{tikzpicture}[scale=0.15,baseline=0.2mm]
		\node(0) at (0,0) {$\cdot$};
		\node(1) at (0,1) {$\cdot$};
		\node(2) at (0,2) {$\cdot$};
		
		\draw (0,0) -- (0,1);
		\draw plot [smooth, tension=1.5] coordinates {(0,0) (0.5,1) (0,2)};
		\draw (0,1) -- (0,2);
	\end{tikzpicture}
}

\newcommand{\cpppa}{
	\begin{tikzpicture}[scale=0.15,baseline=0.2mm]
		\node(0) at (0,0) {$\cdot$};
		\node(1) at (0,1) {$\cdot$};
		\node(2) at (0,2) {$\cdot$};
		\node(3) at (0,3) {$\cdot$};
		
	\end{tikzpicture}
}

\newcommand{\cpppb}{
	\begin{tikzpicture}[scale=0.15,baseline=0.2mm]
		\node(0) at (0,0) {$\cdot$};
		\node(1) at (0,1) {$\cdot$};
		\node(2) at (0,2) {$\cdot$};
		\node(3) at (0,3) {$\cdot$};
		
		\draw (0,0) -- (0,1);
	\end{tikzpicture}
}

\newcommand{\cpppc}{
	\begin{tikzpicture}[scale=0.15,baseline=0.2mm]
		\node(0) at (0,0) {$\cdot$};
		\node(1) at (0,1) {$\cdot$};
		\node(2) at (0,2) {$\cdot$};
		\node(3) at (0,3) {$\cdot$};
		
		\draw (0,0) -- (0,1);
		\draw plot [smooth, tension=1.5] coordinates {(0,0) (0.5,1) (0,2)};
	\end{tikzpicture}
}

\newcommand{\cpppd}{
	\begin{tikzpicture}[scale=0.15,baseline=0.2mm]
		\node(0) at (0,0) {$\cdot$};
		\node(1) at (0,1) {$\cdot$};
		\node(2) at (0,2) {$\cdot$};
		\node(3) at (0,3) {$\cdot$};
		
		\draw (0,1) -- (0,2);
	\end{tikzpicture}
}

\newcommand{\cpppe}{
	\begin{tikzpicture}[scale=0.15,baseline=0.2mm]
		\node(0) at (0,0) {$\cdot$};
		\node(1) at (0,1) {$\cdot$};
		\node(2) at (0,2) {$\cdot$};
		\node(3) at (0,3) {$\cdot$};
		
		\draw (0,0) -- (0,1);
		\draw plot [smooth, tension=1.5] coordinates {(0,0) (0.5,1) (0,2)};
		\draw (0,1) -- (0,2);
	\end{tikzpicture}
}

\newcommand{\cpppf}{
	\begin{tikzpicture}[scale=0.15,baseline=0.2mm]
		\node(0) at (0,0) {$\cdot$};
		\node(1) at (0,1) {$\cdot$};
		\node(2) at (0,2) {$\cdot$};
		\node(3) at (0,3) {$\cdot$};
		
		\draw (0,0) -- (0,1);
		\draw plot [smooth, tension=1.5] coordinates {(0,0) (0.5,1) (0,2)};
		\draw plot [smooth, tension=1.5] coordinates {(0,0) (1,1.5) (0,3)};
	\end{tikzpicture}
}

\newcommand{\cpppg}{
	\begin{tikzpicture}[scale=0.15,baseline=0.2mm]
		\node(0) at (0,0) {$\cdot$};
		\node(1) at (0,1) {$\cdot$};
		\node(2) at (0,2) {$\cdot$};
		\node(3) at (0,3) {$\cdot$};
		
		\draw (0,0) -- (0,1);
		\draw plot [smooth, tension=1.5] coordinates {(0,0) (0.5,1) (0,2)};
		\draw (0,1) -- (0,2);
		\draw plot [smooth, tension=1.5] coordinates {(0,0) (1,1.5) (0,3)};
	\end{tikzpicture}
}

\newcommand{\cppph}{
	\begin{tikzpicture}[scale=0.15,baseline=0.2mm]
		\node(0) at (0,0) {$\cdot$};
		\node(1) at (0,1) {$\cdot$};
		\node(2) at (0,2) {$\cdot$};
		\node(3) at (0,3) {$\cdot$};
		
		\draw (0,1) -- (0,2);
		\draw plot [smooth, tension=1.5] coordinates {(0,1) (-0.5,2) (0,3)};
	\end{tikzpicture}
}

\newcommand{\cpppi}{
	\begin{tikzpicture}[scale=0.15,baseline=0.2mm]
		\node(0) at (0,0) {$\cdot$};
		\node(1) at (0,1) {$\cdot$};
		\node(2) at (0,2) {$\cdot$};
		\node(3) at (0,3) {$\cdot$};
		
		\draw (0,0) -- (0,1);
		\draw plot [smooth, tension=1.5] coordinates {(0,0) (0.5,1) (0,2)};
		\draw (0,1) -- (0,2);
		\draw plot [smooth, tension=1.5] coordinates {(0,0) (1,1.5) (0,3)};
		\draw plot [smooth, tension=1.5] coordinates {(0,1) (-0.5,2) (0,3)};
	\end{tikzpicture}
}

\newcommand{\cpppj}{
	\begin{tikzpicture}[scale=0.15,baseline=0.2mm]
		\node(0) at (0,0) {$\cdot$};
		\node(1) at (0,1) {$\cdot$};
		\node(2) at (0,2) {$\cdot$};
		\node(3) at (0,3) {$\cdot$};
		
		\draw (0,2) -- (0,3);
	\end{tikzpicture}
}

\newcommand{\cpppk}{
	\begin{tikzpicture}[scale=0.15,baseline=0.2mm]
		\node(0) at (0,0) {$\cdot$};
		\node(1) at (0,1) {$\cdot$};
		\node(2) at (0,2) {$\cdot$};
		\node(3) at (0,3) {$\cdot$};
		
		\draw (0,1) -- (0,2);
		\draw plot [smooth, tension=1.5] coordinates {(0,1) (-0.5,2) (0,3)};
		\draw (0,2) -- (0,3);
	\end{tikzpicture}
}

\newcommand{\cpppl}{
	\begin{tikzpicture}[scale=0.15,baseline=0.2mm]
		\node(0) at (0,0) {$\cdot$};
		\node(1) at (0,1) {$\cdot$};
		\node(2) at (0,2) {$\cdot$};
		\node(3) at (0,3) {$\cdot$};
		
		\draw (0,0) -- (0,1);
		\draw (0,2) -- (0,3);
	\end{tikzpicture}
}

\newcommand{\cpppm}{
	\begin{tikzpicture}[scale=0.15,baseline=0.2mm]
		\node(0) at (0,0) {$\cdot$};
		\node(1) at (0,1) {$\cdot$};
		\node(2) at (0,2) {$\cdot$};
		\node(3) at (0,3) {$\cdot$};
		
		\draw (0,0) -- (0,1);
		\draw plot [smooth, tension=1.5] coordinates {(0,0) (0.5,1) (0,2)};
		\draw plot [smooth, tension=1.5] coordinates {(0,0) (1,1.5) (0,3)};
		\draw (0,2) -- (0,3);
	\end{tikzpicture}
}

\newcommand{\cpppn}{
	\begin{tikzpicture}[scale=0.15,baseline=0.2mm]
		\node(0) at (0,0) {$\cdot$};
		\node(1) at (0,1) {$\cdot$};
		\node(2) at (0,2) {$\cdot$};
		\node(3) at (0,3) {$\cdot$};
		
		\draw (0,0) -- (0,1);
		\draw plot [smooth, tension=1.5] coordinates {(0,0) (0.5,1) (0,2)};
		\draw (0,1) -- (0,2);
		\draw plot [smooth, tension=1.5] coordinates {(0,0) (1,1.5) (0,3)};
		\draw plot [smooth, tension=1.5] coordinates {(0,1) (-0.5,2) (0,3)};
		\draw (0,2) -- (0,3);
	\end{tikzpicture}
}

\newcommand{\indt}{
	\begin{tikzpicture}[scale=0.12,baseline=0.2mm]
		\node(S) at (0,0) {$\cdot$};
		\node(W) at (-1,1) {$\cdot$};
		\node(E) at (1,1) {$\cdot$};
		\node(N) at (0,2) {$\cdot$};
		\path[-]
		
		;
	\end{tikzpicture}
}
\newcommand{\indp}{
	\begin{tikzpicture}[scale=0.12,baseline=0.2mm]
		\node(S) at (0,0) {$\cdot$};
		\node(W) at (-1,1) {$\cdot$};
		\node(E) at (1,1) {$\cdot$};
		\node(N) at (0,2) {$\cdot$};
		\draw (0,0) -- (-1,1) ;
	\end{tikzpicture}
}
\newcommand{\indq}{
	\begin{tikzpicture}[scale=0.12,baseline=0.2mm]
		\node(S) at (0,0) {$\cdot$};
		\node(W) at (-1,1) {$\cdot$};
		\node(E) at (1,1) {$\cdot$};
		\node(N) at (0,2) {$\cdot$};
		\draw (0,0) -- (1,1) ;
	\end{tikzpicture}
}
\newcommand{\indpq}{
	\begin{tikzpicture}[scale=0.12,baseline=0.2mm]
		\node(S) at (0,0) {$\cdot$};
		\node(W) at (-1,1) {$\cdot$};
		\node(E) at (1,1) {$\cdot$};
		\node(N) at (0,2) {$\cdot$};
		\draw (0,0) -- (-1,1) ;
		\draw (0,0) -- (1,1) ;
	\end{tikzpicture}
}
\newcommand{\indpqp}{
	\begin{tikzpicture}[scale=0.12,baseline=0.2mm]
		\node(S) at (0,0) {$\cdot$};
		\node(W) at (-1,1) {$\cdot$};
		\node(E) at (1,1) {$\cdot$};
		\node(N) at (0,2) {$\cdot$};
		\draw (0,0) -- (1,1) ;
		\draw (-1,1) -- (0,2) ;
	\end{tikzpicture}
}
\newcommand{\indpqq}{
	\begin{tikzpicture}[scale=0.12,baseline=0.2mm]
		\node(S) at (0,0) {$\cdot$};
		\node(W) at (-1,1) {$\cdot$};
		\node(E) at (1,1) {$\cdot$};
		\node(N) at (0,2) {$\cdot$};
		\draw (0,0) -- (-1,1) ;
		\draw (1,1) -- (0,2) ;
	\end{tikzpicture}
}
\newcommand{\inds}{
	\begin{tikzpicture}[scale=0.12,baseline=0.2mm]
		\node(S) at (0,0) {$\cdot$};
		\node(W) at (-1,1) {$\cdot$};
		\node(E) at (1,1) {$\cdot$};
		\node(N) at (0,2) {$\cdot$};
		\draw (0,0) -- (-1,1) ;
		\draw (0,0) -- (1,1) ;
		\draw (0,0) -- (0,2) ;
		\draw (-1,1) -- (0,2) ;
		\draw (1,1) -- (0,2) ;
	\end{tikzpicture}
}
\newcommand{\indpqf}{
	\begin{tikzpicture}[scale=0.12,baseline=0.2mm]
		\node(S) at (0,0) {$\cdot$};
		\node(W) at (-1,1) {$\cdot$};
		\node(E) at (1,1) {$\cdot$};
		\node(N) at (0,2) {$\cdot$};
		\draw (0,0) -- (-1,1) ;
		\draw (0,0) -- (1,1) ;
		\draw (0,0) -- (0,2) ;
	\end{tikzpicture}
}
\newcommand{\indpqfp}{
	\begin{tikzpicture}[scale=0.12,baseline=0.2mm]
		\node(S) at (0,0) {$\cdot$};
		\node(W) at (-1,1) {$\cdot$};
		\node(E) at (1,1) {$\cdot$};
		\node(N) at (0,2) {$\cdot$};
		\draw (0,0) -- (-1,1) ;
		\draw (0,0) -- (1,1) ;
		\draw (0,0) -- (0,2) ;
		\draw (-1,1) -- (0,2) ;
	\end{tikzpicture}
}
\newcommand{\indpqfq}{
	\begin{tikzpicture}[scale=0.12,baseline=0.2mm]
		\node(S) at (0,0) {$\cdot$};
		\node(W) at (-1,1) {$\cdot$};
		\node(E) at (1,1) {$\cdot$};
		\node(N) at (0,2) {$\cdot$};
		\draw (0,0) -- (-1,1) ;
		\draw (0,0) -- (1,1) ;
		\draw (0,0) -- (0,2) ;
		\draw (1,1) -- (0,2) ;
	\end{tikzpicture}
}

\newcommand{\kindt}{
	\begin{tikzpicture}[scale=0.12,baseline=0.2mm]
		\node(S) at (0,0) {$\cdot$};
		\node(W) at (-1.5,1) {$\cdot$};
		\node(C) at (0,1) {$\cdot$};
		\node(E) at (1.5,1) {$\cdot$};
		\node(N) at (0,2) {$\cdot$};
	\end{tikzpicture}
}

\newcommand{\kinds}{
	\begin{tikzpicture}[scale=0.12,baseline=0.2mm]
		\node(S) at (0,0) {$\cdot$};
		\node(W) at (-1.5,1) {$\cdot$};
		\node(C) at (0,1) {$\cdot$};
		\node(E) at (1.5,1) {$\cdot$};
		\node(N) at (0,2) {$\cdot$};
		
		\draw (0,0) -- (-1.5,1) ;
		\draw (0,0) -- (0,1) ;
		\draw (0,0) -- (1.5,1) ;
		\draw (-1.5,1) -- (0,2) ;
		\draw (0,1) -- (0,2) ;
		\draw (1.5,1) -- (0,2) ;
		\draw plot [smooth, tension=1.5] coordinates {(0,0) (-0.5,1) (0,2)};
	\end{tikzpicture}
}

\newcommand{\kinda}{
	\begin{tikzpicture}[scale=0.12,baseline=0.2mm]
		\node(S) at (0,0) {$\cdot$};
		\node(W) at (-1.5,1) {$\cdot$};
		\node(C) at (0,1) {$\cdot$};
		\node(E) at (1.5,1) {$\cdot$};
		\node(N) at (0,2) {$\cdot$};
		
		\draw (0,0) -- (-1.5,1) ;
	\end{tikzpicture}
}

\newcommand{\kindb}{
	\begin{tikzpicture}[scale=0.12,baseline=0.2mm]
		\node(S) at (0,0) {$\cdot$};
		\node(W) at (-1.5,1) {$\cdot$};
		\node(C) at (0,1) {$\cdot$};
		\node(E) at (1.5,1) {$\cdot$};
		\node(N) at (0,2) {$\cdot$};
		
		\draw (0,0) -- (0,1) ;
	\end{tikzpicture}
}

\newcommand{\kindc}{
	\begin{tikzpicture}[scale=0.12,baseline=0.2mm]
		\node(S) at (0,0) {$\cdot$};
		\node(W) at (-1.5,1) {$\cdot$};
		\node(C) at (0,1) {$\cdot$};
		\node(E) at (1.5,1) {$\cdot$};
		\node(N) at (0,2) {$\cdot$};
		
		\draw (0,0) -- (1.5,1) ;
	\end{tikzpicture}
}

\newcommand{\kindab}{
	\begin{tikzpicture}[scale=0.12,baseline=0.2mm]
		\node(S) at (0,0) {$\cdot$};
		\node(W) at (-1.5,1) {$\cdot$};
		\node(C) at (0,1) {$\cdot$};
		\node(E) at (1.5,1) {$\cdot$};
		\node(N) at (0,2) {$\cdot$};
		
		\draw (0,0) -- (-1.5,1) ;
		\draw (0,0) -- (0,1) ;
	\end{tikzpicture}
}

\newcommand{\kindac}{
	\begin{tikzpicture}[scale=0.12,baseline=0.2mm]
		\node(S) at (0,0) {$\cdot$};
		\node(W) at (-1.5,1) {$\cdot$};
		\node(C) at (0,1) {$\cdot$};
		\node(E) at (1.5,1) {$\cdot$};
		\node(N) at (0,2) {$\cdot$};
		
		\draw (0,0) -- (-1.5,1) ;
		\draw (0,0) -- (1.5,1) ;
	\end{tikzpicture}
}

\newcommand{\kindbc}{
	\begin{tikzpicture}[scale=0.12,baseline=0.2mm]
		\node(S) at (0,0) {$\cdot$};
		\node(W) at (-1.5,1) {$\cdot$};
		\node(C) at (0,1) {$\cdot$};
		\node(E) at (1.5,1) {$\cdot$};
		\node(N) at (0,2) {$\cdot$};
		
		\draw (0,0) -- (0,1) ;
		\draw (0,0) -- (1.5,1) ;
	\end{tikzpicture}
}

\newcommand{\kindabc}{
	\begin{tikzpicture}[scale=0.12,baseline=0.2mm]
		\node(S) at (0,0) {$\cdot$};
		\node(W) at (-1.5,1) {$\cdot$};
		\node(C) at (0,1) {$\cdot$};
		\node(E) at (1.5,1) {$\cdot$};
		\node(N) at (0,2) {$\cdot$};
		
		\draw (0,0) -- (-1.5,1) ;
		\draw (0,0) -- (0,1) ;
		\draw (0,0) -- (1.5,1) ;
	\end{tikzpicture}
}

\newcommand{\kindka}{
	\begin{tikzpicture}[scale=0.12,baseline=0.2mm]
		\node(S) at (0,0) {$\cdot$};
		\node(W) at (-1.5,1) {$\cdot$};
		\node(C) at (0,1) {$\cdot$};
		\node(E) at (1.5,1) {$\cdot$};
		\node(N) at (0,2) {$\cdot$};
		
		\draw (0,0) -- (0,1) ;
		\draw (0,0) -- (1.5,1) ;
		\draw (-1.5,1) -- (0,2) ;
	\end{tikzpicture}
}

\newcommand{\kindkb}{
	\begin{tikzpicture}[scale=0.12,baseline=0.2mm]
		\node(S) at (0,0) {$\cdot$};
		\node(W) at (-1.5,1) {$\cdot$};
		\node(C) at (0,1) {$\cdot$};
		\node(E) at (1.5,1) {$\cdot$};
		\node(N) at (0,2) {$\cdot$};
		
		\draw (0,0) -- (-1.5,1) ;
		\draw (0,0) -- (1.5,1) ;
		\draw (0,1) -- (0,2) ;
	\end{tikzpicture}
}

\newcommand{\kindkc}{
	\begin{tikzpicture}[scale=0.12,baseline=0.2mm]
		\node(S) at (0,0) {$\cdot$};
		\node(W) at (-1.5,1) {$\cdot$};
		\node(C) at (0,1) {$\cdot$};
		\node(E) at (1.5,1) {$\cdot$};
		\node(N) at (0,2) {$\cdot$};
		
		\draw (0,0) -- (-1.5,1) ;
		\draw (0,0) -- (0,1) ;
		\draw (1.5,1) -- (0,2) ;
	\end{tikzpicture}
}

\newcommand{\kindkt}{
	\begin{tikzpicture}[scale=0.12,baseline=0.2mm]
		\node(S) at (0,0) {$\cdot$};
		\node(W) at (-1.5,1) {$\cdot$};
		\node(C) at (0,1) {$\cdot$};
		\node(E) at (1.5,1) {$\cdot$};
		\node(N) at (0,2) {$\cdot$};
		
		\draw (0,0) -- (-1.5,1) ;
		\draw (0,0) -- (0,1) ;
		\draw (0,0) -- (1.5,1) ;
		\draw plot [smooth, tension=1.5] coordinates {(0,0) (-0.5,1) (0,2)};
	\end{tikzpicture}
}

\newcommand{\kindkat}{
	\begin{tikzpicture}[scale=0.12,baseline=0.2mm]
		\node(S) at (0,0) {$\cdot$};
		\node(W) at (-1.5,1) {$\cdot$};
		\node(C) at (0,1) {$\cdot$};
		\node(E) at (1.5,1) {$\cdot$};
		\node(N) at (0,2) {$\cdot$};
		
		\draw (0,0) -- (-1.5,1) ;
		\draw (0,0) -- (0,1) ;
		\draw (0,0) -- (1.5,1) ;
		\draw (-1.5,1) -- (0,2) ;
		\draw plot [smooth, tension=1.5] coordinates {(0,0) (-0.5,1) (0,2)};	
	\end{tikzpicture}
}

\newcommand{\kindkbt}{
	\begin{tikzpicture}[scale=0.12,baseline=0.2mm]
		\node(S) at (0,0) {$\cdot$};
		\node(W) at (-1.5,1) {$\cdot$};
		\node(C) at (0,1) {$\cdot$};
		\node(E) at (1.5,1) {$\cdot$};
		\node(N) at (0,2) {$\cdot$};
		
		\draw (0,0) -- (-1.5,1) ;
		\draw (0,0) -- (0,1) ;
		\draw (0,0) -- (1.5,1) ;
		\draw (0,1) -- (0,2) ;
		\draw plot [smooth, tension=1.5] coordinates {(0,0) (-0.5,1) (0,2)};
	\end{tikzpicture}
}

\newcommand{\kindkct}{
	\begin{tikzpicture}[scale=0.12,baseline=0.2mm]
		\node(S) at (0,0) {$\cdot$};
		\node(W) at (-1.5,1) {$\cdot$};
		\node(C) at (0,1) {$\cdot$};
		\node(E) at (1.5,1) {$\cdot$};
		\node(N) at (0,2) {$\cdot$};
		
		\draw (0,0) -- (-1.5,1) ;
		\draw (0,0) -- (0,1) ;
		\draw (0,0) -- (1.5,1) ;
		\draw (1.5,1) -- (0,2) ;
		\draw plot [smooth, tension=1.5] coordinates {(0,0) (-0.5,1) (0,2)};
	\end{tikzpicture}
}

\newcommand{\kindkab}{
	\begin{tikzpicture}[scale=0.12,baseline=0.2mm]
		\node(S) at (0,0) {$\cdot$};
		\node(W) at (-1.5,1) {$\cdot$};
		\node(C) at (0,1) {$\cdot$};
		\node(E) at (1.5,1) {$\cdot$};
		\node(N) at (0,2) {$\cdot$};
		
		\draw (0,0) -- (-1.5,1) ;
		\draw (0,0) -- (0,1) ;
		\draw (0,0) -- (1.5,1) ;
		\draw (-1.5,1) -- (0,2) ;
		\draw (0,1) -- (0,2) ;
		\draw plot [smooth, tension=1.5] coordinates {(0,0) (-0.5,1) (0,2)};
	\end{tikzpicture}
}

\newcommand{\kindkac}{
	\begin{tikzpicture}[scale=0.12,baseline=0.2mm]
		\node(S) at (0,0) {$\cdot$};
		\node(W) at (-1.5,1) {$\cdot$};
		\node(C) at (0,1) {$\cdot$};
		\node(E) at (1.5,1) {$\cdot$};
		\node(N) at (0,2) {$\cdot$};
		
		\draw (0,0) -- (-1.5,1) ;
		\draw (0,0) -- (0,1) ;
		\draw (0,0) -- (1.5,1) ;
		\draw (-1.5,1) -- (0,2) ;
		\draw (1.5,1) -- (0,2) ;
		\draw plot [smooth, tension=1.5] coordinates {(0,0) (-0.5,1) (0,2)};
	\end{tikzpicture}
}

\newcommand{\kindkbc}{
	\begin{tikzpicture}[scale=0.12,baseline=0.2mm]
		\node(S) at (0,0) {$\cdot$};
		\node(W) at (-1.5,1) {$\cdot$};
		\node(C) at (0,1) {$\cdot$};
		\node(E) at (1.5,1) {$\cdot$};
		\node(N) at (0,2) {$\cdot$};
		
		\draw (0,0) -- (-1.5,1) ;
		\draw (0,0) -- (0,1) ;
		\draw (0,0) -- (1.5,1) ;
		\draw (0,1) -- (0,2) ;
		\draw (1.5,1) -- (0,2) ;
		\draw plot [smooth, tension=1.5] coordinates {(0,0) (-0.5,1) (0,2)};
	\end{tikzpicture}
}

\newcommand{\qeia}{
	\begin{tikzpicture}[scale=0.13,baseline=0.2mm]
		\node(r) at (0,0) {$\cdot$};
		\node(s) at (0,1) {$\cdot$};
		\node(c) at (0,2) {$\cdot$};
		\node(n) at (0,3) {$\cdot$};
		\node(w) at (-1.5,2) {$\cdot$};
		\node(e) at (1.5,2) {$\cdot$};
		
	\end{tikzpicture}
}

\newcommand{\qeib}{
	\begin{tikzpicture}[scale=0.13,baseline=0.2mm]
		\node(r) at (0,0) {$\cdot$};
		\node(s) at (0,1) {$\cdot$};
		\node(c) at (0,2) {$\cdot$};
		\node(n) at (0,3) {$\cdot$};
		\node(w) at (-1.5,2) {$\cdot$};
		\node(e) at (1.5,2) {$\cdot$};
		
		\draw (0,0) -- (0,1) ;
	\end{tikzpicture}
}
\newcommand{\qeic}{
	\begin{tikzpicture}[scale=0.13,baseline=0.2mm]
		\node(r) at (0,0) {$\cdot$};
		\node(s) at (0,1) {$\cdot$};
		\node(c) at (0,2) {$\cdot$};
		\node(n) at (0,3) {$\cdot$};
		\node(w) at (-1.5,2) {$\cdot$};
		\node(e) at (1.5,2) {$\cdot$};
		
		\draw (0,0) -- (0,1) ;
		\draw plot [smooth, tension=1.5] coordinates {(0,0) (-1.25,0.8) (-1.5,2)};
	\end{tikzpicture}
}
\newcommand{\qeid}{
	\begin{tikzpicture}[scale=0.13,baseline=0.2mm]
		\node(r) at (0,0) {$\cdot$};
		\node(s) at (0,1) {$\cdot$};
		\node(c) at (0,2) {$\cdot$};
		\node(n) at (0,3) {$\cdot$};
		\node(w) at (-1.5,2) {$\cdot$};
		\node(e) at (1.5,2) {$\cdot$};
		
		\draw (0,0) -- (0,1) ;
		
		\draw plot [smooth, tension=1.5] coordinates {(0,0) (-1.25,0.8) (-1.5,2)};
		\draw plot [smooth, tension=1.5] coordinates {(0,0) (0.75,1) (0,2)};
		
	\end{tikzpicture}
}
\newcommand{\qeie}{
	\begin{tikzpicture}[scale=0.13,baseline=0.2mm]
		\node(r) at (0,0) {$\cdot$};
		\node(s) at (0,1) {$\cdot$};
		\node(c) at (0,2) {$\cdot$};
		\node(n) at (0,3) {$\cdot$};
		\node(w) at (-1.5,2) {$\cdot$};
		\node(e) at (1.5,2) {$\cdot$};
		
		\draw (0,0) -- (0,1) ;
		
		\draw plot [smooth, tension=1.5] coordinates {(0,0) (-1.25,0.8) (-1.5,2)};
		\draw plot [smooth, tension=1.5] coordinates {(0,0) (0.75,1) (0,2)};
		\draw plot [smooth, tension=1.5] coordinates {(0,0) (1.25,0.8) (1.5,2)};
		
	\end{tikzpicture}
}
\newcommand{\qeif}{
	\begin{tikzpicture}[scale=0.13,baseline=0.2mm]
		\node(r) at (0,0) {$\cdot$};
		\node(s) at (0,1) {$\cdot$};
		\node(c) at (0,2) {$\cdot$};
		\node(n) at (0,3) {$\cdot$};
		\node(w) at (-1.5,2) {$\cdot$};
		\node(e) at (1.5,2) {$\cdot$};
		
		\draw (0,0) -- (0,1) ;
		
		\draw plot [smooth, tension=1.5] coordinates {(0,0) (-1.25,0.8) (-1.5,2)};
		\draw plot [smooth, tension=1.5] coordinates {(0,0) (0.75,1) (0,2)};
		\draw plot [smooth, tension=1.5] coordinates {(0,0) (1.25,0.8) (1.5,2)};
		
		\draw plot [smooth, tension=1.5] coordinates {(0,0) (2.5,1.75) (0,3)};
	\end{tikzpicture}
}
\newcommand{\qeig}{
	\begin{tikzpicture}[scale=0.13,baseline=0.2mm]
		\node(r) at (0,0) {$\cdot$};
		\node(s) at (0,1) {$\cdot$};
		\node(c) at (0,2) {$\cdot$};
		\node(n) at (0,3) {$\cdot$};
		\node(w) at (-1.5,2) {$\cdot$};
		\node(e) at (1.5,2) {$\cdot$};
		
		\draw (0,0) -- (0,1) ;
		\draw (0,1) -- (-1.5,2) ;
		
		\draw plot [smooth, tension=1.5] coordinates {(0,0) (-1.25,0.8) (-1.5,2)};
	\end{tikzpicture}
}
\newcommand{\qeih}{
	\begin{tikzpicture}[scale=0.13,baseline=0.2mm]
		\node(r) at (0,0) {$\cdot$};
		\node(s) at (0,1) {$\cdot$};
		\node(c) at (0,2) {$\cdot$};
		\node(n) at (0,3) {$\cdot$};
		\node(w) at (-1.5,2) {$\cdot$};
		\node(e) at (1.5,2) {$\cdot$};
		
		\draw (0,0) -- (0,1) ;
		\draw (0,1) -- (-1.5,2) ;
		
		\draw plot [smooth, tension=1.5] coordinates {(0,0) (-1.25,0.8) (-1.5,2)};
		\draw plot [smooth, tension=1.5] coordinates {(0,0) (0.75,1) (0,2)};
	\end{tikzpicture}
}
\newcommand{\qeii}{
	\begin{tikzpicture}[scale=0.13,baseline=0.2mm]
		\node(r) at (0,0) {$\cdot$};
		\node(s) at (0,1) {$\cdot$};
		\node(c) at (0,2) {$\cdot$};
		\node(n) at (0,3) {$\cdot$};
		\node(w) at (-1.5,2) {$\cdot$};
		\node(e) at (1.5,2) {$\cdot$};
		
		\draw (0,0) -- (0,1) ;
		\draw (0,1) -- (-1.5,2) ;
		
		\draw plot [smooth, tension=1.5] coordinates {(0,0) (-1.25,0.8) (-1.5,2)};
		\draw plot [smooth, tension=1.5] coordinates {(0,0) (0.75,1) (0,2)};
		\draw plot [smooth, tension=1.5] coordinates {(0,0) (1.25,0.8) (1.5,2)};
		
	\end{tikzpicture}
}
\newcommand{\qeij}{
	\begin{tikzpicture}[scale=0.13,baseline=0.2mm]
		\node(r) at (0,0) {$\cdot$};
		\node(s) at (0,1) {$\cdot$};
		\node(c) at (0,2) {$\cdot$};
		\node(n) at (0,3) {$\cdot$};
		\node(w) at (-1.5,2) {$\cdot$};
		\node(e) at (1.5,2) {$\cdot$};
		
		\draw (0,0) -- (0,1) ;
		\draw (0,1) -- (-1.5,2) ;
		
		\draw plot [smooth, tension=1.5] coordinates {(0,0) (-1.25,0.8) (-1.5,2)};
		\draw plot [smooth, tension=1.5] coordinates {(0,0) (0.75,1) (0,2)};
		\draw plot [smooth, tension=1.5] coordinates {(0,0) (1.25,0.8) (1.5,2)};
		
		\draw plot [smooth, tension=1.5] coordinates {(0,0) (2.5,1.75) (0,3)};
	\end{tikzpicture}
}
\newcommand{\qeik}{
	\begin{tikzpicture}[scale=0.13,baseline=0.2mm]
		\node(r) at (0,0) {$\cdot$};
		\node(s) at (0,1) {$\cdot$};
		\node(c) at (0,2) {$\cdot$};
		\node(n) at (0,3) {$\cdot$};
		\node(w) at (-1.5,2) {$\cdot$};
		\node(e) at (1.5,2) {$\cdot$};
		
		\draw (0,0) -- (0,1) ;
		\draw (0,1) -- (-1.5,2) ;
		\draw (0,1) -- (0,2) ;
		
		\draw plot [smooth, tension=1.5] coordinates {(0,0) (-1.25,0.8) (-1.5,2)};
		\draw plot [smooth, tension=1.5] coordinates {(0,0) (0.75,1) (0,2)};
	\end{tikzpicture}
}
\newcommand{\qeil}{
	\begin{tikzpicture}[scale=0.13,baseline=0.2mm]
		\node(r) at (0,0) {$\cdot$};
		\node(s) at (0,1) {$\cdot$};
		\node(c) at (0,2) {$\cdot$};
		\node(n) at (0,3) {$\cdot$};
		\node(w) at (-1.5,2) {$\cdot$};
		\node(e) at (1.5,2) {$\cdot$};
		
		\draw (0,0) -- (0,1) ;
		\draw (0,1) -- (-1.5,2) ;
		\draw (0,1) -- (0,2) ;
		
		\draw plot [smooth, tension=1.5] coordinates {(0,0) (-1.25,0.8) (-1.5,2)};
		\draw plot [smooth, tension=1.5] coordinates {(0,0) (0.75,1) (0,2)};
		\draw plot [smooth, tension=1.5] coordinates {(0,0) (1.25,0.8) (1.5,2)};
		
	\end{tikzpicture}
}
\newcommand{\qeim}{
	\begin{tikzpicture}[scale=0.13,baseline=0.2mm]
		\node(r) at (0,0) {$\cdot$};
		\node(s) at (0,1) {$\cdot$};
		\node(c) at (0,2) {$\cdot$};
		\node(n) at (0,3) {$\cdot$};
		\node(w) at (-1.5,2) {$\cdot$};
		\node(e) at (1.5,2) {$\cdot$};
		
		\draw (0,0) -- (0,1) ;
		\draw (0,1) -- (-1.5,2) ;
		\draw (0,1) -- (0,2) ;
		
		\draw plot [smooth, tension=1.5] coordinates {(0,0) (-1.25,0.8) (-1.5,2)};
		\draw plot [smooth, tension=1.5] coordinates {(0,0) (0.75,1) (0,2)};
		\draw plot [smooth, tension=1.5] coordinates {(0,0) (1.25,0.8) (1.5,2)};
		
		\draw plot [smooth, tension=1.5] coordinates {(0,0) (2.5,1.75) (0,3)};
	\end{tikzpicture}
}
\newcommand{\qein}{
	\begin{tikzpicture}[scale=0.13,baseline=0.2mm]
		\node(r) at (0,0) {$\cdot$};
		\node(s) at (0,1) {$\cdot$};
		\node(c) at (0,2) {$\cdot$};
		\node(n) at (0,3) {$\cdot$};
		\node(w) at (-1.5,2) {$\cdot$};
		\node(e) at (1.5,2) {$\cdot$};
		
		\draw (0,0) -- (0,1) ;
		\draw (0,1) -- (-1.5,2) ;
		\draw (0,1) -- (0,2) ;
		\draw (1.5,2) -- (0,3) ;
		
		\draw plot [smooth, tension=1.5] coordinates {(0,0) (-1.25,0.8) (-1.5,2)};
		\draw plot [smooth, tension=1.5] coordinates {(0,0) (0.75,1) (0,2)};
		\draw plot [smooth, tension=1.5] coordinates {(0,0) (1.25,0.8) (1.5,2)};
		
		\draw plot [smooth, tension=1.5] coordinates {(0,0) (2.5,1.75) (0,3)};
	\end{tikzpicture}
}
\newcommand{\qeio}{
	\begin{tikzpicture}[scale=0.13,baseline=0.2mm]
		\node(r) at (0,0) {$\cdot$};
		\node(s) at (0,1) {$\cdot$};
		\node(c) at (0,2) {$\cdot$};
		\node(n) at (0,3) {$\cdot$};
		\node(w) at (-1.5,2) {$\cdot$};
		\node(e) at (1.5,2) {$\cdot$};
		
		\draw (0,0) -- (0,1) ;
		\draw (0,1) -- (-1.5,2) ;
		\draw (0,1) -- (0,2) ;
		\draw (0,1) -- (1.5,2) ;
		\draw (1.5,2) -- (0,3) ;
		
		\draw plot [smooth, tension=1.5] coordinates {(0,0) (-1.25,0.8) (-1.5,2)};
		\draw plot [smooth, tension=1.5] coordinates {(0,0) (0.75,1) (0,2)};
		\draw plot [smooth, tension=1.5] coordinates {(0,0) (1.25,0.8) (1.5,2)};
		\draw plot [smooth, tension=1.5] coordinates {(0,1) (-0.75,2) (0,3)};
		
		\draw plot [smooth, tension=1.5] coordinates {(0,0) (2.5,1.75) (0,3)};
	\end{tikzpicture}
}
\newcommand{\qeip}{
	\begin{tikzpicture}[scale=0.13,baseline=0.2mm]
		\node(r) at (0,0) {$\cdot$};
		\node(s) at (0,1) {$\cdot$};
		\node(c) at (0,2) {$\cdot$};
		\node(n) at (0,3) {$\cdot$};
		\node(w) at (-1.5,2) {$\cdot$};
		\node(e) at (1.5,2) {$\cdot$};
		
		\draw (0,0) -- (0,1) ;
		\draw (0,1) -- (-1.5,2) ;
		\draw (0,1) -- (0,2) ;
		\draw (0,1) -- (1.5,2) ;
		\draw (0,2) -- (0,3) ;
		\draw (1.5,2) -- (0,3) ;
		
		\draw plot [smooth, tension=1.5] coordinates {(0,0) (-1.25,0.8) (-1.5,2)};
		\draw plot [smooth, tension=1.5] coordinates {(0,0) (0.75,1) (0,2)};
		\draw plot [smooth, tension=1.5] coordinates {(0,0) (1.25,0.8) (1.5,2)};
		\draw plot [smooth, tension=1.5] coordinates {(0,1) (-0.75,2) (0,3)};
		
		\draw plot [smooth, tension=1.5] coordinates {(0,0) (2.5,1.75) (0,3)};
	\end{tikzpicture}
}
\newcommand{\qeiq}{
	\begin{tikzpicture}[scale=0.13,baseline=0.2mm]
		\node(r) at (0,0) {$\cdot$};
		\node(s) at (0,1) {$\cdot$};
		\node(c) at (0,2) {$\cdot$};
		\node(n) at (0,3) {$\cdot$};
		\node(w) at (-1.5,2) {$\cdot$};
		\node(e) at (1.5,2) {$\cdot$};
		
		\draw (0,0) -- (0,1) ;
		\draw (0,1) -- (-1.5,2) ;
		\draw (0,1) -- (0,2) ;
		\draw (0,1) -- (1.5,2) ;
		\draw (-1.5,2) -- (0,3) ;
		\draw (0,2) -- (0,3) ;
		\draw (1.5,2) -- (0,3) ;
		
		\draw plot [smooth, tension=1.5] coordinates {(0,0) (-1.25,0.8) (-1.5,2)};
		\draw plot [smooth, tension=1.5] coordinates {(0,0) (0.75,1) (0,2)};
		\draw plot [smooth, tension=1.5] coordinates {(0,0) (1.25,0.8) (1.5,2)};
		\draw plot [smooth, tension=1.5] coordinates {(0,1) (-0.75,2) (0,3)};
		
		\draw plot [smooth, tension=1.5] coordinates {(0,0) (2.5,1.75) (0,3)};
	\end{tikzpicture}
}
\newcommand{\qeir}{
	\begin{tikzpicture}[scale=0.13,baseline=0.2mm]
		\node(r) at (0,0) {$\cdot$};
		\node(s) at (0,1) {$\cdot$};
		\node(c) at (0,2) {$\cdot$};
		\node(n) at (0,3) {$\cdot$};
		\node(w) at (-1.5,2) {$\cdot$};
		\node(e) at (1.5,2) {$\cdot$};
		
		\draw (0,0) -- (0,1) ;
		\draw (0,1) -- (-1.5,2) ;
		\draw (0,1) -- (0,2) ;
		\draw (0,1) -- (1.5,2) ;
		
		\draw plot [smooth, tension=1.5] coordinates {(0,0) (-1.25,0.8) (-1.5,2)};
		\draw plot [smooth, tension=1.5] coordinates {(0,0) (0.75,1) (0,2)};
		\draw plot [smooth, tension=1.5] coordinates {(0,0) (1.25,0.8) (1.5,2)};
	\end{tikzpicture}
}
\newcommand{\qeis}{
	\begin{tikzpicture}[scale=0.13,baseline=0.2mm]
		\node(r) at (0,0) {$\cdot$};
		\node(s) at (0,1) {$\cdot$};
		\node(c) at (0,2) {$\cdot$};
		\node(n) at (0,3) {$\cdot$};
		\node(w) at (-1.5,2) {$\cdot$};
		\node(e) at (1.5,2) {$\cdot$};
		
		\draw (0,0) -- (0,1) ;
		\draw (0,1) -- (-1.5,2) ;
		\draw (0,1) -- (0,2) ;
		\draw (0,1) -- (1.5,2) ;
		
		\draw plot [smooth, tension=1.5] coordinates {(0,0) (-1.25,0.8) (-1.5,2)};
		\draw plot [smooth, tension=1.5] coordinates {(0,0) (0.75,1) (0,2)};
		\draw plot [smooth, tension=1.5] coordinates {(0,0) (1.25,0.8) (1.5,2)};
		
		\draw plot [smooth, tension=1.5] coordinates {(0,0) (2.5,1.75) (0,3)};
	\end{tikzpicture}
}
\newcommand{\qeit}{
	\begin{tikzpicture}[scale=0.13,baseline=0.2mm]
		\node(r) at (0,0) {$\cdot$};
		\node(s) at (0,1) {$\cdot$};
		\node(c) at (0,2) {$\cdot$};
		\node(n) at (0,3) {$\cdot$};
		\node(w) at (-1.5,2) {$\cdot$};
		\node(e) at (1.5,2) {$\cdot$};
		
		\draw (0,0) -- (0,1) ;
		\draw (0,1) -- (-1.5,2) ;
		\draw (0,1) -- (0,2) ;
		\draw (0,1) -- (1.5,2) ;
		
		\draw plot [smooth, tension=1.5] coordinates {(0,0) (-1.25,0.8) (-1.5,2)};
		\draw plot [smooth, tension=1.5] coordinates {(0,0) (0.75,1) (0,2)};
		\draw plot [smooth, tension=1.5] coordinates {(0,0) (1.25,0.8) (1.5,2)};
		\draw plot [smooth, tension=1.5] coordinates {(0,1) (-0.75,2) (0,3)};
		
		\draw plot [smooth, tension=1.5] coordinates {(0,0) (2.5,1.75) (0,3)};
	\end{tikzpicture}
}
\newcommand{\qeiu}{
	\begin{tikzpicture}[scale=0.13,baseline=0.2mm]
		\node(r) at (0,0) {$\cdot$};
		\node(s) at (0,1) {$\cdot$};
		\node(c) at (0,2) {$\cdot$};
		\node(n) at (0,3) {$\cdot$};
		\node(w) at (-1.5,2) {$\cdot$};
		\node(e) at (1.5,2) {$\cdot$};
		
		\draw (0,0) -- (0,1) ;
		\draw (0,1) -- (-1.5,2) ;
		\draw (0,1) -- (0,2) ;
		\draw (1.5,2) -- (0,3) ;
		
		\draw plot [smooth, tension=1.5] coordinates {(0,0) (-1.25,0.8) (-1.5,2)};
		\draw plot [smooth, tension=1.5] coordinates {(0,0) (0.75,1) (0,2)};
	\end{tikzpicture}
}
\newcommand{\qeiv}{
	\begin{tikzpicture}[scale=0.13,baseline=0.2mm]
		\node(r) at (0,0) {$\cdot$};
		\node(s) at (0,1) {$\cdot$};
		\node(c) at (0,2) {$\cdot$};
		\node(n) at (0,3) {$\cdot$};
		\node(w) at (-1.5,2) {$\cdot$};
		\node(e) at (1.5,2) {$\cdot$};
		
		\draw (0,1) -- (-1.5,2) ;
		\draw (0,1) -- (0,2) ;
		\draw (0,1) -- (1.5,2) ;
	\end{tikzpicture}
}
\newcommand{\qeiw}{
	\begin{tikzpicture}[scale=0.13,baseline=0.2mm]
		\node(r) at (0,0) {$\cdot$};
		\node(s) at (0,1) {$\cdot$};
		\node(c) at (0,2) {$\cdot$};
		\node(n) at (0,3) {$\cdot$};
		\node(w) at (-1.5,2) {$\cdot$};
		\node(e) at (1.5,2) {$\cdot$};
		
		\draw (0,1) -- (-1.5,2) ;
		\draw (0,1) -- (0,2) ;
		\draw (0,1) -- (1.5,2) ;
		\draw plot [smooth, tension=1.5] coordinates {(0,1) (-0.75,2) (0,3)};
		
	\end{tikzpicture}
}
\newcommand{\qeix}{
	\begin{tikzpicture}[scale=0.13,baseline=0.2mm]
		\node(r) at (0,0) {$\cdot$};
		\node(s) at (0,1) {$\cdot$};
		\node(c) at (0,2) {$\cdot$};
		\node(n) at (0,3) {$\cdot$};
		\node(w) at (-1.5,2) {$\cdot$};
		\node(e) at (1.5,2) {$\cdot$};
		
		\draw (0,1) -- (-1.5,2) ;
	\end{tikzpicture}
}
\newcommand{\qeiy}{
	\begin{tikzpicture}[scale=0.13,baseline=0.2mm]
		\node(r) at (0,0) {$\cdot$};
		\node(s) at (0,1) {$\cdot$};
		\node(c) at (0,2) {$\cdot$};
		\node(n) at (0,3) {$\cdot$};
		\node(w) at (-1.5,2) {$\cdot$};
		\node(e) at (1.5,2) {$\cdot$};
		
		\draw (0,1) -- (-1.5,2) ;
		\draw (0,1) -- (0,2) ;
	\end{tikzpicture}
}
\newcommand{\qeiz}{
	\begin{tikzpicture}[scale=0.13,baseline=0.2mm]
		\node(r) at (0,0) {$\cdot$};
		\node(s) at (0,1) {$\cdot$};
		\node(c) at (0,2) {$\cdot$};
		\node(n) at (0,3) {$\cdot$};
		\node(w) at (-1.5,2) {$\cdot$};
		\node(e) at (1.5,2) {$\cdot$};
		
		\draw (0,1) -- (-1.5,2) ;
		\draw (0,1) -- (0,2) ;
		\draw (1.5,2) -- (0,3) ;
		
	\end{tikzpicture}
}
\newcommand{\qeiza}{
	\begin{tikzpicture}[scale=0.13,baseline=0.2mm]
		\node(r) at (0,0) {$\cdot$};
		\node(s) at (0,1) {$\cdot$};
		\node(c) at (0,2) {$\cdot$};
		\node(n) at (0,3) {$\cdot$};
		\node(w) at (-1.5,2) {$\cdot$};
		\node(e) at (1.5,2) {$\cdot$};
		
		\draw (0,1) -- (-1.5,2) ;
		\draw (0,1) -- (0,2) ;
		\draw (0,1) -- (1.5,2) ;
		\draw (1.5,2) -- (0,3) ;
		
		\draw plot [smooth, tension=1.5] coordinates {(0,1) (-0.75,2) (0,3)};
		
	\end{tikzpicture}
}
\newcommand{\qeizb}{
	\begin{tikzpicture}[scale=0.13,baseline=0.2mm]
		\node(r) at (0,0) {$\cdot$};
		\node(s) at (0,1) {$\cdot$};
		\node(c) at (0,2) {$\cdot$};
		\node(n) at (0,3) {$\cdot$};
		\node(w) at (-1.5,2) {$\cdot$};
		\node(e) at (1.5,2) {$\cdot$};
		
		\draw (0,1) -- (-1.5,2) ;
		\draw (0,1) -- (0,2) ;
		\draw (0,1) -- (1.5,2) ;
		\draw (0,2) -- (0,3) ;
		\draw (1.5,2) -- (0,3) ;
		
		\draw plot [smooth, tension=1.5] coordinates {(0,1) (-0.75,2) (0,3)};
		
	\end{tikzpicture}
}

\newcommand{\qeizc}{
	\begin{tikzpicture}[scale=0.13,baseline=0.2mm]
		\node(r) at (0,0) {$\cdot$};
		\node(s) at (0,1) {$\cdot$};
		\node(c) at (0,2) {$\cdot$};
		\node(n) at (0,3) {$\cdot$};
		\node(w) at (-1.5,2) {$\cdot$};
		\node(e) at (1.5,2) {$\cdot$};
		
		\draw (0,1) -- (-1.5,2) ;
		\draw (0,1) -- (0,2) ;
		\draw (0,1) -- (1.5,2) ;
		\draw (-1.5,2) -- (0,3) ;
		\draw (0,2) -- (0,3) ;
		\draw (1.5,2) -- (0,3) ;
		
		\draw plot [smooth, tension=1.5] coordinates {(0,1) (-0.75,2) (0,3)};
		
	\end{tikzpicture}
}

\newcommand{\dsia}{
	\begin{tikzpicture}[scale=0.2,baseline=0.2mm]
		\node(S) at (0,0) {$\cdot$};
		\node(W) at (-1.5,1) {$\cdot$};
		\node(C) at (0,1) {$\cdot$};
		\node(E) at (1.5,1) {$\cdot$};
		\node(N) at (0,2) {$\cdot$};
		\node(3) at (0.6,1.2) {$\cdot$};
		
	\end{tikzpicture}
}

\newcommand{\dsib}{
	\begin{tikzpicture}[scale=0.2,baseline=0.2mm]
		\node(S) at (0,0) {$\cdot$};
		\node(W) at (-1.5,1) {$\cdot$};
		\node(C) at (0,1) {$\cdot$};
		\node(E) at (1.5,1) {$\cdot$};
		\node(N) at (0,2) {$\cdot$};
		\node(3) at (0.6,1.2) {$\cdot$};
		
		\draw (0,0) -- (-1.5,1) ;
		\draw (0,0) -- (0,1) ;
		\draw (0,0) -- (1.5,1) ;
		
	\end{tikzpicture}
}

\newcommand{\dsic}{
	\begin{tikzpicture}[scale=0.2,baseline=0.2mm]
		\node(S) at (0,0) {$\cdot$};
		\node(W) at (-1.5,1) {$\cdot$};
		\node(C) at (0,1) {$\cdot$};
		\node(E) at (1.5,1) {$\cdot$};
		\node(N) at (0,2) {$\cdot$};
		\node(3) at (0.6,1.2) {$\cdot$};
		
		\draw (0,0) -- (0.6,1.2) ;
		
	\end{tikzpicture}
}

\newcommand{\dsid}{
	\begin{tikzpicture}[scale=0.2,baseline=0.2mm]
		\node(S) at (0,0) {$\cdot$};
		\node(W) at (-1.5,1) {$\cdot$};
		\node(C) at (0,1) {$\cdot$};
		\node(E) at (1.5,1) {$\cdot$};
		\node(N) at (0,2) {$\cdot$};
		\node(3) at (0.6,1.2) {$\cdot$};
		
		\draw (0,0) -- (-1.5,1) ;
		\draw (0,0) -- (0,1) ;
		\draw (0,0) -- (1.5,1) ;
		
		\draw (0,0) -- (0.6,1.2) ;
		
	\end{tikzpicture}
}

\newcommand{\dsie}{
	\begin{tikzpicture}[scale=0.2,baseline=0.2mm]
		\node(S) at (0,0) {$\cdot$};
		\node(W) at (-1.5,1) {$\cdot$};
		\node(C) at (0,1) {$\cdot$};
		\node(E) at (1.5,1) {$\cdot$};
		\node(N) at (0,2) {$\cdot$};
		\node(3) at (0.6,1.2) {$\cdot$};
		
		\draw (0,0) -- (-1.5,1) ;
		\draw (0,0) -- (0,1) ;
		\draw (0,0) -- (1.5,1) ;
		
		\draw (0,0) -- (0.6,1.2) ;
		
		\draw plot [smooth, tension=1.5] coordinates {(0,0) (-0.5,1) (0,2)};
	\end{tikzpicture}
}

\newcommand{\dsif}{
	\begin{tikzpicture}[scale=0.2,baseline=0.2mm]
		\node(S) at (0,0) {$\cdot$};
		\node(W) at (-1.5,1) {$\cdot$};
		\node(C) at (0,1) {$\cdot$};
		\node(E) at (1.5,1) {$\cdot$};
		\node(N) at (0,2) {$\cdot$};
		\node(3) at (0.6,1.2) {$\cdot$};
		
		\draw (0,0) -- (-1.5,1) ;
		\draw (0,0) -- (0,1) ;
		\draw (0,0) -- (1.5,1) ;
		
		\draw (0,0) -- (0.6,1.2) ;
		
		\draw plot [smooth, tension=1.5] coordinates {(0,0) (-0.5,1) (0,2)};
		\draw (-1.5,1) -- (0,2) ;
		\draw (0,1) -- (0,2) ;
		\draw (1.5,1) -- (0,2) ;
	\end{tikzpicture}
}

\newcommand{\dsig}{
	\begin{tikzpicture}[scale=0.2,baseline=0.2mm]
		\node(S) at (0,0) {$\cdot$};
		\node(W) at (-1.5,1) {$\cdot$};
		\node(C) at (0,1) {$\cdot$};
		\node(E) at (1.5,1) {$\cdot$};
		\node(N) at (0,2) {$\cdot$};
		\node(3) at (0.6,1.2) {$\cdot$};
		
		\draw (0,0) -- (-1.5,1) ;
		\draw (0,0) -- (0,1) ;
		\draw (0,0) -- (1.5,1) ;
		
		\draw (0.6,1.2) -- (0,2) ;
	\end{tikzpicture}
}

\newcommand{\dsih}{
	\begin{tikzpicture}[scale=0.2,baseline=0.2mm]
		\node(S) at (0,0) {$\cdot$};
		\node(W) at (-1.5,1) {$\cdot$};
		\node(C) at (0,1) {$\cdot$};
		\node(E) at (1.5,1) {$\cdot$};
		\node(N) at (0,2) {$\cdot$};
		\node(3) at (0.6,1.2) {$\cdot$};
		
		\draw (0,0) -- (-1.5,1) ;
		\draw (0,0) -- (0,1) ;
		\draw (0,0) -- (1.5,1) ;
		
		\draw (0,0) -- (0.6,1.2) ;
		
		\draw plot [smooth, tension=1.5] coordinates {(0,0) (-0.5,1) (0,2)};
		\draw (0.6,1.2) -- (0,2) ;
	\end{tikzpicture}
}

\newcommand{\dsii}{
	\begin{tikzpicture}[scale=0.2,baseline=0.2mm]
		\node(S) at (0,0) {$\cdot$};
		\node(W) at (-1.5,1) {$\cdot$};
		\node(C) at (0,1) {$\cdot$};
		\node(E) at (1.5,1) {$\cdot$};
		\node(N) at (0,2) {$\cdot$};
		\node(3) at (0.6,1.2) {$\cdot$};
		
		\draw (0,0) -- (-1.5,1) ;
		\draw (0,0) -- (0,1) ;
		\draw (0,0) -- (1.5,1) ;
		
		\draw (0,0) -- (0.6,1.2) ;
		
		\draw plot [smooth, tension=1.5] coordinates {(0,0) (-0.5,1) (0,2)};
		\draw (-1.5,1) -- (0,2) ;
		\draw (0,1) -- (0,2) ;
		\draw (1.5,1) -- (0,2) ;
		\draw (0.6,1.2) -- (0,2) ;
	\end{tikzpicture}
}

\theoremstyle{plain}
\newtheorem{lem}{Lemma}[section]
\newtheorem{thm}[lem]{Theorem}
\newtheorem*{thm*}{Theorem}
\newtheorem{cor}[lem]{Corollary}
\newtheorem{prop}[lem]{Proposition}

\newtheorem*{thmDUab}{Theorem \protect{\ref{thm:discfinab}}}
\newtheorem*{thmLU}{Theorem \protect{\ref{thm:LUCpn}} and Theorem \protect{\ref{thm:Cpqlinisom}}}

\theoremstyle{remark}
\newtheorem{rem}[lem]{Remark}

\theoremstyle{definition}
\newtheorem{ex}[lem]{Example}

\newtheorem{defn}[lem]{Definition}

\newtheorem{nota}[lem]{Notation}
\newtheorem{constr}[lem]{Construction}

\newtheorem{prob}[lem]{Problem}
\newtheorem{conv}[lem]{Convention}

\begin{document}
\maketitle

\begin{abstract} We study the indexing systems that correspond to equivariant Steiner and linear isometries operads. When $G$ is a finite abelian group, we prove that a $G$-indexing system is realized by a Steiner operad if and only if it is generated by cyclic $G$-orbits. When $G$ is a finite cyclic group, whose order is either a prime power or a product of two distinct primes greater than $3$, we prove that a $G$-indexing system is realized by a linear isometries operad if and only if it satisfies Blumberg and Hill's horn-filling condition.

We also repackage the data in an indexing system as a certain kind of partial order. We call these posets transfer systems, and we develop basic tools for computing with them.
\end{abstract}

\tableofcontents

\section{Introduction}

A $N_\infty$ operad is a representing object for homotopy commutative monoids, equipped with additional equivariant transfer maps. Though there are many different point-set models of $N_\infty$ operads, the homotopy theory of $N_\infty$ operads is much simpler, because it is essentially algebraic. This is true in at least two senses. On the one hand, the classification theorem states that for any finite group $G$, the homotopy category $\t{Ho}(N_\infty\t{-}\b{Op})$ of $N_\infty$ $G$-operads is equivalent to the poset $\b{Ind}$ of all $G$-indexing systems (cf. \cite{BH}, \cite{BonPer}, \cite{GutWhite}, \cite{Rubin}). Here, the indexing system associated to a $N_\infty$ operad is an algebraic object that encodes the additional transfers. On the other hand, one can also model the entire homotopy theory of $N_\infty$ $G$-operads with discrete operads in the category of $G$-sets (cf. \cite[Theorem 3.6]{Rubin}). It follows that all homotopical constructions on $N_\infty$ operads can be performed in pure combinatorics, and then transported into topology after the fact.

That being said, this point of view completely ignores the naturally occurring geometry. One of the initial motivations behind the study of $N_\infty$ operads was a desire to understand additive and multiplicative structures on $G$-spectra over various universes $U$. Such operations are naturally parametrized by Steiner operads $\c{K}(U)$ and linear isometries operads $\c{L}(U)$, and since these operads are so canonical, Blumberg and Hill asked ``whether or not all homotopy types in $N_\infty\t{-}\b{Op}$ are realized by the operads that ``arise in nature'', i.e., the equivariant Steiner and linear isometries operads'' \cite[p. 22]{BH}. In this paper, we investigate this problem and some of its extensions.

The answer to Blumberg and Hill's question depends on the ambient group, but it is no in most cases. Given a prime $p$, the answer is yes for the cyclic groups $C_p$ and $C_{p^2}$, but it is no for $C_{p^n}$ when $n \geq 3$, and it is no for $(C_p)^{\times n}$ when $n \geq 2$. Given distinct primes $p$ and $q$, it is yes for $C_{pq}$ provided that $p, q > 3$, but no otherwise. In general, if $G$ contains a tower $1 \subsetneq L \subsetneq H \subsetneq G$, or if $G$ is a non-cyclic finite abelian group, then there is at least one $N_\infty$ $G$-operad that is inequivalent to every Steiner and linear isometries operad (Theorems \ref{thm:notSteinerorlinisom} and \ref{thm:notLorD}). There are often many more. For example, only 9 of the 19 homotopy types of $N_\infty$ $K_4$-operads correspond to Steiner and linear isometries operads. We fare no better in the non-abelian case. Only $5$ of the $9$ $\Sigma_3$-homotopy types and only $22$ of the $68$ $Q_8$-homotopy types correspond to such operads.

Thus, we refine Blumberg and Hill's question. For any given group $G$, we ask which properties characterize the Steiner and linear isometries $G$-operads among all $N_\infty$ $G$-operads. In light of the equivalence $\t{Ho}(N_\infty\t{-}\b{Op}) \simeq \b{Ind}$, we seek algebraic properties that detect when a given $G$-indexing system $\c{I} \in \b{Ind}$ corresponds to some $\c{K}(U)$ or $\c{L}(U)$. A first observation\footnote{Reproduced in this paper as Proposition \ref{prop:BHcond}.} is that every $\c{I}$ obtained from a linear isometries operad satisfies the condition below \cite[p. 17]{BH}.
	\begin{equation*}\tag{$\Lambda$}
	\t{If $K \subset L \subset H \subset G$ and $H/K \in \c{I}$, then $L/K \in \c{I}$ and $H/L \in \c{I}$.}
	\end{equation*}
We are after further conditions that encode the peculiarities of Steiner and linear isometries operads.

The homotopy types of Steiner and linear isometries $G$-operads are determined by the representation theory of $G$ over the reals, but the translation to the algebra of indexing systems is surprisingly bad. The collection $\b{Uni}$ of all isoclasses of $G$-universes forms a cube, the poset category $\b{Ind}$ is a lattice, and we obtain two natural functions $\b{Uni} \rightrightarrows \b{Ind}$ by sending a $G$-universe $U$ to the indexing systems corresponding to $\c{K}(U)$ and $\c{L}(U)$. \emph{Neither of these functions are lattice maps in general} (Proposition \ref{prop:notlattice}). Thus, we eschew a top down approach in favor of a more direct attack. We elaborate on Blumberg and Hill's computations, and then we look for patterns after the fact. We prove the following.

\begin{thmDUab} Let $G$ be a finite abelian group and let $\c{I}$ be a $G$-indexing system. Then $\c{I}$ corresponds to a $G$-Steiner operad if and only if $\c{I}$ is generated by a set of cyclic $G$-orbits.
\end{thmDUab}

The key point is that $\c{K}(U)$ may be analyzed one irreducible subrepresentation $V \subset U$ at a time. We have no such luck for $\c{L}(U)$, so we specialize $G$ further.

\begin{thmLU} Let $G$ be a finite cyclic group, whose order is either a prime power or a product of two distinct primes greater than $3$. If $\c{I}$ is a $G$-indexing system, then $\c{I}$ corresponds to a $G$-linear isometries operad if and only if $\c{I}$ satisfies condition $(\Lambda)$.
\end{thmLU}

Blumberg and Hill already noted the necessity of condition $(\Lambda)$, and we prove sufficiency by direct construction. Note that condition $(\Lambda)$ is not sufficient when $G = C_{2q}$ or $C_{3q}$, because there are not enough $G$-representations in these cases.

Theorems \ref{thm:discfinab}, \ref{thm:LUCpn}, and \ref{thm:Cpqlinisom} are stated in \S\S\ref{sec:Steiner}--\ref{sec:linisom} using different, but logically equivalent formalism. Briefly, an indexing system $\c{I}$ is completely determined by the orbits it contains, and with a bit of thought, one can also recast all structure in $\c{I}$ in terms of orbits. We call the result a \emph{transfer system} (Definition \ref{defn:transys}). The switch to transfer systems makes many of our computations easier, and it also streamlines our notation. The notion of a transfer system was also discovered in independent work of Balchin, Barnes, and Roitzheim \cite{BBR}. They use transfer system formalism in their beautiful proof that $\b{Ind}(C_{p^n})$ is isomorphic to the $(n+1)$st Stasheff polytope. We are confident that transfer systems will have further uses.

The remainder of this paper is organized as follows.
In \S\ref{sec:overview}, we review some background material and give a more leisurely introduction to the characterization problem.
In \S\ref{sec:transys}, we introduce transfer systems. We prove that they are equivalent to indexing systems (Theorem \ref{thm:transys}), and then we give a few examples (Figures \ref{fig:Cpn}--\ref{fig:S3}).
From here, we turn to the characterization problem. In \S\ref{sec:Steiner}, we analyze Steiner operads, starting with general finite groups (Theorem \ref{thm:discind}), and then specializing to finite abelian groups (Theorem \ref{thm:discfinab}). In \S\ref{sec:linisom}, we do the same for linear isometries operads. There is not much we can say in general, so we quickly specialize to finite cyclic groups $C_N$, and then to $C_{p^n}$ and $C_{pq}$ (Theorems \ref{thm:LUCpn} and \ref{thm:Cpqlinisom}). Appendix \ref{sec:filtTr} describes a general method for filtering the lattice $\b{Tr}(G)$ of all $G$-transfer systems (Construction \ref{constr:filTr}). These lattices can be quite large, even for small groups $G$ (cf. Figure \ref{fig:Q8}, p. \pageref{fig:Q8}), and our filtration helps organize the data. Appendix \ref{sec:genindsys} explains how to compute the transfer system generated from a prescribed set of relations (Construction \ref{constr:transysgen}), and then examines a few interesting cases.

\begin{conv}\label{conv:Glambdasigma} Throughout this paper, $G$ denotes a finite group with unit $e$. When $G = C_n$, we write $\lambda(k) = \lambda_n(k) : C_n \to SO(2) \cong S^1$ for the two-dimensional real representation of $C_n$ that sends a chosen generator $g \in C_n$ to $e^{2 \pi i k / n}$. When $G$ is non-cyclic, we use $\lambda$ to denote the pullback of such a representation along a quotient $G  \twoheadrightarrow C_n$. We write $\sigma$ for the sign representation of $C_2$ and its pullbacks.
\end{conv}

\subsection*{Acknowledgements} We thank Mike Hill for sharing many of the ideas behind this work, and for countless hours of conversation. We also thank Peter May for helpful comments on an earlier draft of this paper. This research was supported by NSF Grant DMS--1803426.

\section{The characterization problem} \label{sec:overview}

In this section, we describe the characterization problem (Problem \ref{prob:image}) and indicate some obstacles towards its solution (Proposition \ref{prop:notlattice}). The passage from the real representation theory of a group $G$ to the algebra of $G$-indexing systems is less transparent than one might hope, and this is why we take such a hands-on approach in later sections.

\subsection{Overview}

We begin by reviewing the basic theory of $N_\infty$ operads, with an eye towards Steiner and linear isometries operads. For further discussion, we recommend \cite{BH} and \cite{GM}.

Let $G$ be a finite group and let $\b{Top}^G$ be the category of left $G$-spaces. Throughout this paper, we understand $G$-operads to be symmetric operads in $\b{Top}^G$ with respect to the cartesian product. The prototypical example is the endomorphism $G$-operad for a $G$-space $X$. Its $n$th level is the $G$-space $\b{Top}_G(X^{\times n},X)$ of all continuous, but not necessarily equivariant, maps $X^{\times n} \to X$. The group $G$ acts by conjugation. Little discs operads constitute another important class of examples. Suppose $V$ is a finite dimensional real $G$-representation and let $D(V)$ denote the unit disc in $V$. A little $V$-disc is an affine, but not necessarily equivariant, map of the form $av + b : D(V) \to D(V)$. The $n$th level of the little $V$-discs operad $\c{D}(V)$ is the space of all $n$-tuples of disjoint little $V$-discs.

A $N_\infty$ $G$-operad is a $G$-operad $\s{O}$ that has the following three properties:
	\begin{enumerate}
		\item{}the $G \times \Sigma_n$-space $\s{O}(n)$ is $\Sigma_n$-free for every $n \geq 0$,
		\item{}the fixed-point subspace $\s{O}(n)^{\Gamma}$ is either empty or contractible for every $n \geq 0$ and subgroup $\Gamma \subset G \times \Sigma_n$, and
		\item{}the fixed-point subspace $\s{O}(n)^G$ is nonempty for every $n \geq 0$.
	\end{enumerate}
Such operads parametrize the multiplicative structures that typically remain on localizations of genuine commutative ring $G$-spectra. These operads also parametrize the canonical additive and multiplicative structures on spectra over incomplete universes. Recall that a $G$-universe is a countably infinite dimensional real $G$-inner product space, which contains each of its subrepresentations infinitely often, and which contains trivial summands. For any $G$-universe $U$, the natural multiplication on spectra indexed over $U$ is parametrized by the $N_\infty$ linear isometries operad $\c{L}(U)$. Its $n$th level is the space of all linear, but not necessarily equivariant, isometries $U^{\oplus n} \hookrightarrow U$. One would like to say that the natural additive structure is parametrized by the $N_\infty$ operad $\c{D}(U) = \t{colim}_{V \subset U} \c{D}(V)$, where $V$ ranges over all finite dimensional subrepresentations of $U$. Unfortunately, the operad $\c{D}(U)$ does not naturally act on equivariant infinite loop spaces, because the point-set level colimit that defines $\c{D}(U)$ is not compatible with suspension.

The standard workaround is to use Steiner operads $\c{K}(U)$ instead. Suppose $V$ is a  finite dimensional real $G$-representation and let $R_V$ be the $G$-space of distance-reducing, but not necessarily equivariant, embeddings $V \hookrightarrow V$. A $V$-Steiner path is a map $h : [0,1] \to R_V$ such that $h(1) = \t{id}$, and the $n$th level of the Steiner operad $\c{K}(V)$ is the space of all $n$-tuples $(h_1,\dots,h_n)$ of $V$-Steiner paths such that the images of $h_1(0)$, \dots, $h_n(0)$ are disjoint. For any $G$-universe $U$, we let $\c{K}(U) = \t{colim}_{V \subset U} \c{K}(V)$. These Steiner operads do act on equivariant infinite loop spaces.

We declare a map $\vp : \s{O}_1 \to \s{O}_2$ between $N_\infty$ operads to be a \emph{weak equivalence} if $\vp : \s{O}_1(n)^\Gamma \to \s{O}_2(n)^\Gamma$ is a weak homotopy equivalence of topological spaces for every $n \geq 0$ and subgroup $\Gamma \subset G \times \Sigma_n$. Under mild point-set level conditions, a weak equivalence between $N_\infty$ operads induces a Quillen equivalence between the associated model categories of algebra $G$-spectra \cite[Theorem A.3]{BH}. The Steiner operad $\c{K}(U)$ is equivalent to the infinite little discs operad $\c{D}(U)$, but there are generally universes $U$ such that $\c{K}(U)$ and $\c{L}(U)$ are inequivalent \cite[Theorem 4.22]{BH}. 

By the usual product trick \cite[Proposition 3.10]{MayGILS}, the homotopy type of a $N_\infty$ $G$-operad $\s{O}$ is completely determined by the subgroups $\Gamma \subset G \times \Sigma_n$ such that $\s{O}(n)^{\Gamma}$ is nonempty. Moreover, the set of such $\Gamma$ must be closed under subconjugacy, and must satisfy additional closure conditions that encode operadic composition. It is convenient to phrase these conditions in coordinate-free terms. By $\Sigma$-freeness, the relevant subgroups $\Gamma \subset G \times \Sigma_n$ all intersect $\{e\} \times \Sigma_n$ trivially. Such subgroups are typically called \emph{graph subgroups}. Every graph subgroup $\Gamma \subset G \times \Sigma_n$ is the graph of a permutation representation $\sigma : H \to \Sigma_n$ of some subgroup $H \subset G$. Conversely, if $T$ is a finite $H$-set, then the graph of a permutation representation of $T$ is a graph subgroup $\Gamma(T) \subset G \times \Sigma_{\abs{T}}$. We say that a $N_\infty$ operad $\s{O}$ \emph{admits} $T$ if $\s{O}(\abs{T})^{\Gamma(T)}$ is nonempty. The (graded) class of all admissible sets of a $N_\infty$ operad forms an \emph{indexing system} in the sense below.

\begin{defn}\label{defn:indsys} Let $G$ be a finite group and let $\b{Sub}(G)$ denote the set of all subgroups of $G$. A \emph{class of finite $G$-subgroup actions} is a class $\c{X}$, equipped with a function $\c{X} \to \b{Sub}(G)$, such that the fiber over each $H \in \b{Sub}(G)$ is a class of finite $H$-sets. Write $\c{X}(H)$ for the fiber over $H$.

A \emph{$G$-indexing system} $\c{I}$ is a class of finite $G$-subgroup actions which satisfies the following closure conditions:
	\begin{enumerate}
		\item{}(trivial sets) For any subgroup $H \subset G$,	the class $\c{I}(H)$ contains all finite, trivial $H$-actions.
		\item{}(isomorphism) For any subgroup $H \subset G$ and finite $H$-sets $S$ and $T$, if $S \in \c{I}(H)$ and $S \cong T$, then $T \in \c{I}(H)$.
		\item{}(restriction) For any subgroups $K \subset H \subset G$ and finite $H$-set $T$, if $T \in \c{I}(H)$, then $\t{res}^H_K T \in \c{I}(K)$.
		\item{}(conjugation) For any subgroup $H \subset G$, group element $g \in G$, and finite $H$-set $T$, if $T \in \c{I}(H)$, then $c_g T \in \c{I}(gHg^{-1})$.
		\item{}(subobjects) For any subgroup $H \subset G$ and finite $H$-sets $S$ and $T$, if $T \in \c{I}(H)$ and $S \subset T$, then $S \in \c{I}(H)$.
		\item{}(coproducts) For any subgroup $H \subset G$ and finite $H$-sets $S$ and $T$, if $S \in \c{I}(H)$ and $T \in \c{I}(H)$, then $S \sqcup T \in \c{I}(H)$.
		\item{}(self-induction) For any subgroups $K \subset H \subset G$ and finite $K$-set $T$, if $T \in \c{I}(K)$ and $H/K \in \c{I}(H)$, then $\t{ind}_K^H T \in \c{I}(H)$.
	\end{enumerate}
We call the elements of $\c{I}(H)$ the \emph{admissible $H$-sets of $\c{I}$}. Let $\b{Ind} = \b{Ind}(G)$ denote the poset of all $G$-indexing systems, ordered under inclusion.
\end{defn}

For any group $G$, there is a maximum indexing system $\ub{Set}$, whose $H$-fiber is the class of all finite $H$-sets, and there is a minimum indexing system $\ub{triv}$, whose $H$-fiber is the class of all finite, trivial $H$-actions. The meet of two indexing systems $\c{I}$ and $\c{J}$ is the levelwise intersection $(\c{I} \land \c{J})(H) = \c{I}(H) \cap \c{J}(H)$, and the join of two indexing systems is the smallest indexing system that contains the levelwise union $(\c{I} \cup \c{J})(H) = \c{I}(H) \cup \c{J}(H)$.  Thus $\b{Ind}$ is a lattice. It is finite because indexing systems are determined by the orbits they contain.

\begin{defn}A $G$-indexing system $\c{I}$ is a \emph{$\Lambda$-indexing system} if it also satisfies
	\begin{itemize}
		\item[$(\Lambda)$] For any subgroups $K \subset L \subset H \subset G$, if $H/K \in \c{I}(H)$, then $L/K \in \c{I}(L)$ and $H/L \in \c{I}(H)$.
	\end{itemize}
\end{defn}

If $\c{I}$ is any indexing system and $H/K \in \c{I}(H)$, then $L/K \in \c{I}(L)$ because $\c{I}$ is closed under restriction and subobjects. The extra condition for $\Lambda$-indexing systems is that $H/L \in \c{I}(H)$. The class of admissible sets of a linear isometries operad is always a $\Lambda$-indexing system \cite[p. 17]{BH}.

\begin{rem}Condition $(\Lambda)$ is a kind of horn-filling property. Suppose that we have a chain of subgroups $H_0 \subset H_1 \subset H_2 \subset \cdots \subset H_n$, regarded as a $n$-simplex in $\b{Sub}(G)$. If the orbit $H_n/H_0$ is admissible for a $\Lambda$-indexing system $\c{I}$, then every suborbit $H_i/H_j$ with $i \geq j$ must also be admissible for $\c{I}$.
\end{rem}

Taking admissible sets defines a functor $A : N_\infty\t{-}\b{Op} \to \b{Ind}$ from the category of $N_\infty$ $G$-operads to the poset category $\b{Ind}$ of $G$-indexing systems. The classification theorem states that $A$ induces an equivalence after we invert weak equivalences.

\begin{thm}[\cite{BH}, \cite{BonPer}, \cite{GutWhite}, \cite{Rubin}]\label{thm:classNinfty} Taking admissible sets induces an equivalence $A : \t{Ho}(N_\infty\t{-}\b{Op}) \to \b{Ind}$ of $1$-categories.
\end{thm}

In their pioneering work, Blumberg and Hill proved that $A : \t{Ho}(N_\infty\t{-}\b{Op}) \to \b{Ind}$ is full and faithful \cite[Theorem 3.24]{BH}, and subsequent, independent work in \cite{BonPer}, \cite{GutWhite}, and \cite{Rubin} established surjectivity. However, the constructions in \cite{BonPer}, \cite{GutWhite}, and \cite{Rubin} are all essentially algebraic. For example, the simplest $N_\infty$ operads considered in \cite{Rubin} are constructed by generating a free discrete operad on the desired operations, and then attaching cells to kill all homotopy. All $N_\infty$ operads arise in this way, up to equivalence, which reflects the fact that the definitions of $N_\infty$ operads and indexing systems only axiomatize general features of equivariant composition. It is natural to ask how the geometry of Steiner and linear isometries operads is encoded by the algebra of indexing systems.

\begin{prob}\label{prob:image} Given a finite group $G$, identify extra algebraic conditions on indexing systems that characterize the images of the Steiner operads and linear isometries operads under the map $A : \t{Ho}(N_\infty\t{-}\b{Op}) \to \b{Ind}$.
\end{prob}

We shall solve this problem in a few, special cases.

\subsection{Structural obstacles}

Suppose that $U$ is a $G$-universe, and let $\c{K}(U)$ and $\c{L}(U)$ be the corresponding Steiner and linear isometries operads. Problem \ref{prob:image} asks what the possible values of $A(\c{K}(U))$ and $A(\c{L}(U))$ are. One's first thought might be to leverage relations between universes into relations between indexing systems. Unfortunately, this approach does not work as well as one might hope.

To start, note that the admissible sets of $\c{K}(U)$ and $\c{L}(U)$ depend only on the isomorphism class of $U$. Thus, we introduce notation.

\begin{defn} Let $\b{Uni} = \b{Uni}(G)$ denote the set of all isomorphism classes $[U]$ of $G$-universes $U$.
\end{defn}

We declare $[U] \leq [U']$ if there is a $G$-embedding $U \hookrightarrow U'$ for some representatives $U$ and $U'$. The minimum element of $\b{Uni}$ is the class of a trivial universe, and the maximum element is the class of a complete universe. The join of $[U]$ and $[U']$ is represented by $U \oplus U'$, and the meet $[U] \land [U']$ is the universe that contains infinitely many copies of each irreducible $V$ that embeds into both $U$ and $U'$. Thus $\b{Uni}$ is a lattice. It is isomorphic to a $n$-cube, where $n$ is the number of nontrivial irreducible real representations of $G$, up to isomorphism.

The lattice $\b{Uni}(G)$ carries a right action by the group $\b{Aut}(G)$ of automorphisms of $G$. Given a class $[U] \in \b{Uni}(G)$ and an automorphism $\sigma \in \b{Aut}(G)$, we declare $[U] \cdot \sigma$ to be the class represented by the $G$-universe $G \stackrel{\sigma}{\to} G \to O(U)$. On the other hand, $\b{Ind}(G)$ is also a lattice, and it inherits a right $\b{Aut}(G)$-action from the corresponding action on $G$. Explicitly, given $\sigma \in \b{Aut}(G)$ we declare
	\begin{enumerate}[label=(\roman*)]
		\item{} $g \cdot \sigma = \sigma^{-1}(g)$ for all $g \in G$,
		\item{} $H \cdot \sigma = \sigma^{-1} H$ for every subgroup $H \subset G$,
		\item{} $T \cdot \sigma = [ \sigma^{-1}H \stackrel{\sigma}{\to} H \to \b{Perm}(T)]$ for every subgroup $H \subset G$ and finite $H$-set $T$, and
		\item{} $\c{I} \cdot \sigma = \{T \cdot \sigma \, | \, T \in \c{I} \}$ for every $G$-indexing system $\c{I}$.
	\end{enumerate}
These formulas define $\b{Aut}(G)$-actions on $G$, $\b{Sub}(G)$, $\ub{Set}$, and $\b{Ind}(G)$.

In an ideal world, the functions
	\[
	A(\c{K}(-)) , \, A(\c{L}(-)) : \b{Uni}(G) \rightrightarrows \b{Ind}(G)
	\]
would preserve all structure in sight. But our world is not ideal.

\begin{prop}\label{prop:notlattice} Neither $A(\c{K}(-))$ nor $A(\c{L}(-))$ is a lattice map in general.
\end{prop}

\begin{proof} Example \ref{ex:C5notlat} shows that $A(\c{K}(-))$ does not preserve meets when $G = C_5$, and Example \ref{ex:C9notlat} shows that $A(\c{L}(-))$ does not preserve the order when $G = C_9$. The calculations in Examples \ref{ex:K4linisom}, \ref{ex:Q8linisom}, and \ref{ex:C6linisom} provide similar counterexamples.
\end{proof}

By \cite[Theorem 4.19]{BH}, the operad $\c{K}(U) \simeq \c{D}(U)$ admits an orbit $H/K$ if and only if there is an $H$-embedding $H/K \hookrightarrow \t{res}^G_H U$. This complicates things when nonisomorphic $G$-representations decompose into the same orbits. Recall the $C_n$-representations $\lambda(k)$ described in Convention \ref{conv:Glambdasigma}.

\begin{ex}\label{ex:C5notlat} Suppose $G = C_5$. The $C_5$-universes $U_1 = (\bb{R} \oplus \lambda(1))^\infty$ and $U_2 = (\bb{R} \oplus \lambda(2))^\infty$ are incomparable, but the free orbit $C_5/C_1$ embeds into both of them. Hence $A(\c{K}(U_1)) = A(\c{K}(U_2)) = \ub{Set}$, and hence $A(\c{K}(U_1)) \land A(\c{K}(U_2)) = \ub{Set}$. On the other hand, we have $[U_1] \land [U_2] = [\bb{R}^\infty]$, and thus $A(\c{K}([U_1] \land [U_2])) = \ub{triv}$. Therefore $A(\c{K}(-))$ does not preserve meets.
\end{ex}

As for linear isometries operads, \cite[Theorem 4.18]{BH} states that the operad $\c{L}(U)$ admits an orbit $H/K$ if and only if there is an $H$-embedding $\t{ind}_K^H\t{res}^G_K U \hookrightarrow \t{res}^G_H U$. This complicates things because we cannot analyze the problem one irreducible subrepresentation of $U$ at a time.

\begin{ex}\label{ex:C9notlat} Let $G = C_9$ and consider the incomplete universes $U_1 = (\bb{R} \oplus \lambda(3))^\infty$ and $U_2 = (\bb{R} \oplus \lambda(1) \oplus \lambda(3))^\infty$. Then $[U_1] < [U_2]$, but we shall see that $A(\c{L}(U_1))$ and $A(\c{L}(U_2))$ are incomparable.

First, consider the admissibles of $\c{L}(U_1)$. The restriction $\t{res}^{C_9}_{C_3} U_1$ is a trivial $C_3$-universe and $\t{ind}_{C_3}^{C_9}\t{res}^{C_9}_{C_3}U_1 \cong U_1$. Hence $\c{L}(U_1)$ admits $C_9/C_3$. On the other hand, $\t{ind}_{C_1}^{C_3}\t{res}^{C_9}_{C_1} U_1$ is a complete $C_3$-universe, and thus $\c{L}(U_1)$ does not admit $C_3/C_1$.

Now consider the admissibles of $\c{L}(U_2)$. The restriction $\t{res}^{C_9}_{C_3}U_2$ is a complete $C_3$-universe, and hence $\c{L}(U_2)$ admits $C_3/C_1$. On the other hand, $\t{ind}_{C_3}^{C_9}\t{res}^{C_9}_{C_3}U_2$ is a complete $C_9$-universe, and therefore $\c{L}(U_2)$ does not admit $C_9/C_3$.
\end{ex}

That being said, we can salvage the situation to some extent.

\begin{prop}\label{prop:admDUmap} The function $A_{\c{K}} = A(\c{K}(-)) : \b{Uni} \to \b{Ind}$ is $\b{Aut}(G)$-equivariant, and it preserves the order, the maximum element, the minimum element, and joins. It is not always order-reflecting, meet-preserving, or injective.
\end{prop}

\begin{proof} Composing embeddings of orbits with embeddings of universes proves that $A_{\c{K}}$ is order-preserving, and applying $(-) \cdot \sigma$ and $(-) \cdot \sigma^{-1}$ shows that $T$ embeds into $\t{res}^G_H U$ if and only if $T \cdot \sigma$ embeds into $U \cdot \sigma$. It follows that $A_{\c{K}}$ preserves the $\b{Aut}(G)$-action. We have $A_{\c{K}}([\bb{R}^\infty]) = \ub{triv}$ because the only orbits that embed in $\bb{R}^\infty$ are trivial, and $A_{\c{K}}([\bb{R}[G]^\infty]) = \ub{Set}$ because every orbit embeds in $\bb{R}[G]$. Proposition \ref{prop:discind} implies that $A_{\c{K}}$ preserves joins, and Example \ref{ex:C5notlat} shows that $A_{\c{K}}$ need not reflect the order, preserve meets, or be injective.
\end{proof}

\begin{prop}\label{prop:admLUmap} The function $A_{\c{L}} = A(\c{L}(-)) : \b{Uni} \to \b{Ind}$ is $\b{Aut}(G)$-equivariant, and it preserves maximum and minimum elements. It is not always order-preserving, order-reflecting, join-preserving, meet-preserving, or injective.
\end{prop}

\begin{proof} We begin with the $\b{Aut}(G)$-equivariance. Right multiplication $(-) \cdot \sigma$ preserves embeddings, and it commutes with restriction and induction. Therefore $\t{ind}_K^H\t{res}^G_H U$ embeds into $\t{res}^G_H U$ if and only if $\t{ind}_{\sigma^{-1}K}^{\sigma^{-1}H} \t{res}^G_{\sigma^{-1}H} (U\sigma)$ embeds into $\t{res}^G_{\sigma^{-1}H}(U\sigma)$. It follows $H/K \in A_{\c{L}}([U])$ if and only if $(H/K)\sigma \in A_{\c{L}}([U]\sigma)$, and passing to coproducts shows $A_{\c{L}}([U]) \sigma = A_{\c{L}}([U]\sigma)$. The map $A_{\c{L}}$ preserves minimum and maximum elements because no nontrivial universe embeds in a trivial one, and every universe embeds into a complete one.

Consider the universes in Example \ref{ex:C5notlat} once more. Keeping the same notation, we have $A_{\c{L}}(U_1) = A_{\c{L}}(U_2) = \ub{triv}$, and therefore $A_{\c{L}}$ is not injective or order-reflecting for $G = C_5$. Example \ref{ex:C9notlat} shows that $A_{\c{L}}$ is not order-preserving for $G=C_9$, and therefore $A_{\c{L}}$ does not  preserve all joins and meets in that case, either.
\end{proof}

The failure of $A_{\c{K}}$ to preserve meets is a nuisance, but it is counterbalanced by the fact that $A_{\c{K}}$ preserves joins. The failure of $A_{\c{L}}$ to preserve the order is more serious. It precludes a clean, structural approach to Problem \ref{prob:image} for linear isometries operads. To move forward, we elaborate on Blumberg and Hill's calculations of $A_{\c{K}}$ and $A_{\c{L}}$, and then we study the formulas that fall out.

\section{Transfer systems}\label{sec:transys}

In this section, we take a detour to introduce formalism that simplifies our discussion of Problem \ref{prob:image}. Indexing systems are proper class-sized objects, but they are determined by finite sets of orbits. Reformulating Definition \ref{defn:indsys} in these terms leads to our notion of a \emph{transfer system} (Definition \ref{defn:transys}). We prove that transfer systems are equivalent to indexing systems (Theorem \ref{thm:transys}) and to the indexing categories of \cite{BH2} (Corollary \ref{cor:transysBH2}). We also give a handful of examples in \S\ref{subsec:exts}. We reiterate that Balchin, Barnes, and Roitzheim \cite{BBR} have independently developed the same formalism.

\subsection{The data in an indexing system}

There are several ways to think of indexing systems. From an operadic standpoint, they are equivalent to homotopy types of $N_\infty$ operads (Theorem \ref{thm:classNinfty}). From an algebraic standpoint, they are equivalent to indexing categories in the sense below \cite[Theorem 3.17]{BH2}.

\begin{defn}Let $\b{Set}_{fin}^G$ denote the category of finite $G$-sets. An \emph{indexing category} is a wide, pullback stable, finite coproduct complete subcategory $\s{D} \subset \b{Set}_{fin}^G$. We write $\b{IndCat}$ for the poset of all indexing categories.
\end{defn}

Such categories naturally parametrize the transfers on incomplete Mackey functors and the norms on incomplete Tambara functors.

Our transfer systems encode generating data in indexing systems and indexing categories. Informally, a transfer system is a diagram of the orbits in an indexing system, or the intersection of an indexing category $\s{D} \subset \b{Set}^G_{fin}$ with the orbit category $\s{O}_G$. We consider the relationship to indexing systems first.

\begin{defn}\label{defn:graphindsys} Suppose $\c{I}$ is a $G$-indexing system. We define the \emph{graph of $\c{I}$} to be the set $\b{Sub}(G)$, equipped with the binary relation $\to_{\c{I}}$ below:
	\[
	K \to_{\c{I}} H \quad\t{ if and only if }\quad K \subset H \t{ and } H/K \in \c{I}.
	\]
\end{defn}

We think of subgroups $H \subset G$ as vertices, and relations $K \to_{\c{I}} H$ as directed edges. The indexing system axioms imply the following properties of $\to_{\c{I}}$.

\begin{prop}\label{prop:indsystransys}Suppose that $\c{I}$ is a $G$-indexing system. Then $\to \,\, = \,\, \to_{\c{I}}$ is:
	\begin{enumerate}[label=(\alph*)]
		\item{}a partial order,
		\item{}a refinement of the subset relation: if $K \to H$, then $K \subset H$,
		\item{}closed under conjugation: if $K \to H$, then $(gKg^{-1}) \to (gHg^{-1})$ for every group element $g \in G$, and
		\item{}closed under restriction: if $K \to H$ and $L \subset H$, then $(K \cap L) \to L $.
	\end{enumerate}
If $\c{I}$ is a $\Lambda$-indexing system, then $\to$ also is:
	\begin{enumerate}[resume,label=(\alph*)]
		\item{}saturated: if $K \to H$ and $K \subset L \subset H$, then $K \to L$ and $L \to H$.
	\end{enumerate}
\end{prop}

\begin{proof} Part $(b)$ holds by fiat. For $(a)$, reflexivity holds because $\c{I}$ contains all trivial actions, and antisymmetry follows from $(b)$. For transitivity, suppose $K \to L$ and $L \to H$. Then $L/K \in \c{I}$ and $H/L \in \c{I}$, and hence $H/K \cong \t{ind}_L^H L/K \in \c{I}$ because $\c{I}$ is closed under isomorphism and self-induction. Condition $(c)$ holds because if $K \to H$, then $H/K \in \c{I}$, and hence $gHg^{-1}/gKg^{-1} \cong c_g H/K \in \c{I}$ because $\c{I}$ is closed under isomorphism and conjugation. Condition $(d)$ holds because we have an embedding $L/(L \cap K) \hookrightarrow \t{res}^H_L H/K$, and $\c{I}$ is closed under restriction and subobjects. Condition $(e)$ is a restatement of condition $(\Lambda)$.
\end{proof}

Thus, we make a definition.

\begin{defn}\label{defn:transys} Let $G$ be a finite group. A \emph{$G$-transfer system} is a partial order on $\b{Sub}(G)$, which refines the subset relation, and which is closed under conjugation and restriction in the sense of Proposition \ref{prop:indsystransys}. We use arrows $\to$ to denote transfer systems. A transfer system $\to$ is \emph{saturated} if it also satisfies condition $(e)$ above. Let $\b{Tr} = \b{Tr}(G)$ denote the poset of all $G$-transfer systems $\to$ ordered under refinement, i.e. declare $\to_1 \, \leq \, \to_2$ if and only if $K \to_1 H$ implies $K \to_2 H$ for all $K,H \subset G$.
\end{defn}

\begin{rem} We explain the terminology. Suppose $\s{O}$ is a $N_\infty$ $G$-operad. The transfer system $\to_{\s{O}}$ corresponding to the class of $\s{O}$-admissible sets satisfies
	\[
	K \to_{\s{O}} H	\quad\t{if and only if}\quad	 K \subset H \t{ and }\s{O}(\abs{H:K})^{\Gamma(H/K)} \neq \varnothing .
	\]
We shall see that a relation $K \to_{\s{O}} H$ gives rise to a transfer map on $\s{O}$-algebras.

Suppose $K$ and $H$ are subgroups such that $K \to_{\s{O}} H$, and write $n = \abs{H:K}$. Order $H/K$ as $\{r_1K < \dots < r_n K\}$, let $\Gamma = \Gamma(H/K)$ be graph of the the corresponding permutation representation $\sigma : H \to \Sigma_n$, and choose a $\Gamma$-fixed operation $f \in \s{O}(n)$. If $X$ is an $\s{O}$-algebra $G$-space, then there is a transfer map
	\[
	\t{tr}_K^H(x) = f(r_1 x , \dots , r_n x) : X^K \to X^H .
	\]
On the other hand, if we regard $f$ as a map $\frac{G\times \Sigma_n}{\Gamma} \to \s{O}(n)$, then we obtain a $G$-map
	\[
	\ol{\t{tr}}_K^H : G \times_H X^{\times H/K} \cong \frac{G \times \Sigma_n}{\Gamma} \underset{\Sigma_n}{\times} X^{\times n} \stackrel{f \times \t{id}}{\longrightarrow} \s{O}(n) \underset{\Sigma_n}{\times} X^{\times n} \longrightarrow X .
	\]
Here $X^{\times H/K}$ is the space $X^{\times n}$ equipped with the $H$-action
	\[
	h (x_1, \dots, x_n) = (h x_{\sigma(h)^{-1}1} , \dots, h x_{\sigma(h)^{-1}n} ) 
	\]
and the isomorphism $G \times_H X^{\times H/K} \cong \frac{G \times \Sigma_n}{\Gamma} \times_{\Sigma_n} X^{\times n}$ identifies $[g , (x_1,\dots,x_n) ]$ with $[ [g,1] , (gx_1,\dots,gx_n)]$. The map $\ol{\t{tr}}_K^H$ is an external version of $\t{tr}_K^H : X^K \to X^H$. We recover $\t{tr}_K^H$ by taking $H$-fixed points of the adjoint $X^{\times H/K} \to \t{res}^G_H X$, and then composing with the map $X^K \cong (X^{\times H/K})^H$ that identifies $x$ with $(r_1 x , \dots , r_n x)$.

Similarly, if $E$ is an $\s{O}$-algebra $G$-spectrum, then by \cite[Construction 6.5]{BH}, we obtain an external norm map 
	\[
	\ol{\t{n}}_K^H : G_+ \land_H N_K^H \t{res}^G_K E \cong \frac{G \times \Sigma_n}{\Gamma} {}_+ \underset{\Sigma_n}{\land} E^{\land n} \stackrel{f_+ \land \t{id}}{\longrightarrow} \s{O}(n)_+ \underset{\Sigma_n}{\land} E^{\land n} \longrightarrow E .
	\]
\end{rem}

The construction of a transfer system from an indexing system is reversible, because indexing systems are determined by their orbits.

\begin{prop}\label{prop:transysindsys}If $\to$ is a $G$-transfer system, then there is a unique $G$-indexing system $\c{I} = \c{I}_{\to}$ such that $\to_{\c{I}} \,\, = \,\, \to$. More specifically, $\c{I}_{\to}(H)$ is the class of all finite coproducts of $H$-orbits $H/K$ such that $K \to H$. The transfer system $\to$ is saturated if and only if $\c{I}_{\to}$ is a $\Lambda$-indexing system.
\end{prop}

\begin{proof}Fix a transfer system $\to$. If $\c{I}$ is an indexing system such that $\to_{\c{I}} \,\, = \,\, \to$, then the orbits of $\c{I}$ must be those $H/K$ such that $K \to H$, and $\c{I}$ must be the class of all finite coproducts of such orbits. Therefore $\c{I}$ is unique if it exists.

We check that this recipe works. Define
	\[
	\c{I}_{\to}(H) := \Bigg\{ \t{ finite $H$-sets } T \, \Bigg| \, \begin{array}{c}
		\t{there exist $n \geq 0$ and $K_1,\dots, K_n \subset H$ such that} \\
		\t{$T \cong \coprod_{i=1}^n H/K_i$ and $K_i \to H$ for $i = 1, \dots, n$}
		\end{array}
	\Bigg\},
	\]
where empty coproducts are understood to be $\varnothing$. We must check that $\c{I} = \c{I}_{\to}$ is a $G$-indexing system, and that $\to_{\c{I}} \,\, = \,\, \to$.

We verify the axioms in Definition \ref{defn:indsys}. Condition (1) holds because $\to$ is reflexive. Condition (2) holds because coproducts are only defined up to isomorphism. Condition (3) holds because if $T \cong \coprod_{i=1}^n H/K_i$ with $K_i \to H$, then for any $L \subset H$,
	\[
	\t{res}^H_L T \cong \coprod_{i=1}^n \t{res}^H_L H/K_i \cong \coprod_{i=1}^n \coprod_{a \in L \backslash H / K_i} L/(L \cap aK_i a^{-1}).
	\]
The right hand side is a finite coproduct, and if $K_i \to H$, then for any $a \in L \backslash H /K_i$, we have $(a K_i a^{-1}) \to (a H a^{-1}) = H$ and also $(L \cap a K_i a^{-1}) \to L$, because $\to$ is closed under conjugation and restriction. Condition (4) holds because if $T \cong \coprod_{i=1}^n H/K_i$, then $c_g T \cong \coprod_{i=1}^n gHg^{-1}/gK_i g^{-1}$, and $\to$ is closed under conjugation. Condition (5) holds because every subobject of $T \cong \coprod_i H/K_i \in \c{I}$, is still just a finite coproduct of orbits $H/K$ with $K \to H$. Similarly for condition (6).

Suppose that $H/K \in \c{I}$. Then $H/K \cong H/K'$ for some $K' \to H$. Therefore $K = hK' h^{-1}$ for some $h \in H$, and thus $K = h K' h^{-1} \to h H h^{-1} = H$. Condition (7) follows, because if $T \cong \coprod_{i=1}^n K/L_i \in \c{I}$ for some $L_i \to K$ and $H/K \in \c{I}$, then $K \to H$, and therefore $L_i \to H$ by transitivity. Thus, $\t{ind}_K^H T \cong \coprod_{i=1}^n H/L_i \in \c{I}$. This proves that $\c{I}$ is an indexing system, and it is easy to see that $\to_{\c{I}} \, = \, \to$.

Suppose the transfer system $\to$ is saturated. If $H/K \in \c{I}$ and $K \subset L \subset H$, then $K \to H$ as above, and therefore $K \to L \to H$. Hence $L/K, H/L \in \c{I}$, and hence $\c{I}$ is a $\Lambda$-indexing system. The converse is similar.
\end{proof}

In summary, we obtain the following result.

\begin{thm}\label{thm:transys} The maps $\to_{\bullet} \, : \b{Ind} \stackrel{}{\rightleftarrows} \b{Tr} : \c{I}_{\bullet}$ are inverse order isomorphisms, and they restrict to an isomorphism between the subposet of $\Lambda$-indexing systems and the subposet of saturated transfer systems.
\end{thm}

\begin{proof} The set maps $\to_\bullet$ and $\c{I}_\bullet$ are inverse by Propositions \ref{prop:indsystransys} and \ref{prop:transysindsys}. We must check that they are order-preserving. Suppose that $\c{I} \subset \c{J}$. If $K \to_{\c{I}} H$, then $H/K \in \c{I} \subset \c{J}$, and therefore $K \to_{\c{J}} H$. Thus $\to_{\c{I}}$ refines $\to_{\c{J}}$. Conversely, if $\to_{\c{I}}$ refines $\to_{\c{J}}$,  then every orbit in $\c{I}$ is also contained in $\c{J}$. Therefore $\c{I} \subset \c{J}$, because $\c{I}$ is generated by its orbits.
\end{proof}

This makes precise the intuition that transfer systems are the sets of orbits in indexing systems.

We now consider the relationship between transfer systems and indexing categories, starting with a review of Blumberg and Hill's isomorphism $\b{Ind} \cong \b{IndCat}$. For any indexing system $\c{I}$, let $\b{Set}^G_{\c{I}} \subset \b{Set}^G_{fin}$ be the indexing category consisting of those $f : S \to T$ such that $G_{f(s)} / G_s \in \c{I}$ for every $s \in S$. Conversely, given any indexing category $\s{D} \subset \b{Set}_{fin}^G$, let $\c{I}_{\s{D}}$ be the indexing system whose admissible $H$-sets are those $T$ such that $T = p^{-1}(eH)$ for some $p : S \to G/H$ in $\s{D}$.\footnote{The indexing system $\c{I}_{\s{D}}$ is obtained from the construction in \cite[Lemma 3.18]{BH2} by composing with the equivalence $\b{Set}^G_{/G/H} \simeq \b{Set}^H$, and then taking object classes.}

\begin{thm}[\protect{\cite[Theorem 3.17]{BH2}}]\label{thm:BH2} The maps $\b{Set}^G_{\bullet} : \b{Ind} \rightleftarrows \b{IndCat} : \c{I}_{\bullet}$ are inverse lattice isomorphisms.
\end{thm}

We obtain a composite isomorphism $\b{Tr} \cong \b{Ind} \cong \b{IndCat}$. Unwinding the definitions and simplifying yields the following formulas. For any transfer system $\rightsquigarrow \, \in \b{Tr}$, let $\b{Set}^G_{\rightsquigarrow} \in \b{IndCat}$ consist of those morphisms $f : S \to T$ in $\b{Set}^G_{fin}$ such that $G_s \rightsquigarrow G_{f(s)}$ for every $s \in S$. Conversely, for any $\s{D} \in \b{IndCat}$, let $\to_{\s{D}} \, \in \b{Tr}$ be the transfer system defined by
	\[
	K \to_{\s{D}} H	\quad\t{ if and only if }\quad K \subset H \t{ and } ( \pi : G/K \to G/H ) \in \s{D},
	\]
where $\pi$ is the canonical projection map $\pi(gK) = gH$.

\begin{cor}\label{cor:transysBH2} The lattice maps $\b{Set}^G_{\bullet} : \b{Tr} \rightleftarrows \b{IndCat} : {}\to_{\bullet}$ are inverse.
\end{cor}

This makes precise the intuition that transfer systems are the intersection of indexing categories with the orbit category $\s{O}_G$.

There is a chain of equivalences
	\[
	\t{Ho}(N_\infty\t{-}\b{Op}) \simeq \b{Ind} \cong \b{IndCat} \cong \b{Tr},
	\]
and therefore these structures all contain the same information. It is easy to identify the essential group-theoretic data from the standpoint of transfer systems.

\begin{cor}\label{cor:transys} The lattices $\b{Ind}$, $\b{IndCat}$, and $\b{Tr}$, and the $1$-category $\t{Ho}(N_\infty\t{-}\b{Op})$ are determined by the lattice $\b{Sub}(G)$, together with the orbit space of 
	\[
	\subset_G  \,\, = \{(K,H) \in \b{Sub}(G)^{\times 2} \, | \, K \subset H\}
	\]
under the diagonal conjugation $G$-action.
\end{cor}

In particular, the lattice $\b{Sub}(G)$ determines everything if $G$ is finite abelian, or if all subgroups of $G$ are normal (e.g. if $G = Q_8$). In general, we must remember $\subset_G \!\! / G$ and not just the set $\b{Sub}(G)/G$ of conjugacy classes of subgroups, because the actions on the fibers of $\subset_G \,\,\twoheadrightarrow (\b{Sub}(G)/G)^{\times 2}$ need not be transitive.

\begin{ex}\label{ex:conjS4} Let $G = \Sigma_4$. There are three conjugate copies of $D_8$ in $\Sigma_4$, obtained by ordering the vertices of a square, and then taking the images of the associated permutation representations. There are three double-transpositions in $\Sigma_4$, which generate three conjugate copies of $C_2$. These subgroups determine a copy of the bipartite graph
	
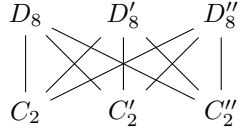
\[
	\begin{tikzpicture}[scale=1.3]
		\node(00) at (0,0) {$C_2$};
		\node(10) at (1,0) {$C_2'$};
		\node(20) at (2,0) {$C_2''$};
		\node(01) at (0,1) {$D_8$};
		\node(11) at (1,1) {$D_8'$};
		\node(21) at (2,1) {$D_8''$};
		\path[-]
		(00) edge node {} (01)
		(00) edge node {} (11)
		(00) edge node {} (21)
		(10) edge node {} (01)
		(10) edge node {} (11)
		(10) edge node {} (21)
		(20) edge node {} (01)
		(20) edge node {} (11)
		(20) edge node {} (21)
		;
	\end{tikzpicture}
	\]
in $\b{Sub}(\Sigma_4)$. For each copy of $D_8$, one inclusion of $C_2$ corresponds to the rotation by $\pi$, and the other two inclusions correspond to reflections. Without loss of generality, we may assume that the vertical inclusions above are the rotations. We obtain two conjugacy classes of edges:
	
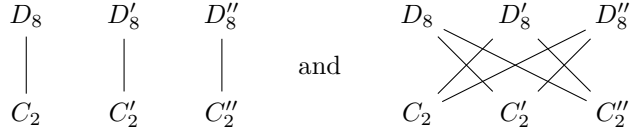
\[
	\begin{tikzpicture}[scale=1.3]
		\node(00) at (0,0) {$C_2$};
		\node(10) at (1,0) {$C_2'$};
		\node(20) at (2,0) {$C_2''$};
		\node(01) at (0,1) {$D_8$};
		\node(11) at (1,1) {$D_8'$};
		\node(21) at (2,1) {$D_8''$};
		\path[-]
		(00) edge node {} (01)
		(10) edge node {} (11)
		(20) edge node {} (21)
		;
		
		\node(00') at (4,0) {$C_2$};
		\node(10') at (5,0) {$C_2'$};
		\node(20') at (6,0) {$C_2''$};
		\node(01') at (4,1) {$D_8$};
		\node(11') at (5,1) {$D_8'$};
		\node(21') at (6,1) {$D_8''$};
		\path[-]
		(00') edge node {} (11')
		(00') edge node {} (21')
		(10') edge node {} (01')
		(10') edge node {} (21')
		(20') edge node {} (01')
		(20') edge node {} (11')
		;
		
		\node at (3,0.5) {and};
	\end{tikzpicture} .
	\]
Thus, to specify a $\Sigma_4$-transfer system, it is not enough to declare $[C_2] \to [D_8]$. We must also know which copies of $C_2$ are related to which copies of $D_8$.
\end{ex}

\subsection{Examples of transfer systems}\label{subsec:exts}

We now describe the lattice $\b{Tr}(G)$ for a few small groups $G$. These examples illustrate how our formalism works, and they will be useful in the upcoming discussion of Steiner and linear isometries operads.

The lattice $\b{Tr}(G)$ is determined by $\b{Sub}(G)$, equipped with the conjugation $G$-action. Thus, we focus on groups with small subgroup lattices. We start with the case of a tower. Balchin, Barnes, and Roitzheim have proven a marvelous theorem. Using a clever inductive argument, they show that $\b{Tr}(C_{p^n})$ is isomorphic to the $(n+1)$st associahedron for any prime $p$ and integer $n \geq 0$. To give the idea, we draw $\b{Tr}(C_{p^n})$ for $0 \leq n \leq 3$ in Figure \ref{fig:Cpn} (p. \pageref{fig:Cpn}), but we heartily recommend their paper for the general argument.

\begin{figure}
	\caption{}
	\label{fig:Cpn}
	\[
	\begin{array}{c}
	\begin{tikzpicture}[scale=1.25]
		\node(01) at (0,1) {$C_1$};
		\node(11) at (0,2) {$C_p$};
		\path[-]
		(01) edge node {} (11)
		;
		
		\node(A1) at (1.5,1) {$\cpa$};
		\node(B1) at (1.5,2) {$\cpb$};
		\path[-]
		(A1) edge node {} (B1)
		;
		
		\node at (0,3.75) {$\b{Sub}(C_{p^1})$};
		\node at (1.5,3.75) {$\b{Tr}(C_{p^1})$};
		
		\node(00) at (-3,1.5) {$C_1$};
		\node(A0) at (-1.5,1.5) {$\cdot$};
		
		\node at (-3,3.75) {$\b{Sub}(C_{p^0})$};
		\node at (-1.5,3.75) {$\b{Tr}(C_{p^0})$};
		
		\node(02) at (3,0.5) {$C_1$};
		\node(12) at (3,1.5) {$C_p$};
		\node(22) at (3,2.5) {$C_{p^2}$};
		
		\path[-]
		(02) edge node {} (12)
		(12) edge node {} (22)
		(02) edge [bend right] node {} (22)
		;
		
		\node(A2) at (4.5,0) {$\cppa$};
		\node(B2) at (4.5,1) {$\cppb$};
		\node(C2) at (4.5,2) {$\cppc$};
		\node(D2) at (5.5,1) {$\cppd$};
		\node(E2) at (5.5,3) {$\cppe$};
		
		\path[-]
		(A2) edge node {} (B2)
		(B2) edge node {} (C2)
		(A2) edge node {} (D2)
		(C2) edge node {} (E2)
		(D2) edge node {} (E2)
		;
		
		\node at (3,3.75) {$\b{Sub}(C_{p^2})$};
		\node at (5,3.75) {$\b{Tr}(C_{p^2})$};
		
		\draw (-0.75,-0.25) -- (-0.75,4);
		\draw (2.25,-0.25) -- (2.25,4);
		\draw (-3.75,3.5) -- (5.75,3.5);
	\end{tikzpicture}
	\\
	\begin{tikzpicture}[scale=1.4]
		\node(A) at (0,0) {$\cpppa$};
		
		\node(B) at (0,1) {$\cpppb$};
		
		\node(C) at (0,2) {$\cpppc$};
		\node(D) at (1,1) {$\cpppd$};
		\node(E) at (1,3) {$\cpppe$};
		
		\node(F) at (0,3) {$\cpppf$};
		\node(G) at (1,4) {$\cpppg$};
		\node(H) at (2,2) {$\cppph$};
		\node(I) at (2,5) {$\cpppi$};
		
		\node(J) at (-3,1) {$\cpppj$};
		\node(K) at (-1,3) {$\cpppk$};
		\node(L) at (-3,2) {$\cpppl$};
		\node(M) at (-3,4) {$\cpppm$};
		\node(N) at (-1,6) {$\cpppn$};
		
		\path[-]
		(A) edge node {} (B)
		
		(B) edge node {} (C)
		(A) edge node {} (D)
		(C) edge node {} (E)
		(D) edge node {} (E)
		
		(C) edge node {} (F)
		(E) edge node {} (G)
		(F) edge node {} (G)
		
		(D) edge node {} (H) 
		(G) edge node {} (I)
		(H) edge node {} (I)
		
		(A) edge node {} (J)
		(H) edge node {} (K)
		(J) edge node {} (K)
		(J) edge node {} (L)
		(B) edge node {} (L)
		(L) edge node {} (M)
		(F) edge node {} (M)
		(M) edge node {} (N)
		(K) edge node {} (N)
		(I) edge node {} (N)
		;
		
		\node(0) at (-5.5,2) {$C_1$};
		\node(1) at (-5.5,3) {$C_p$};
		\node(2) at (-5.5,4) {$C_{p^2}$};
		\node(3) at (-5.5,5) {$C_{p^3}$};
		
		\path[-]
		(0) edge node {} (1)
		(1) edge node {} (2)
		(2) edge node {} (3)
		
		(0) edge [bend right] node {} (2)
		(0) edge [bend right] node {} (3)
		(1) edge [bend left] node {} (3)
		;
		
		\node at (-5.5,6.75) {$\b{Sub}(C_{p^3})$};
		\node at (-0.5,6.75) {$\b{Tr}(C_{p^3})$};
		
		\draw (-6.5,6.5) -- (2.25,6.5);
	\end{tikzpicture}
	\end{array}
	\]
\end{figure}

Next, we generalize orthogonally. The lattice $\b{Sub}(C_{p^2})$ is a three-tiered tower, and in Figure \ref{fig:CpqK4} (p. \pageref{fig:CpqK4}), we show what happens as the number of intermediate subgroups increases. We start with $C_{pq}$, where $p < q$ are prime, and the Klein four group $K_4$. Write $K_4 = \{1,a,b,c\}$, where $1$ is the identity and $ab = c$.

\begin{figure}
	\caption{}
	\label{fig:CpqK4}	
	\[
	\begin{array}{c}
	\begin{tikzpicture}[scale=1.1]
		\node(s) at (0,1.5) {$C_1$};
		\node(n) at (0,3.5) {$C_{pq}$};
		\node(e) at (-1,2.5) {$C_p$};
		\node(w) at (1,2.5) {$C_q$};
		
		\path[-]
		(s) edge node {} (w)
		(s) edge node {} (e)
		(w) edge node {} (n)
		(e) edge node {} (n)
		(s) edge node {} (n)
		;
		
		\node(A) at (4,0) {$\indt$};
		
		\node(B) at (3,1) {$\indp$};
		
		\node(C) at (5,1) {$\indq$};
		\node(D) at (4,2) {$\indpq$};
		
		\node(E') at (4,3) {$\indpqf$};
		
		\node(F) at (5,2.5) {$\indpqp$};
		\node(G) at (5,4) {$\indpqfp$};
		
		\node(H) at (3,2.5) {$\indpqq$};
		\node(I) at (3,4) {$\indpqfq$};
		\node(J) at (4,5) {$\inds$};
		
		\path[-]
		(A) edge node {} (B)
		
		(A) edge node {} (C)
		(B) edge node {} (D)
		(C) edge node {} (D)
		
		(D) edge node {} (E')
		
		(C) edge node {} (F)
		(F) edge node {} (G)
		(E') edge node {} (G)
		
		(B) edge node {} (H)
		(H) edge node {} (I)
		(E') edge node {} (I)
		(I) edge node {} (J)
		(G) edge node {} (J)
		;
		
		\node(qA) at (8,0) {$\indt$};
		\node(qB) at (8,1) {$\indq$};
		\node(qD) at (7,2) {$\indpq$};
		\node(qE) at (7,3) {$\indpqf$};
		\node(qF) at (8,2.5) {$\indpqp$};
		\node(qG) at (8,4) {$\indpqfp$};
		\node(qJ) at (8,5) {$\inds$};
		
		\path[-]
		(qA) edge node {} (qB)
		(qB) edge node {} (qD)
		(qD) edge node {} (qE)
		(qB) edge node {} (qF)
		(qE) edge node {} (qG)
		(qF) edge node {} (qG)
		(qG) edge node {} (qJ)
		;
		
		\node at (0,5.75) {$\b{Sub}(C_{pq})$};
		\node at (4,5.75) {$\b{Tr}(C_{pq})$};
		\node at (7.5,5.75) {$\b{Tr}(C_{pq})/\Sigma_2$};
		
		\draw (-1,5.5) -- (8.5,5.5) ;
	\end{tikzpicture}
	\\
	\begin{tikzpicture}[scale=1.1]
		\node(00) at (0,0) {$\kindt$};
		\node(-11) at (-2,1.5) {$\kinda$};
		\node(01) at (0,1.5) {$\kindb$};
		\node(11) at (2,1.5) {$\kindc$};
		\node(-12) at (-2,3) {$\kindab$};
		\node(02) at (0,3) {$\kindac$};
		\node(12) at (2,3) {$\kindbc$};
		\node(h2h) at (1,3.75) {$\kindabc$};
		\node(-13) at (-2,4.5) {$\kindkc$};
		\node(03) at (0,4.5) {$\kindkb$};
		\node(13) at (2,4.5) {$\kindka$};
		\node(h3h) at (1,5.25) {$\kindkt$};
		\node(-14) at (-2,6) {$\kindkct$};
		\node(04) at (0,6) {$\kindkbt$};
		\node(14) at (2,6) {$\kindkat$};
		\node(-15) at (-2,7.5) {$\kindkbc$};
		\node(05) at (0,7.5) {$\kindkac$};
		\node(15) at (2,7.5) {$\kindkab$};
		\node(06) at (0,9) {$\kinds$};
		
		\path[-]
		(00) edge node {} (-11)
		(00) edge node {} (01)
		(00) edge node {} (11)
		(-11) edge node {} (-12)
		(-11) edge node {} (02)
		(01) edge node {} (-12)
		(01) edge node {} (12)
		(11) edge node {} (02)
		(11) edge node {} (12)
		(-12) edge node {} (h2h)
		(02) edge node {} (h2h)
		(12) edge node {} (h2h)
		(-12) edge node {} (-13)
		(02) edge node {} (03)
		(12) edge node {} (13)
		(h2h) edge node {} (h3h)
		(-13) edge node {} (-14)
		(03) edge node {} (04)
		(13) edge node {} (14)
		(h3h) edge node {} (-14)
		(h3h) edge node {} (04)
		(h3h) edge node {} (14)
		(-14) edge node {} (-15)
		(-14) edge node {} (05)
		(04) edge node {} (-15)
		(04) edge node {} (15)
		(14) edge node {} (05)
		(14) edge node {} (15)
		(-15) edge node {} (06)
		(05) edge node {} (06)
		(15) edge node {} (06)
		;
		
		\node(nN) at (-4.5,5.5) {$K_4$};
		\node(nC) at (-4.5,4.5) {$\la b \ra$};
		\node(nW) at (-5.5,4.5) {$\la a \ra$};
		\node(nE) at (-3.5,4.5) {$\la c \ra$};
		\node(nS) at (-4.5,3.5) {$1$};
		
		\path[-]
		(nS) edge node {} (nW)
		(nS) edge node {} (nC) 
		(nS) edge node {} (nE)
		(nW) edge node {} (nN)
		(nC) edge node {} (nN)
		(nE) edge node {} (nN)
		(nS) edge [bend left] node {} (nN)
		;
		
		\node(q00) at (4.5,0) {$\kindt$};
		\node(q22) at (4.5,1.5) {$\kindc$};
		\node(q24) at (4.5,3) {$\kindbc$};
		\node(q15) at (3.5,3.75) {$\kindabc$};
		\node(q26) at (4.5,4.5) {$\kindka$};
		\node(q17) at (3.5,5.25) {$\kindkt$};
		\node(q28) at (4.5,6) {$\kindkat$};
		\node(q29) at (4.5,7.5) {$\kindkab$};
		\node(q010) at (4.5,9) {$\kinds$};
		
		\path[-]
		(q00) edge node {} (q22)
		(q22) edge node {} (q24)
		(q24) edge node {} (q15)
		(q15) edge node {} (q17)
		(q24) edge node {} (q26)
		(q17) edge node {} (q28)
		(q26) edge node {} (q28)
		(q28) edge node {} (q29)
		(q29) edge node {} (q010)
		;
		
		\node at (0,9.75) {$\b{Tr}(K_4)$};
		\node at (-4.5,9.75) {$\b{Sub}(K_4)$};
		\node at (4,9.75) {$\b{Sub}(K_4)/\Sigma_3$};
		
		\draw (-5.75,9.5) -- (5.25,9.5) ;
	\end{tikzpicture}
	\end{array}
	\]
\end{figure}

The pentagons that show up in $\b{Tr}(C_{pq})$ and $\b{Tr}(K_4)$ are copies of the pentagon that appears in $\b{Tr}(C_{p^2})$. More generally, suppose that $G$ is a finite abelian group with $n$ proper, nontrivial subgroups that are pairwise incomparable. Then $\b{Tr}(G)$ decomposes as a stacked pair of $n$-cubes with a layer of $n$ transfer systems between them. Thus, if $G = (C_p)^{\times 2}$ for a prime $p$, then there are $p+1$ intermediate subgroups and $2^{p+2} + p + 1$ transfer systems.

Now consider the quaternion group $Q_8 = \{ \pm 1 , \pm i , \pm j , \pm k \}$. Its subgroup lattice is obtained from the tower $\b{Sub}(C_{p^3})$ by widening the upper two links into a copy of $\b{Sub}(K_4)$. Accordingly, the lattice $\b{Tr}(Q_8)$ exhibits features of both $\b{Tr}(C_{p^3})$ and $\b{Tr}(K_4)$, but the mixing is nontrivial. There are $68$ total $Q_8$-transfer systems, and the group $\b{Out}(Q_8) \cong \Sigma_3$ acts on $\b{Tr}(Q_8)$ because all subgroups of $Q_8$ are normal. As a $\Sigma_3$-poset, $\b{Tr}(Q_8)$ is a sum $\Sigma_3/1 + 17 \cdot \Sigma_3/\la (12) \ra + 11 \cdot \Sigma_3/\Sigma_3$ of $29$ orbits, and we draw the quotient in Figure \ref{fig:Q8} (p. \pageref{fig:Q8}).

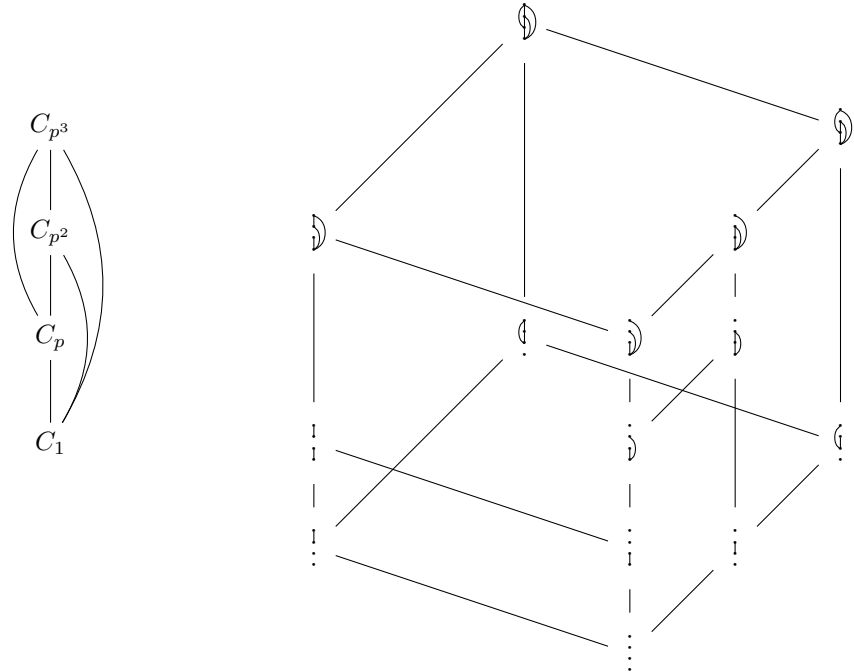
\begin{figure}
\caption{}
\label{fig:Q8}
\[
\begin{tikzpicture}[scale=1.25]
	\node(A) at (-6.25,1) {$\qeia$};
	
	\node(B) at (-5,2) {$\qeib$};
	\node(C) at (-3.75,3) {$\qeic$};
	\node(D) at (-2.5,4) {$\qeid$};
	\node(E) at (-1.25,5) {$\qeie$};
	\node(F) at (0,6) {$\qeif$};
	
	\node(G) at (-3.75,4) {$\qeig$};
	\node(H) at (-2.5,5) {$\qeih$};
	\node(I) at (-1.25,6) {$\qeii$};
	\node(J) at (0,7) {$\qeij$};
	
	\node(K) at (-2.5,6) {$\qeik$};
	\node(L) at (-1.25,7) {$\qeil$};
	\node(M) at (0,8) {$\qeim$};
	\node(N) at (0.625,9.5) {$\qein$};
	\node(O) at (0,12) {$\qeio$};
	\node(P) at (0,13) {$\qeip$};
	\node(Q) at (0,14) {$\qeiq$};
	
	\node(R) at (-1.875,8.5) {$\qeir$};
	\node(S) at (-0.625,9.5) {$\qeis$};
	\node(T) at (-0.625,10.5) {$\qeit$};
	
	\node(U) at (-1.875,7.5) {$\qeiu$};
	
	\node(V) at (-6.875,4.5) {$\qeiv$};
	\node(W) at (-6.875,5.5) {$\qeiw$};
	
	\node(X) at (-6.25,2) {$\qeix$};
	\node(Y) at (-6.25,3) {$\qeiy$};
	\node(Z) at (-5.625,4.5) {$\qeiz$};
	\node(ZA) at (-6.25,7) {$\qeiza$};
	\node(ZB) at (-6.25,8) {$\qeizb$};
	\node(ZC) at (-6.25,9) {$\qeizc$};
	
	\path[-]
	(A) edge node {} (B)
	
	(B) edge node {} (C)
	(C) edge node {} (D)
	(D) edge node {} (E)
	(E) edge node {} (F)
	
	(C) edge node {} (G)
	(D) edge node {} (H)
	(E) edge node {} (I)
	(F) edge node {} (J)
	
	(G) edge node {} (H)
	(H) edge node {} (I)
	(I) edge node {} (J)
	
	(H) edge node {} (K)
	(I) edge node {} (L)
	(J) edge node {} (M)
	
	(K) edge node {} (L)
	(L) edge node {} (M)
	(M) edge node {} (N)
	(N) edge node {} (O)
	(O) edge node {} (P)
	(P) edge node {} (Q)
	
	(L) edge node {} (R)
	(M) edge node {} (S)
	(T) edge node {} (O)
	
	(R) edge node {} (S)
	(S) edge node {} (T)
	
	(K) edge node {} (U)
	(U) edge node {} (N)
	
	(V) edge node {} (R)
	(W) edge node {} (T)
	
	(V) edge node {} (W)
	
	(A) edge node {} (X)
	(X) edge node {} (Y)
	(Y) edge node {} (Z)
	(Z) edge node {} (ZA)
	(ZA) edge node {} (ZB)
	(ZB) edge node {} (ZC)
	
	(X) edge node {} (G)
	(Y) edge node {} (K)
	(Z) edge node {} (U)
	(ZA) edge node {} (O)
	(ZB) edge node {} (P)
	(ZC) edge node {} (Q)
	
	(Y) edge node {} (V)
	(W) edge node {} (ZA)
	;
	
	\node(n1) at (-6.25,11) {$1$};
	\node(n-1) at (-6.25,12) {$\la -1 \ra$};
	\node(ni) at (-7.5,13) {$\la i \ra$};
	\node(nj) at (-6.25,13) {$\la j \ra$};
	\node(nk) at (-5,13) {$\la k \ra$};
	\node(nq) at (-6.25,14) {$Q_8$};
	
	\path[-]
	(n1) edge node {} (n-1)
	(n-1) edge node {} (ni)
	(n-1) edge node {} (nj)
	(n-1) edge node {} (nk)
	(ni) edge node {} (nq)
	(nj) edge node {} (nq)
	(nk) edge node {} (nq)
	(n1) edge [bend left] node {} (ni)
	(n1) edge [bend right] node {} (nj)
	(n1) edge [bend right] node {} (nk)
	(n-1) edge [bend left] node {} (nq)
	;
	\draw plot [smooth, tension=1.5] coordinates {(-6,11) (-4.25,12.5) (-6,14)};
	
	\node at (-6.25,14.75) {$\b{Sub}(Q_8)$};
	\node at (0,14.75) {$\b{Tr}(Q_8)/\Sigma_3$};
	\draw (-8,14.5) -- (1.75,14.5) ;
\end{tikzpicture}
\]
\end{figure}

There is an evident copy of $\b{Tr}(K_4)/\Sigma_3$ on the left edge of $\b{Tr}(Q_8)/\Sigma_3$. As we move to the right, partially grown copies sprout up from the bottom, and we end with another fully grown copy of $\b{Tr}(K_4)/\Sigma_3$ on the right. There is also a copy of the associahedron $\b{Tr}(C_{p^3})$ in $\b{Tr}(Q_8)$, spanned by the $Q_8$-transfer systems below.
	\[
	\qeia \, \qeib \, \qeie \, \qeiv \, \qeir \, \qeif \, \qeis \, \qeiw \, \qeit \, \qeiz \, \qeiu \, \qein \, \qeizc \, \qeiq
	\]

So far, we have only studied groups for which every subgroup is normal. We consider $G = \Sigma_3$ in Figure \ref{fig:S3} (p. \pageref{fig:S3}) for a change.

\begin{figure}
	\caption{}
	\label{fig:S3}
	\begin{tikzpicture}[scale=1.25]
		\node(A) at (0,0) {$\dsia$};
		\node(B) at (0,1.25) {$\dsib$};
		\node(C') at (1.5,1.25) {$\dsic$};
		\node(D) at (1.5,2.5) {$\dsid$};
		\node(E') at (1.5,3.75) {$\dsie$};
		\node(F) at (1.5,5) {$\dsif$};
		\node(G) at (-1.5,2.5) {$\dsig$};
		\node(H) at (0,3.75) {$\dsih$};
 		\node(I) at (0,6) {$\dsii$};
		\path[-]
		(A) edge node {} (B)
		
		(A) edge node {} (C')
		(B) edge node {} (D)
		(C') edge node {} (D)
		
		(D) edge node {} (E')
		(E') edge node {} (F)
		(B) edge node {} (G)
		(D) edge node {} (H)
		(G) edge node {} (H)
		(H) edge node {} (I)
		(F) edge node {} (I)
		;
		
		\node(ns) at (-5.5,1.75) {$1$};
		\node(nw) at (-7.5,3) {$\la (12) \ra$};
		\node(nc) at (-5.5,3) {$\la (13) \ra$};
		\node(ne) at (-3.5,3) {$\la (23) \ra$};
		\node(nn) at (-5.5,4.25) {$\Sigma_3$};
		\node(nc') at (-4.6,3.25) {$\la (123) \ra$};
		
		\path[-]
		(ns) edge node {} (nw)
		(ns) edge node {} (nc)
		(ns) edge node {} (ne)
		(ns) edge node {} (nc')
		(ns) edge [bend left=40] node {} (nn)
		(nw) edge node {} (nn)
		(nc) edge node {} (nn)
		(ne) edge node {} (nn)
		(nc') edge node {} (nn)
		;
		
		\node at (0,7) {$\b{Tr}(\Sigma_3)$};
		\node at (-5.5,7) {$\b{Sub}(\Sigma_3)$};
		
		\draw (-8,6.75) -- (2,6.75) ;
	\end{tikzpicture}
\end{figure}

The group $\Sigma_3$ has four proper, nontrivial subgroups, generated by the transpositions and a three cycle. The former copies of $C_2$ are conjugate, and the latter copy of $C_3$ is normal. In some respects, this allows us to treat all copies of $C_2$ as the same subgroup. For example, if $1 \to \la (12) \ra$, then $1 \to \la \tau \ra$ for every transposition $\tau$, and dually if $\la (12) \ra \to \Sigma_3$. However, we must remember that $\la (12) \ra$, $\la (13) \ra$, and $\la (23) \ra$ are distinct subgroups. The $\Sigma_3$-transfer system generated by $\la (12) \ra \to \Sigma_3$ contains $1 \to \Sigma_3$ because it contains $\la (23) \ra \to \Sigma_3$ by conjugating, and $1 = \la (12) \ra \cap \la (23) \ra \to \la (12) \ra$ by restricting (cf. Example \ref{ex:genS3ts}). This is in sharp contrast to the $C_{pq}$-transfer system generated by $C_p \to C_{pq}$ or the $(C_3)^{\times 2}$-transfer system generated by a single relation of the form $C_3 \to (C_3)^{\times 2}$.

More generally, if $G = D_{2p}$ for a prime $p > 2$, then the set of proper nontrivial subgroups of $G$ consists of $p$ conjugate copies of $C_2$, and one normal copy of $C_p$. One finds that $\b{Tr}(D_{2p}) \cong \b{Tr}(\Sigma_3)$, by the same count.

\section{Steiner operads}\label{sec:Steiner}

In this section, we continue Blumberg and Hill's analysis of equivariant Steiner operads. We identify the $G$-transfer systems that arise from Steiner operads in general (Theorem \ref{thm:discind}), and then we specialize to finite abelian groups (Theorem \ref{thm:discfinab}). In the latter case, we show how to construct a minimal universe $U$ such that $\c{K}(U)$ parametrizes a specified transfer map (Proposition \ref{prop:TadmU}).

\subsection{General finite groups} Suppose that $U$ is a $G$-universe and consider the Steiner operad $\c{K}(U)$. If $K \subset H \subset G$ are subgroups, then by \cite[Theorem 4.19]{BH}, 
	\[
	K \to_{\c{K}(U)} H \quad\t{if and only if}\quad H/K \t{ $H$-embeds into $\t{res}^G_H U$} .
	\]
We begin our analysis by showing $\to \,\, = \,\, \to_{\c{K}(U)}$ is completely determined by transfer relations $K \to G$ such that the target is all of $G$.

Identify a binary relation $R$ on a set $X$ with the set $\{ (x,y) \in X^{\times 2} \, | \, x R y\}$ of all $R$-related pairs. Thus $x R y$ means $(x,y) \in R$, and $R$ refines $S$ if and only if $R \subset S$. If $R$ is any binary relation on $\b{Sub}(G)$ that refines inclusion, then there is minimum transfer system $\to \,\, = \la R \ra$ that contains $R$. Abstractly, $\to$ is the intersection of all transfer systems that contain $R$, but we give an explicit construction in Appendix \ref{sec:genindsys}. We call $\la R \ra$ the transfer system generated by $R$.

\begin{lem}\label{lem:discGorb} Suppose that $U$ is a $G$-universe, and let $\to \,\, = \,\, \to_{\c{K}(U)}$. Then $\to$ is generated by $\{ (K,G) \, | \, K \subset G \t{ and } K \to G \}$.
\end{lem}

\begin{proof} Let $\rightsquigarrow$ be the $G$-transfer system generated by $\{ (K,G) \, | \, K \subset G \t{ and } K \to G \}$. Then $\rightsquigarrow$ refines $\to$ by definition. We must establish the other refinement.

Suppose $K \to H$, choose an $H$-embedding $\vp : H/K \hookrightarrow \t{res}^G_H U$, and let $x = \vp(eK) \in U$. Then $K = H_x = G_x \cap H$, and there is a $G$-embedding $G/G_x \hookrightarrow U$. Therefore $G_x \to G$, which implies $G_x \rightsquigarrow G$, and restricting along $H \subset G$ shows that $K = G_x \cap H \rightsquigarrow H$. Therefore $\to$ refines $\rightsquigarrow$.
\end{proof}

\begin{ex}\label{ex:Ggentr} There are plenty of transfer systems $\to$ such that the refinement $ \la (K,G) \, | \, K \subset G \t{ and } K \to G \ra  \leq \, \to$ is an equality, and plenty such that it is not. If $G = K_4$, then we have an equality for the $\Sigma_3 = \b{Out}(K_4)$-orbits of
	\[
	\kindt \, \kindkc \, \kindkt \, \kindkct \, \kindkbc \, \kinds
	\]
and an inequality for the orbits of
	\[
	\kinda \, \kindab \, \kindabc .
	\]
\end{ex}

Lemma \ref{lem:discGorb} and Proposition \ref{prop:BHcond} imply that a large class of transfer systems are not realized by Steiner or linear isometries operads.

\begin{thm}\label{thm:notSteinerorlinisom} Suppose $G$ is a finite group, $K \subset G$ is a normal subgroup, and $K \subsetneq L \subsetneq H \subsetneq G$ is a chain in $\b{Sub}(G)$. Then the $G$-transfer system $\la (K,H) \ra$ generated by $(K,H)$ is not realized by a $G$-Steiner or a $G$-linear isometries operad.
\end{thm}

\begin{proof} Let $H = H_1 , \dots , H_n$ be the conjugates of $H$ in $G$ and let $\to \,\, = \la (K,H) \ra = \la (K,H_i) \, | \, 1 \leq i \leq n \ra$. Then $\to \,\, = \{ (M,M) \, | \, M \subset G \} \cup \bigcup_{i=1}^n \{ (M \cap K,M) \, | \, M \subset H_i \}$ by Proposition \ref{prop:multitsKnorm}.

If $J \to G$, then $J = G$, and therefore $ \la (J,G) \, | \, J \to G \ra = \Delta \b{Sub}(G) < \,\, \to$. Lemma \ref{lem:discGorb} implies that $\to$ is not realized by any Steiner operad.

On the other hand, $L \not\to H$ because $L \neq H$ and $M \cap K \subset K \subsetneq L$ for all $M \subset G$. Hence $\to$ is not saturated, and Proposition \ref{prop:BHcond} implies that $\to$ is not realized by any linear isometries operad.
\end{proof}

\begin{ex}The $C_{p^3}$-transfer system $\cpppc$ and the $Q_8$-transfer system $\qeic$ are not realized by any Steiner or linear isometries operads.
\end{ex}

We can hone our description of $\to_{\c{K}(U)}$ further. For any $G$-representation $V$, let
	\[
	Orb(V) = \{ (K,G) \, | \, K \subsetneq G \t{ and $G/K$ $G$-embeds into $V$}\}.
	\]

\begin{prop}\label{prop:discind} Let $U$ be a $G$-universe, and suppose that $U \cong \bigoplus_{i \in I} V_i$ for some $G$-representations $V_i$, indexed over a possibly infinite set $I$. Then $\to_{\c{K}(U)}$ is generated by $\bigcup_{i \in I} Orb(V_i)$.
\end{prop}

\begin{proof} Let $\to \,\, = \,\, \to_{\c{K}(U)}$ and let $\rightsquigarrow \,\, = \la \bigcup_{i \in I} Orb(V_i) \ra$. If $(K,G) \in Orb(V_i)$ for some $i \in I$, then there is a composite $G$-embedding $G/K \hookrightarrow V_i \hookrightarrow U$, and therefore $(K , G) \in \,\, \to$. Therefore $\rightsquigarrow$ refines $\to$.

Conversely, suppose $K \to G$ and choose a $G$-embedding $\vp : G/K \hookrightarrow \bigoplus_{i \in I} V_i$. Since $G/K$ is finite, the map $\vp$ factors through some finite sum $V_{i_1} \oplus \cdots \oplus V_{i_n}$. Let $(x_1,\dots,x_n) = \vp(eK) \in V_{i_1} \oplus \cdots \oplus V_{i_n}$. Then $K = G_{(x_1,\dots,x_n)} = G_{x_1} \cap \cdots \cap G_{x_n}$. Since we have $G$-embeddings $G/G_{x_i} \hookrightarrow V_i$, it follows that $G_{x_i} \rightsquigarrow G$, and hence $K = G_{x_1} \cap \dots \cap G_{x_n} \rightsquigarrow G$ by Lemma \ref{lem:transysint}. Therefore $\la (K,G) \, | \, K \to G \ra \subset \,\, \rightsquigarrow$, and the left hand side equals $\to$ by Lemma \ref{lem:discGorb}. This proves that $\to$ refines $\rightsquigarrow$.
\end{proof}

In particular, we may calculate $\to_{\c{K}(U)}$ in terms of the irreducible subrepresentations $V \subset U$. The next result follows easily.

\begin{thm}\label{thm:discind} Suppose $G$ is a finite group and $\to$ is a $G$-transfer system. The following are equivalent:
	\begin{enumerate}
		\item{}There is a $G$-universe $U$ such that $\to \,\, = \,\, \to_{\c{K}(U)}$.
		\item{}There is an integer $n \geq 0$ and nontrivial, irreducible real $G$-representations $V_1 , V_2 , \dots , V_n$ such that $\to \,\, = \la \bigcup_{i=1}^n Orb(V_i) \ra $.
	\end{enumerate}
When $n=0$ in $(2)$, we understand $\to$ to be the minimum transfer system.
\end{thm}

\begin{proof} If $\to \,\, = \la \bigcup_{i=1}^n  Orb(V_i) \ra$ for some sequence of nontrivial, irreducible real $G$-representations $V_i$, then $\to \,\, = \,\, \to_{\c{K}(U)}$ for $U = [\bb{R} \oplus \bigoplus_{i=1}^n V_i]^\infty$, by Proposition \ref{prop:discind}. The converse is similar.
\end{proof}

Thus, we can identify the image of $A_{\c{K}} : \b{Uni}(G) \to \b{Tr}(G)$ by computing orbit decompositions of all irreducible real $G$-representations, and then enumerating the transfer systems generated by combinations of these data. We illustrate by example.

\begin{ex}\label{ex:K4Steiner} Let $G = K_4$ once more, and keep notation as in \S\ref{subsec:exts}. We shall further winnow down the candidates found in Example \ref{ex:Ggentr}. There are three nontrivial, irreducible real $K_4$-representations. We have a sign representation $\sigma_a : K_4 \twoheadrightarrow K_4/\la a \ra \stackrel{\sigma}{\to} O(1)$, which satisfies $Orb(\sigma_a) = \{ (\la a \ra , K_4) \}$, and similarly for $b,c \in K_4$. Thus, there are eight $K_4$-universes, which form four orbits under the $\Sigma_3$-action. We give orbit representatives and their transfer systems below.
	\[
	\begin{array}{|c|c|}
	\hline
	U	&	\to_{\c{K}(U)}	\\
	\hline
	\bb{R}^\infty	&	\kindt	\\
	(\bb{R} \oplus \sigma_c)^\infty	&	\kindkc	\\
	(\bb{R} \oplus \sigma_b \oplus \sigma_c)^\infty	&	\kindkbc	\\
	(\bb{R} \oplus \sigma_a \oplus \sigma_b \oplus \sigma_c)^\infty	&	\kinds	\\
	\hline
	\end{array}
	\]
\end{ex}

\begin{ex}\label{ex:Q8Steiner} If $G = Q_8$, then there are four nontrivial, irreducible real representations. There is a sign representation $\sigma_i : Q_8 \twoheadrightarrow Q_8/ \la i \ra \stackrel{\sigma}{\to} O(1)$, and analogous representations $\sigma_j$ and $\sigma_k$ for $j,k \in Q_8$. There is also a four-dimensional representation $\bb{H}$, obtained by letting $Q_8 \subset \bb{H}$ act on the quaternions by left multiplication. We have $Orb(\bb{H}) = \{ (1, Q_8) \}$, $Orb(\sigma_i) = \{ ( \la i \ra , Q_8) \}$, and similarly for $j$ and $k$. Thus, there are sixteen $Q_8$-universes, which form eight $\Sigma_3$-orbits. We give orbit representatives and their transfer systems below.
	\[
	\begin{array}{|c|c|c|c|}
	\hline
	U	&	\to_{\c{K}(U)}	&	U	&	\to_{\c{K}(U)}	\\
	\hline
	\bb{R}^\infty	&	\qeia	&	(\bb{R} \oplus \bb{H})^\infty	&	\qeif	\\
	(\bb{R} \oplus \sigma_k)^\infty	&	\qeiz	&	(\bb{R} \oplus \sigma_k \oplus \bb{H})^\infty	&	\qein	\\
	(\bb{R} \oplus \sigma_j \oplus \sigma_k)^\infty	&	\qeizb	& (\bb{R} \oplus \sigma_j \oplus \sigma_k \oplus \bb{H})^\infty	&	\qeip	\\
	(\bb{R} \oplus \sigma_i \oplus \sigma_j \oplus \sigma_k)^\infty	&	\qeizc	&	(\bb{R} \oplus \sigma_i \oplus \sigma_j \oplus \sigma_k \oplus \bb{H})^\infty	&	\qeiq	\\
	\hline
	\end{array}
	\]
\end{ex}

\begin{ex}\label{ex:Sigma3Steiner} If $G = \Sigma_3$, then the nontrivial, irreducible real representations are the sign representation $\sigma : \Sigma_3 \twoheadrightarrow \Sigma_3 / \la (123) \ra \stackrel{\sigma}{\to} O(1)$ and the representation $\Delta : \Sigma_3 \to O(2)$ of $\Sigma_3$ as the symmetries of a triangle. We have $Orb(\sigma) = \{ ( \la (123) \ra , \Sigma_3 ) \}$ and $Orb(\Delta) = \{ (\la (12) \ra ,\Sigma_3) , ( \la (13) \ra , \Sigma_3) , ( \la (23) \ra , \Sigma_3) , (1 , \Sigma_3) \}$, and hence the transfer systems for $\Sigma_3$-Steiner operads are
	\[
	\begin{array}{|c|c|}
	\hline
	U	&	\to_{\c{K}(U)}	\\
	\hline
	\bb{R}^\infty	&	\dsia	\\
	(\bb{R} \oplus \sigma)^\infty	&	\dsig	\\
	(\bb{R} \oplus \Delta)^\infty	&	\dsif	\\
	(\bb{R} \oplus \sigma \oplus \Delta)^\infty	&	\dsii	\\
	\hline
	\end{array}
	\]
\end{ex}

\subsection{Finite abelian groups} Theorem \ref{thm:discind} gives a reasonable procedure for computing the image of $A_{\c{K}} : \b{Uni} \to \b{Tr}$, but it is another matter to find a clean description of $\t{im}(A_{\c{K}})$ purely in terms of the algebra of transfer systems. We do not believe there is a uniform solution for all finite groups. However, there is a uniform solution if we restrict to finite abelian groups.

\begin{defn}\label{defn:cyclicorb} Suppose that $G$ is a finite group and that $H \subset G$ is a subgroup of $G$. We say that $H$ is \emph{$G$-cocyclic} if $H$ is a normal subgroup of $G$ and the quotient group $G/H$ is cyclic.
\end{defn}

\begin{thm}\label{thm:discfinab} Suppose $G$ is a finite abelian group and $\to$ is a $G$-transfer system. Then $\to$ corresponds to a $G$-Steiner operad if and only if there is an integer $n \geq 0$ and $G$-cocyclic subgroups $H_1,\dots,H_n \subset G$ such that $\to \,\, = \la (H_i,G) \, | \, 1 \leq i \leq n \ra$.
\end{thm}

\begin{proof} There are two kinds of irreducible real $G$-representations. There are one-dimensional representations, where each $g \in G$ acts as multiplication by $+1$ or $-1$, and there are two-dimensional representations, where each $g \in G$ acts as a rotation by $\theta(g) \in [0,2\pi)$, and at least one $\theta(g)$ is not $0$ or $\pi$. In the former case, we obtain a map $V : G \to O(1) \cong C_2$, and in the latter case we obtain a map $V : G \to C_{|G|} \hookrightarrow SO(2)$, where $C_{|G|}$ embeds in $SO(2)$ as the rotations by multiples of $2\pi/|G|$. Therefore $G/\t{ker} V$ embeds into $C_2$ or $C_{\abs{G}}$, and therefore $\t{ker} V$ is $G$-cocyclic. Now consider the orbit decomposition of $V$. The actions of $C_2$ on $\bb{R}$ and $C_{|G|}$ on $\bb{R}^2$ are free away from the origin. Pulling back to $G$, we see that $G_0 = G$, $G_{x} = \t{ker} V$ for every $x \neq 0$, and therefore $Orb(V) = \{ (\t{ker} V ,G) \}$.

Now suppose $\to$ is a $G$-transfer system. If $\to \,\, = \,\, \to_{\c{K}(U)}$ for some $G$-universe $U \cong [\bb{R} \oplus \bigoplus_{i=1}^n V_i]^\infty$, where each of the the representations $V_i$ is nontrivial and irreducible, then $\to \,\, = \la (\t{ker} V_i , G) \, | \, 1 \leq i \leq n \ra$ by Proposition \ref{prop:discind}. As noted above, each of the subgroups $\t{ker} V_i$ is $G$-cocyclic.

Conversely, suppose $\to \,\, = \la (H_i,G) \, | \, 1 \leq i \leq n \ra$ for some $G$-cocyclic subgroups $H_1,\dots,H_n \subset G$. For each $i$, choose an embedding $G/H_i \hookrightarrow O(2)$ of $G/H_i$ as the rotations by multiples of $2\pi/ \abs{G : H_i}$, and let $\lambda_i : G \twoheadrightarrow G/H_i \hookrightarrow O(2)$ be the pullback to $G$. Then $Orb(\lambda_i) = \{ (H_i,G) \}$. Thus, if $U = [ \bb{R} \oplus \bigoplus_{i=1}^n \lambda_i]^\infty$, then $\to_{\c{K}(U)} \,\, = \la (H_i,G) \, | \, 1 \leq i \leq n \ra = \,\, \to$ by Proposition \ref{prop:discind} again.
\end{proof}

This simplifies further for finite cyclic groups.

\begin{cor}\label{cor:Gdisccyc} Let $n > 0$ be a natural number. A $C_n$-transfer system $\to$ corresponds to a Steiner operad if and only if $\to \,\, = \la (H_i,C_n) \, | \, 1 \leq i \leq m \ra$ for some subgroups $H_1,\dots,H_m \subset C_n$.
\end{cor}

\begin{ex}\label{ex:Cp3CpqSteiner} The $C_{p^3}$-transfer systems corresponding to Steiner operads are
	\[
	\cpppa \,\,\, \cpppj \,\,\, \cppph \,\,\, \cpppf \,\,\, \cpppk \,\,\, \cpppm \,\,\, \cpppi \,\,\, \cpppn .
	\]
The $C_{pq}$-transfer systems corresponding to Steiner operads are
	\[
	\indt \, \indpqp \, \indpqq \, \indpqf \, \indpqfp \, \indpqfq \, \inds .
	\]
\end{ex}

Requiring $H_i \subset G$ to be $G$-cocyclic is a nontrivial constraint. The next example generalizes Example \ref{ex:K4Steiner}.

\begin{ex} Suppose that $G = (C_p)^{\times n}$ for a prime $p$ and integer $n > 0$. A proper subgroup $H \subset G$ is $G$-cocyclic if and only if it is a codimension $1$ subspace of $\bb{F}_p^n$ under the identification $(C_p)^{\times n} \cong (\bb{F}_p^n,+)$. Therefore a $(C_p)^{\times n}$-transfer system $\to$ arises from a Steiner operad if and only if it is generated by relations of the form $(C_p)^{\times n-1} \to (C_p)^{\times n}$, for some embedded copies of $(C_p)^{\times n-1}$ in $(C_p)^{\times n}$.
\end{ex}

We can also combine Theorem \ref{thm:discfinab} with Proposition \ref{prop:BHcond} to exclude transfer systems from the images of $A_{\c{K}}$ and $A_{\c{L}}$. The next result generalizes the fact that no $K_4$-Steiner or linear isometries operad realizes the transfer system $\kindkt$.

\begin{thm}\label{thm:notLorD} Suppose that $G$ is a non-cyclic finite abelian group. Then the $G$-transfer system $\to \,\, = \la (0,G) \ra$ generated by $(0,G)$ alone is not realized by a $G$-Steiner or a $G$-linear isometries operad.
\end{thm}

\begin{proof} We have $\to \,\, = \{ (M,M) \, | \, M \subset G \} \cup \{ (0,M) \, | \, M \subset G\}$, by Corollary \ref{cor:indgen1orb} or by inspection. Thus, if $H \to G$, then either $H = 0$ or $H = G$.

On the other hand, suppose $U$ is a $G$-universe such that $0 \to_{\c{K}(U)} G$. Then by Theorem \ref{thm:discfinab}, $\to_{\c{K}(U)} \,\, = \la (H_i,G) \, | \, 1 \leq i \leq n \ra$ for some $G$-cocyclic subgroups $H_i \subset G$. Since $\to_{\c{K}(U)}$ is nontrivial, some $H_i$ must be a proper subgroup of $G$, and since $G$ is noncyclic, the subgroup $H_i$ must also be nontrivial. Thus $0 \subsetneq H_i \subsetneq G$ and $H_i \to_{\c{K}(U)} G$. It follows that $\to_{\c{K}(U)} \,\, \neq \,\, \to$ for every $G$-universe $U$.

Finally, suppose $U$ is a $G$-universe such that $0 \to_{\c{L}(U)} G$. Then $H \to_{\c{L}(U)} G$ for every $H \subset G$ because $\to_{\c{L}(U)}$ is saturated. Since $G$ is non-cyclic, any non-identity element $x \in G$ generates a proper, nontrivial subgroup $0 \subsetneq \la x \ra \subsetneq G$ such that $\la x \ra \to_{\c{L}(U)} G$. Therefore $\to_{\c{L}(U)} \,\, \neq \,\, \to$ for every $G$-universe $U$.
\end{proof}

\subsection{Parametrizing a transfer map} The previous two sections explain how to compute the transfers parametrized by $\c{K}(U)$, for any given universe $U$. In this section, we turn the problem around. When $G$ is finite abelian, we construct minimal universes $U$ such that $\c{K}(U)$ parametrizes a given transfer $K \to H$.

For any finite abelian group $G$ and proper, $G$-cocyclic subgroup $H \subsetneq G$, let $\lambda_H$ be a two-dimensional real $G$-representation $G \twoheadrightarrow G/H \cong C_{n} \hookrightarrow SO(2)$, obtained by choosing an isomorphism $G/H \cong C_n$ for some $n \geq 2$, and then embedding $C_n$ into $SO(2)$ as the rotations by multiples of $2\pi/n$.

\begin{lem}\label{lem:lambdaHiso} Suppose that $V \subset \lambda_H$ is an irreducible $G$-representation. Then $G_x = H$ for every nonzero $x \in V$.
\end{lem}

\begin{proof} Every nonzero $x \in \lambda_H$ has $G_x = H$,  as explained in the proof of Theorem \ref{thm:discfinab}. This proves the lemma when $\lambda_H$ is irreducible. If $\lambda_H$ is reducible, then it has an invariant one-dimensional subspace. Therefore $G/H \cong C_2$ and $\lambda_H \cong \sigma_H \oplus \sigma_H$, where $\sigma_H$ is $G \twoheadrightarrow G/H \cong C_2 \stackrel{\sigma}{\hookrightarrow} O(1)$. In this case, $V \subset \lambda_H$ is isomorphic to $\sigma_H$, and $G_x = H$ for every nonzero $x \in \sigma_H$.
\end{proof}

We use the representations $\lambda_H$ to construct the desired universes.

\begin{prop}\label{prop:TadmU} Suppose that $G$ is a finite abelian group and that $K \subsetneq H \subset G$ are subgroups. Choose distinct, proper, $G$-cocyclic subgroups $H_1 , \dots , H_m \subsetneq G$ such that $H \cap H_1 \cap \cdots \cap H_m = K$, and let
	\[
	U = [ \bb{R} \oplus \lambda_{H_1} \oplus \cdots \oplus \lambda_{H_m} ]^\infty.
	\]
Then $K \to_{\c{K}(U)} H$, and $[U] \in \b{Uni}$ is minimal with this property if and only if $H \cap H_1 \cap \cdots \cap H_{i-1} \cap H_{i+1} \cap \cdots \cap H_m \supsetneq K$ for every $i = 1 , \dots , m$.
\end{prop}

\begin{proof} Let $U$ be as above. Then $\to_{\c{K}(U)} \,\, = \la (H_i,G) \, | \, 1 \leq i \leq m \ra$ by Lemma \ref{lem:lambdaHiso} and Proposition \ref{prop:discind}, and Proposition \ref{prop:indGorbab} implies $K \to_{\c{K}(U)} H$ because we have assumed $H \cap H_1 \cap \cdots \cap H_m = K$.

Next, let $\lambda_i = \lambda_{H_i}$ and $U_i = [\bb{R} \oplus \bigoplus_{j \neq i} \lambda_j ]^\infty$. Then there is no $G$-embedding $\lambda_i \hookrightarrow U_i$. For if there were an embedding, then an irreducible subrepresentation $V \subset \lambda_i$ would embed in $\bb{R}$ or $\lambda_j$ for some $j \neq i$, but Lemma \ref{lem:lambdaHiso} implies this is impossible because the subgroups $G,H_1 , \dots , H_m$ are all distinct. Therefore $U_i$ is a proper subuniverse of $U$, and it is maximal proper because each $\lambda_i$ is either irreducible, or splits as $\lambda_{H_i} \cong \sigma_{H_i} \oplus \sigma_{H_i}$.

We now consider the minimality of $U$. First, note that $\to_{\c{K}(U_i)} = \la (H_j,G) \, | \, j \neq i \ra$. Thus, if $H \cap H_1 \cap \cdots \cap H_{i-1} \cap H_{i+1} \cap \cdots \cap H_m = K$ for some $i$, then $K \to_{\c{K}(U_i)} H$ by Proposition \ref{prop:indGorbab}. In this case, $[U]$ is not minimal among the classes $[U'] \in \b{Uni}$ such that $K \to_{\c{K}(U')} H$.

Now suppose that $H \cap H_1 \cap \cdots \cap H_{i-1} \cap H_{i+1} \cap \cdots \cap H_m \supsetneq K$ for each $i = 1, \dots, m$. Then $K \not\to_{\c{K}(U_i)} H$ for every $i$, by Proposition \ref{prop:indGorbab}. Therefore $[U]$ is minimal, because any proper subuniverse $U' \hookrightarrow U$ of $U$ $G$-embeds into one of the $U_i$, and hence $K \not\to_{\c{K}(U')} H$ as well.
\end{proof}

\begin{ex} We indicate how this works for $G = K_4$. Keep notation as in Example \ref{ex:K4Steiner}. The proper, $K_4$-cocyclic subgroups are $\la a \ra$, $\la b \ra$, and $\la c \ra$, and the corresponding $\lambda$ representations are $\lambda_{\la a \ra} \cong \sigma_a \oplus \sigma_a$, $\lambda_{\la b \ra} \cong \sigma_b \oplus \sigma_b$, and $\lambda_{\la c \ra} \cong \sigma_c \oplus \sigma_c$.

Suppose we wish to parametrize $\la a \ra \to K_4$ with a Steiner operad. Following Proposition \ref{prop:TadmU}, we need a set of $K_4$-cocyclic subgroups that intersect to $\la a \ra$. The singleton $\{ \la a \ra \}$ works, and $U = [ \bb{R} \oplus \lambda_{\la a \ra} ]^\infty \cong [ \bb{R} \oplus \sigma_a ]^{\infty}$ is a minimal universe such that $\la a \ra \to_{\c{K}(U)} K_4$.

Now suppose we wish to parametrize $1 \to \la a \ra$. We need $K_4$-cocyclic subgroups $H_1,H_2,\dots$ such that $\la a \ra \cap H_1 \cap H_2 \cap \cdots = 1$. This holds as long as we include one of $\la b \ra$ or $\la c \ra$. Therefore $1 \to_{\c{K}(U)} \la a \ra$ holds whenever $\sigma_b$ or $\sigma_c$ embed in $U$, and the universes $U = [\bb{R} \oplus \sigma_b]^\infty$ and $[\bb{R} \oplus \sigma_c ]^\infty$ are minimal for this transfer.

Finally, suppose we wish to parametrize $1 \to K_4$. Since any two of $\la a \ra$, $\la b \ra$, and $\la c \ra$ intersect trivially, we have $1 \to_{\c{K}(U)} K_4$ for any $U = [ \bb{R} \oplus \sigma_x \oplus \sigma_y ]^\infty$ such that $x \neq y$, or for $U = [ \bb{R} \oplus \sigma_a \oplus \sigma_b \oplus \sigma_c ]^\infty$. The former are minimal.
\end{ex}

In general, if $0 \to_{\c{K}(U)} G$ for a non-cyclic finite abelian group $G$, then we should expect $U$ to be reasonably large. The next example illustrates.

\begin{ex} Let $G = (C_p)^{\times n} \cong (\bb{F}_p^n,+)$ and suppose $V \subsetneq \bb{F}_p^n$ is a proper subspace. Choose lines $\ell_1, \dots , \ell_m$ such that $\bb{F}_p^n = V \oplus \ell_1 \oplus \cdots \oplus \ell_m$, let $W_i = V \oplus \ell_1 \oplus \cdots \oplus \ell_{i-1} \oplus \ell_{i+1} \oplus \cdots \oplus \ell_m$, and let $\lambda_i = \lambda_{W_i}$ be the pullback of the representation $\lambda : C_p \hookrightarrow SO(2)$ along the quotient $\pi_i : \bb{F}_p^n \twoheadrightarrow \bb{F}_p^n/W_i \cong C_p$. Then $U = [\bb{R} \oplus \lambda_1 \oplus \cdots \oplus \lambda_m]^\infty$ is a minimal universe such that $V \to_{\c{K}(U)} \bb{F}_p^n$.
\end{ex}

Thus, Steiner operads have difficulty parametrizing a transfer $K \to G$ when $K$ is low in $\b{Sub}(G)$, at least when $G$ is finite abelian and has many cyclic summands.

\section{Linear isometries operads}\label{sec:linisom}

In this section, we continue Blumberg and Hill's analysis of equivariant linear isometries operads. We begin with some generalities, and then we restrict our ambient group $G$ to finite cyclic groups. When the order of $G$ is a prime power or a product of two distinct primes, we obtain strong results about the image of $A_{\c{L}} : \b{Uni}(G) \to \b{Tr}(G)$ (Theorems \ref{thm:LUCpn} and \ref{thm:Cpqlinisom}).

\subsection{General results} Suppose that $U$ is a $G$-universe, and consider the linear isometries operad $\c{L}(U)$. If $K \subset H \subset G$ are subgroups, then by \cite[Theorem 4.18]{BH},
	\[
	K \to_{\c{L}(U)} H	\quad\t{if and only if}\quad \t{ind}_K^H\t{res}^G_K U \t{ $H$-embeds into } \t{res}^G_H U .
	\]
Such $H$-embeddings may be constructed one subrepresentation at a time because $\t{ind}$ and $\t{res}$ preserve direct sums, and $U$ is a universe. In particular, it is enough to show that for every irreducible $H$-representation $V \subset \t{res}^G_H U$ and every irreducible $W \subset \t{ind}_K^H \t{res}^H_K V$, there is an $H$-embedding of $W$ into $\t{res}^G_H U$.

The condition above always determines if a relation $K \to_{\c{L}(U)} H$ holds or not, but checking it for every possible inclusion $K \subset H$ is recipe for boredom. We review a few techniques for ruling out transfers, following Blumberg and Hill.

\begin{prop}[\protect{\cite[Corollary 4.20]{BH}}] The transfer system $\to_{\c{L}(U)}$ is a refinement of $\to_{\c{K}(U)}$ for any $G$-universe $U$.
\end{prop}

The transfer system $\to_{\c{K}(U)}$ can be computed using Proposition \ref{prop:discind}, which gives an easy upper bound on $\to_{\c{L}(U)}$. The next observation cuts things down further.

\begin{prop}[\protect{\cite[p. 17]{BH}}]\label{prop:BHcond} The transfer system $\to_{\c{L}(U)}$ is saturated for every $G$-universe $U$.
\end{prop}

Briefly, if $K \subset L \subset H$, then $\t{ind}_L^H\t{res}^H_L V$ embeds into $\t{ind}^H_K \t{res}^H_K V$ because the unit of the adjunction $\t{res}^L_K \dashv \t{coind}_K^L \cong \t{ind}_K^L$ is injective, and the right adjoint $\t{ind}_L^H$ preserves monomorphisms.

From here, we can start ruling out relations $K \to_{\c{L}(U)} H$ on a case-by-case basis.

\begin{prop}\label{prop:admlinirrep} Suppose $U$ is a $G$-universe and assume $K \subset H \subset G$ are subgroups such that $K \to_{\c{L}(U)} H$. Then $\t{res}^G_H U$ contains every irreducible $H$-representation with nonzero $K$-fixed points.
\end{prop}

\begin{proof} The trivial $G$-representation $\bb{R}$ embeds into $U$, and therefore there is a chain of $H$-embeddings $\t{ind}_K^H \bb{R} \hookrightarrow \t{ind}_K^H\t{res}^G_K U \hookrightarrow \t{res}^G_H U$. If $W$ is an irreducible $H$-representation such that $W^K \neq \{0\}$, then any nonzero $x \in W^K$ determines a nonzero map $\t{ind}_K^H \bb{R} \to W$. Therefore $W$ $H$-embeds into $\t{ind}_K^H \bb{R}$ and $\t{res}^G_H U$.
\end{proof}

The following special case is used in the proof of \cite[Theorem 4.22]{BH}.

\begin{cor} Keep notation as above. If $1 \to_{\c{L}(U)} G$, then $U$ is complete.
\end{cor}

These tricks will only take us so far, because they are based on one-way implications. There are universes $U$ such that $\to_{\c{L}(U)}$ strictly refines $\to_{\c{K}(U)}$ (cf. \cite[Theorem 4.22]{BH}), there are saturated transfer systems that are not realized by linear isometries operads (cf. Examples \ref{ex:K4linisom}--\ref{ex:Sigma3linisom}), and as the next example shows, the relation $K \to_{\c{L}(U)} H$ need not hold even if $\t{res}^G_H U$ contains all irreducible $H$-representations with nonzero $K$-fixed points. 

\begin{ex} \label{ex:linfpirred} Suppose $G = K_4$ is the Klein four-group and keep notation as in Example \ref{ex:K4Steiner}. Let $K = \la a \ra$, $H = K_4$, and consider the universe $U= [\bb{R} \oplus \sigma_a \oplus \sigma_b]^\infty$. Then every irreducible $K_4$-representation with nonzero $\la a \ra$-fixed points embeds into $U$. However, $\la a \ra \not\to_{\c{L}(U)} K_4$ because $\t{ind}_{\{1,a\}}^{K_4}\t{res}^{K_4}_{\{1,a\}} \sigma_b \cong \t{ind}_{\{1,a\}}^{K_4} \sigma \cong \sigma_b \oplus \sigma_c$, and $\sigma_c$ does not embed into $U$.
\end{ex}

Ultimately, we need to start checking relations $K \to_{\c{L}(U)} H$ individually, i.e. we need to compute the universe $\t{res}^G_H U$ and, if it is not complete, the universe $\t{ind}_K^H \t{res}^G_K U$. Fortunately, saturation implies we do not need to consider all possible inclusions $K \subset H$. It is sometimes simpler to study the longest possible transfer relations, because they imply the intermediate relations, and it is sometimes simpler to study the shortest possible relations. The next result will be useful in our analysis of $C_{p^n}$-linear isometries operads.

\begin{defn}Suppose $G$ is a finite group and $K \subset H \subset G$ are subgroups. We say that the pair $(K,H)$ is \emph{irreducible} if $K$ is a maximal, proper subgroup of $H$.
\end{defn}

\begin{prop}\label{prop:simporbLambda} If $\to$ is a saturated $G$-transfer system, then $\to$ is generated by the relation $\{(K,H) \, | \, K \to H \t{ and } (K,H) \t{ is irreducible} \}$.
\end{prop}

\begin{proof} Let $\rightsquigarrow \,\, = \la (K,H) \, | \, K \to H \t{ and } (K,H) \t{ is irreducible} \ra$. Then $\rightsquigarrow$ refines $\to$ by definition. For the other refinement, suppose that $K \to H$ for some subgroups $K \subset H \subset G$. Since $G$ is a finite group, we can choose a (nonunique) chain of subgroups $K = K_0 \subsetneq K_1 \subsetneq \cdots \subsetneq K_n = H$ such that $(K_i,K_{i+1})$ is irreducible for every $i$. Since $K_0 \to K_n$ and $\to$ is saturated, we have $K_i \to K_{i+1}$, and hence $K_i \rightsquigarrow K_{i+1}$ for all $i$. The chain $K = K_0 \rightsquigarrow \cdots \rightsquigarrow K_n = H$ implies $K \rightsquigarrow H$.
\end{proof}

There is not much more we can say about $\to_{\c{L}(U)}$ at this level of generality. We give a few examples, and then specialize to finite cyclic groups.

\begin{ex}\label{ex:K4linisom} Let $G = K_4$ and keep notation as in Example \ref{ex:K4Steiner}. The next table depicts $\to_{\c{L}(U)}$ for a set of $\Sigma_3$-orbit representatives of $\b{Uni}(K_4)$.
	\[
	\begin{array}{|c|c|}
	\hline
	U	&	\to_{\c{L}(U)}	\\
	\hline
	\bb{R}^\infty	&	\kindt	\\
	(\bb{R} \oplus \sigma_c)^\infty	&	\kindkc	\\
	(\bb{R} \oplus \sigma_b \oplus \sigma_c)^\infty	&	\kindabc	\\
	(\bb{R} \oplus \sigma_a \oplus \sigma_b \oplus \sigma_c)^\infty	&	\kinds	\\
	\hline
	\end{array}
	\]
Thus, the saturated $K_4$-transfer systems
	\[
	\kinda \, \t{and} \, \kindab
	\]
are not realized by $K_4$-linear isometries operads, and the inclusion of the second universe into the third is not preserved. Combined with Example \ref{ex:K4Steiner}, we see that the $\Sigma_3$-orbits of the $K_4$-transfer systems
	\[
	\kinda \, \kindab \, \kindkt \, \kindkct
	\]
are not realized by Steiner or linear isometries operads.
\end{ex}

\begin{ex}\label{ex:Q8linisom} Let $G = Q_8$ and keep notation as in Example \ref{ex:Q8Steiner}. The next table depicts $\to_{\c{L}(U)}$ for a set of $\Sigma_3$-orbit representatives of $\b{Uni}(Q_8)$.
	\[
	\begin{array}{|c|c|c|c|}
	\hline
	U	&	\to_{\c{L}(U)}	&	U	&	\to_{\c{L}(U)}	\\
	\hline
	\bb{R}^\infty	&	\qeia	&	(\bb{R} \oplus \bb{H})^\infty	&	\qeib	\\
	(\bb{R} \oplus \sigma_k)^\infty	&	\qeiz	&	(\bb{R} \oplus \sigma_k \oplus \bb{H})^\infty	&	\qeiu	\\
	(\bb{R} \oplus \sigma_j \oplus \sigma_k)^\infty	&	\qeiv	& (\bb{R} \oplus \sigma_j \oplus \sigma_k \oplus \bb{H})^\infty	&	\qeir	\\
	(\bb{R} \oplus \sigma_i \oplus \sigma_j \oplus \sigma_k)^\infty	&	\qeizc	&	(\bb{R} \oplus \sigma_i \oplus \sigma_j \oplus \sigma_k \oplus \bb{H})^\infty	&	\qeiq	\\
	\hline
	\end{array}
	\]
Thus, the saturated $Q_8$-transfer systems
	\[
	\qeix \ \qeiy \, \qeig \, \qeik
	\]
are not realized by $Q_8$-linear isometries operads, and the inclusions of the universes on the second line into the universes on the third are not preserved. Combined with Example \ref{ex:Q8Steiner}, we see that the $\Sigma_3$-orbits of the $Q_8$-transfer systems
	\[
	\qeix \qeiy \qeiw \qeiza \qeic \qeig \qeid \qeih \qeik \qeie \qeii \qeil \qeij \qeim \qeis \qeit \qeio
	\]
are not realized by Steiner or linear isometries operads.
\end{ex}

\begin{ex}\label{ex:Sigma3linisom} Let $G = \Sigma_3$ and keep notation as in Example \ref{ex:Sigma3Steiner}. The transfer systems for $\Sigma_3$-linear isometries operads are
	\[
	\begin{array}{|c|c|}
	\hline
	U	&	\to_{\c{L}(U)}	\\
	\hline
	\bb{R}^\infty	&	\dsia	\\
	(\bb{R} \oplus \sigma)^\infty	&	\dsig	\\
	(\bb{R} \oplus \Delta)^\infty	&	\dsid	\\
	(\bb{R} \oplus \sigma \oplus \Delta)^\infty	&	\dsii	\\
	\hline
	\end{array}
	\]
Thus, the saturated $\Sigma_3$-transfer systems
	\[
	\dsib \,\t{and}\, \dsic 
	\]
are not realized by $\Sigma_3$-linear isometries operads. Combined with Example \ref{ex:Sigma3Steiner}, we see that the $\Sigma_3$-transfer systems
	\[
	\dsib \, \dsic  \, \dsie \, \dsih
	\]
are not realized by Steiner or linear isometries operads.
\end{ex}

In Examples \ref{ex:K4linisom}--\ref{ex:Sigma3linisom}, every saturated transfer system not realized by a linear isometries operad also is not realized by a Steiner operad. We see no reason why this should be true in general, but we do not know any counterexamples.

\subsection{Finite cyclic groups} Let $G = C_n$ for some natural number $n$. We shall describe a systematic method for computing the transfer systems of $C_n$-linear isometries operads, in terms of two-dimensional rotation representations.

\begin{nota}\label{nota:lambda} For any finite cyclic group $C_n$ with chosen generator $g$, let
	\[
	\lambda_n(m) : C_n \to S^1 \cong SO(2)
	\]
be the $C_n$-representation that sends $g$ to $e^{2 \pi i m/n}$. The character of $\lambda_n(m)$ is
	\[
	\chi(g^j) = 2\t{cos}(2\pi mj/n) = e^{2 \pi i mj/n} + e^{-2 \pi i mj/n} .
	\]

Suppose $d$ and $n$ are natural numbers such that $d \mid n$. We write $\t{res}^n_d$ and $\t{ind}_d^n$ for restriction and induction along the inclusion $C_d \hookrightarrow C_n$ that sends the chosen generator of $C_d$ to the $\frac{n}{d}$th power of the chosen generator of $C_n$.
\end{nota}

The representations $\lambda_n(m)$ have the following properties.

\begin{lem}\label{lem:lambdareps} Suppose $m$, $m'$, $n$, and $d$ are natural numbers.
	\begin{enumerate}
		\item{}If $m \equiv m'$ mod $n$, then $\lambda_n(m) = \lambda_n(m')$.
		\item{}There is always an isomorphism $\lambda_n(m) \cong \lambda_n(-m)$.
		\item{}If $d \mid n$, then $\t{res}^n_d \lambda_n(m) = \lambda_d (m)$.
		\item{}If $d \mid n$, then $\t{ind}_d^n \lambda_d(m) \cong \bigoplus_{a=0}^{n/d - 1} \lambda_n(m + da)$.
	\end{enumerate}
\end{lem}
\begin{proof} The first three statements are clear. For the fourth statement, we compute characters. The character of $\t{ind}_d^n \lambda_d(m)$ is
	\[
	\chi(g^j) = \left\{ \begin{array}{cl}
		\frac{n}{d}( e^{2 \pi i mj/n} + e^{-2 \pi i mj /n} )	 &\t{if $\frac{n}{d} \mid j$}	\\
		0		&\t{otherwise}
	\end{array}\right. ,
	\]
and the character of $\bigoplus_{a=0}^{n/d - 1} \lambda_n(m + da)$ is
	\[
	\chi(g^j) = \Big[ e^{2 \pi i mj/n} \cdot \sum_{a=0}^{n/d-1} (e^{2 \pi i dj/n})^a \Big] + \Big[ e^{-2\pi i mj/n} \cdot \sum_{a=0}^{n/d-1} (e^{-2 \pi i dj/n})^a \Big] .
	\]
These two functions are equal.
\end{proof}

The representation $\lambda_n(m)$ is irreducible, unless
	\begin{enumerate}[label=(\alph*)]
		\item{}$m \equiv 0$ mod $n$, in which case $\lambda_n(m) \cong \bb{R} \oplus \bb{R}$, or
		\item{}$n$ is even and $m \equiv n/2$ mod $n$, in which case $\lambda_n(m) \cong \sigma \oplus \sigma$.
	\end{enumerate}
By parts (1) and (2) of Lemma \ref{lem:lambdareps}, it follows that the irreducible, real $C_n$-representations are
	\[
	\begin{array}{|c|c|}
		\hline
		\t{$n$ odd}	&	\t{$n$ even}	\\
		\hline
		\bb{R}	&	\bb{R}	\\
		\lambda_n(1) \cong \lambda_n(n-1)	&	\lambda_n(1) \cong \lambda_n(n-1)	\\
		\lambda_n(2) \cong \lambda_n(n-2)	&	\lambda_n(2) \cong \lambda_n(n-2)	\\
		\vdots	&	\vdots	\\
		\lambda_n( \frac{n-1}{2} ) \cong \lambda_n(\frac{n+1}{2})	& \lambda_n(\frac{n}{2} - 1) \cong \lambda_n(\frac{n}{2} + 1)	\\
			&	\sigma \\
		\hline
		
	\end{array} 
	\]
	
We may treat both cases simultaneously, because every $C_n$-universe contains infinitely many copies of its irreducible subrepresentations.

\begin{lem}\label{lem:Cnunilambda} Every $C_n$-universe $U$ is of the form $U \cong \bigoplus_{i \in I} \lambda_n(i)^\infty$, where $I$ is a subset of $\bb{Z}/n \cong \{0,1, \dots, n-1 \}$ that contains $0$, and which is closed under additive inversion.
\end{lem}

\begin{proof} The representation $\lambda_n(i)$ is well-defined for every $[i] \in \bb{Z}/n$, by Lemma \ref{lem:lambdareps}. Given an arbitrary $C_n$-universe $U$, rewrite the $\bb{R}^\infty$-summand of $U$ as $\lambda_n(0)^\infty$, and rewrite each $\lambda_n(i)^\infty$-summand as $\lambda_n(i)^\infty \oplus \lambda_n(n-i)^\infty$. If $n$ is even, rewrite any $\sigma^\infty$-summand as $\lambda_n(\frac{n}{2})^\infty$.
\end{proof}

The next result computes the transfer system corresponding to $\c{L}( \bigoplus_{i \in I} \lambda_n(i)^\infty)$ in terms of the translation invariance of $I$ and its reductions. Requiring $I$ to be closed under additive inversion eliminates an ambiguity arising from the isomorphism $\lambda_n(m) \cong \lambda_n(-m)$. 

\begin{prop}\label{prop:admCnlinisom} Let $U = \bigoplus_{i \in I} \lambda_n(i)^\infty$, where $I \subset \bb{Z}/n$ contains $0$ and is closed under additive inversion. Then for any natural numbers $d \mid e \mid n$
	\[
	C_d \to_{\c{L}(U)} C_e	\quad\t{if and only if}\quad	(I \t{ mod } e) + d = (I \t{ mod } e).
	\]
\end{prop}

\begin{proof} By \cite[Theorem 4.18]{BH} and Lemma \ref{lem:lambdareps}, we have $C_d \to_{\c{L}(U)} C_e$ if and only if there is a $C_e$-equivariant embedding
	\[
	 \bigoplus_{i \in I} \bigoplus_{a=0}^{e/d-1} \lambda_e(i + da)^\infty \cong \t{ind}_{d}^e\t{res}^n_d U \hookrightarrow \t{res}^n_e U \cong \bigoplus_{i \in I} \lambda_e(i)^\infty .
	\]
We unwind this condition. First, note that we have a $C_e$-equivariant embedding as above if and only if we have $C_e$-embedding $\lambda_e(i + da) \hookrightarrow \bigoplus_{i \in I} \lambda_e(i)^\infty$ for every $\lambda_e(i+da)$ on the left hand side. In turn, we have such embeddings if and only if every such $\lambda_e(i + da)$ is isomorphic to some $\lambda_e(j)$ with $j \in I$, regardless of whether these representations are irreducible or not.

Now $\lambda_e(a) \cong \lambda_e(b)$ if and only if $a \equiv \pm b$ mod $e$. Since $I$ is closed under additive inversion, it follows $C_d \to_{\c{L}(U)} C_e$ if and only if for every $i \in I$ and $a = 0 , \dots, e/d-1$, there is some $j \in I$ such that $i + da \equiv j$ mod $e$. By induction, it is enough to check when $a = 1$. Therefore $C_d \to_{\c{L}(U)} C_e$ if and only if $(I \t{ mod } e) + d \subset (I \t{ mod } e)$, which is equivalent to $(I \t{ mod } e) + d = (I \t{ mod } e)$ because $I$ is finite.
\end{proof}

\begin{ex} The transfer systems for $C_4$-linear isometries operads are
	\[
	\begin{array}{|c|c|}
		\hline
		U	&	\to_{\c{L}(U)}	\\
		\hline
		\lambda_4 (0)^\infty	&	\cppa	\\
		(\lambda_4 (0) \oplus \lambda_4 (1) \oplus \lambda_4 (3))^\infty	&	\cppb \\
		(\lambda_4 (0) \oplus \lambda_4 (2))^\infty	&	\cppd	\\
		(\lambda_4(0) \oplus \lambda_4(1) \oplus \lambda_4(2) \oplus \lambda_4(3))^\infty	&	\cppe	\\
		\hline
	\end{array}
	\]
These are precisely the saturated $C_4$-transfer systems. Since the Steiner operad $\c{K}(\lambda_4(0) \oplus \lambda_4(1) \oplus \lambda_4(3))^\infty$ realizes $\cppc$, every $C_4$-transfer system is realized by some $\c{K}(U)$ or $\c{L}(U)$. The analogous statement for $C_{p^2}$ is true in general (Corollary \ref{cor:CpnSteinerlinisom}). 
\end{ex}

\begin{ex}\label{ex:C6linisom} The transfer systems for $C_6$-linear isometries operads are
	\[
	\begin{array}{|c|c|}
		\hline
		U	&	\to_{\c{L}(U)}	\\
		\hline
		\lambda_6(0)^\infty	&	\indt	\\
		(\lambda_6(0) \oplus \lambda_6(1) \oplus \lambda_6(5))^\infty	&	\indpq	\\
		(\lambda_6(0) \oplus \lambda_6(2) \oplus \lambda_6(4))^\infty	&	\indpqp	\\
		(\lambda_6(0) \oplus \lambda_6(3))^\infty	&	\indpqq	\\
		(\lambda_6(0) \oplus \lambda_6(1) \oplus \lambda_6(2) \oplus \lambda_6(4) \oplus \lambda_6(5))^\infty	&	\indpq	\\
		(\lambda_6(0) \oplus \lambda_6(1) \oplus \lambda_6(3) \oplus \lambda_6(5))^\infty	&	\indpq	\\
		(\lambda_6(0) \oplus \lambda_6(2) \oplus \lambda_6(3) \oplus \lambda_6(4))^\infty	&	\indpq	\\
		(\lambda_6(0) \oplus \lambda_6(1) \oplus \lambda_6(2) \oplus \lambda_6(3) \oplus \lambda_6(4) \oplus \lambda_6(5))^\infty	&	\inds	\\
		\hline
	\end{array}
	\]
We miss the saturated $C_6$-transfer systems $\indp$ and $\indq$, and many inclusions of $(\lambda_6(0) \oplus \lambda_6(1) \oplus \lambda_6(5))^\infty$ and $(\lambda_6(0) \oplus \lambda_6(3))^\infty$ into larger universes are not preserved. 
\end{ex}

\subsection{Two special cases} In this section we restrict the group $G$ even further. We assume $G$ is a finite cyclic group whose order is either a prime power or a product of two distinct primes, and we identify two cases where every saturated $G$-transfer system is realized by a linear isometries operad.

First, suppose $\abs{G}$ is a prime power. Write $G_k = C_{p^k}$ for $k=0,\dots,n$, so that the subgroup lattice of $C_{p^n}$ is
	\[
	\{1\} = G_0 \hookrightarrow G_1 \hookrightarrow \cdots \hookrightarrow G_{n-1} \hookrightarrow G_n = C_{p^n}.
	\]
We choose generators such that each inclusion $G_j \hookrightarrow G_{j+1}$ above sends the generator of $G_j$ to the $p$th power of the generator of $G_{j+1}$.

\begin{thm}\label{thm:LUCpn} Let $\to$ be a $C_{p^n}$-transfer system, where $p$ is a prime and $n > 0$ is a natural number. Then $\to$ is realized by a linear isometries operad if and only if $\to$ is saturated.
\end{thm}

\begin{proof} The ``only if'' direction is Proposition \ref{prop:BHcond}. We prove the ``if'' direction by direct construction. Suppose $\to$ is saturated. By Proposition \ref{prop:simporbLambda}, $\to$ is generated by its irreducible relations.  Thus, there are integers $0 \leq k_1 < \cdots < k_m < n$ such that $\to \,\, = \la (G_{k_i}, G_{k_{i+1}}) \, | \, 1 \leq i \leq m \ra$. Let $I \subset \bb{Z}/p^n$ be the set
	\[
	I = \Big\{ \pm \!( a_1 p^{k_1} + \cdots + a_m p^{k_m} ) \, \Big| \, 0 \leq  a_1 , \dots , a_m < p \Big\} ,
	\]
and let $U = \bigoplus_{i \in I} \lambda_{p^n}(i)^\infty$. We shall prove that $\to \,\, = \,\, \to_{\c{L}(U)}$. To start, note that $\to_{\c{L}(U)}$ is saturated by \cite[p. 17]{BH}, and therefore Proposition \ref{prop:simporbLambda} implies $\to_{\c{L}(U)} \,\, = \la (K,H) \, | \, K \to_{\c{L}(U)} H \t{ and } (K,H) \t{ is irreducible}\ra$. Thus, it will be enough to show that the irreducible relations in $\to_{\c{L}(U)}$ are precisely the pairs $(G_{k_i} , G_{k_{i+1}})$ for $\to$.

Suppose $(G_{k_i} , G_{k_{i+1}})$ is an irreducible generator of $\to$. The set $(I \t{ mod } p^{k_i + 1})$ consists of all residues of the form $\pm ( a_1 p^{k_1} + \cdots + a_i p^{k_i})$ with $0 \leq a_1 , \dots , a_i < p$, and this subset of $\bb{Z}/p^{k_i + 1}$ is closed under $(-) + p^{k_i}$. Therefore $G_{k_i} \to_{\c{L}(U)} G_{k_{i+1}}$.

Now consider an irreducible pair $(G_j,G_{j+1})$ for some $j \neq k_1 , \dots , k_m$. We shall show $G_j \not\to_{\c{L}(U)} G_{j+1}$. We study the cases $j < k_1$, $k_i < j < k_{i+1}$, and $k_m < j$ separately. In each case, it will be enough to show $p^j \notin (I \t{ mod } p^{j+1})$. If $j < k_1$, then $(I \t{ mod } p^{j+1}) = \{0\} \subset \bb{Z}/p^{j+1}$, which does not contain $p^j$. If $k_i < j < k_{i+1}$, then $(I \t{ mod } p^{j+1}) = \{ \pm ( a_1 p^{k_1} + \cdots a_i p^{k_i}) \}$ as above, and $0 \leq a_1 p^{k_1} + \cdots + a_i p^{k_i} < p^j$ for all $0 \leq a_1 , \dots , a_i < p$. Therefore $0 < p^j \mp (a_1 p^{k_1} + \cdots + a_i p^{k_i}) < p^{j+1}$, and hence $p^j \notin (I \t{ mod } p^{j+1})$. The case where $k_m < j$ is similar.
\end{proof}

\begin{cor}\label{cor:CpnSteinerlinisom} Suppose that $p$ is a prime and that $n > 0$ is a natural number. If $n =1$ or $2$, then every $C_{p^n}$-transfer system is realized by some Steiner or linear isometries operad. If $n \geq 3$, then there are $C_{p^n}$-transfer systems that are not realized by any such operad.
\end{cor}

\begin{proof} The result for $C_p$ is trivial, because the minimum and maximum transfer systems are always realized by Steiner and linear isometries operads. For $C_{p^2}$, Theorem \ref{thm:discind} ensures that $\cppa$, $\cppc$, $\cppd$, and $\cppe$ are realized by Steiner operads, and Theorem \ref{thm:LUCpn} ensures that $\cppa$, $\cppb$, $\cppd$, and $\cppe$ are realized by linear isometries operads. These transfer systems exhaust $\b{Tr}(C_{p^2})$. If $n \geq 3$, then Theorem \ref{thm:notSteinerorlinisom} implies the $C_{p^n}$-transfer system $\la (1,C_{p^2}) \ra$ is not realized by any Steiner or linear isometries operad.
\end{proof}

Finally, suppose $\abs{G} = pq$ for primes $p < q$, and recall the notational conventions from Figure \ref{fig:CpqK4}. In the remainder of this section, we shall prove the following result for $C_{pq}$-linear isometries operads.

\begin{thm}\label{thm:Cpqlinisom} Suppose $p$ and $q$ are primes such that $p < q$.
	\begin{enumerate}
		\item{}If $p=2$ and $q=3$, then every saturated $C_{pq}$-transfer system except $\indp$ and $\indq$ is realized by a linear isometries operad.
		\item{}If $p=2$ or $3$ and $q>3$, then every saturated $C_{pq}$-transfer system except $\indq$ is realized by a linear isometries operad.
		\item{}If $p,q > 3$, then a $C_{pq}$-transfer system is realized by a linear isometries operad if and only if it is saturated.
	\end{enumerate}
\end{thm}

Combined with Example \ref{ex:Cp3CpqSteiner}, we deduce the following.

\begin{cor}Suppose $p$ and $q$ are primes such that $p < q$.
	\begin{enumerate}
		\item{}If $p=2$ and $q=3$, then every $C_{pq}$-transfer system except $\indp$ and $\indq$ is realized by a Steiner or linear isometries operad.
		\item{}If $p=2$ or $3$ and $q>3$, then every $C_{pq}$-transfer system except $\indq$ is realized by a Steiner or linear isometries operad.
		\item{}If $p,q > 3$, then every $C_{pq}$-transfer system is realized by a Steiner or linear isometries operad.
	\end{enumerate}
\end{cor}

Part (1) of Theorem \ref{thm:Cpqlinisom} is just Example \ref{ex:C6linisom}, and Lemmas \ref{lem:Cpqlinisom1} and \ref{lem:Cpqlinisom2} below handle the rest. For any set $I \subset \bb{Z}/pq$ that contains $0$ and is closed under additive inversion, let $\to_{I}$ be the transfer system for $\c{L} ( \bigoplus_{i \in I} \lambda_{pq}(i)^\infty )$.

\begin{lem}\label{lem:Cpqlinisom1} Suppose $p$ and $q$ are prime, $p < q$, and $q>3$. Then we have the following transfer systems.
	\[
	\begin{array}{|c|c|}
	\hline
	I \subset \bb{Z}/pq	&	\to_{I}	\\
	\hline
	\{0\}	&	\indt	\\
	\{0, \pm 1 , \pm 2 , \dots , \pm \lfloor p/2 \rfloor \}	&	\indp	\\
	\{ 0, \pm 1, \pm 2, \dots, \pm \lfloor q/2 \rfloor \}	&	\indpq	\\
	\{ 0,p,2p,\dots, p(q-1) \}	&	\indpqp	\\
	\{ 0,q,2q,\dots, (p-1)q \}	&	\indpqq	\\
	\{ 0,1,2,\dots,pq -1 \}		&	\inds	\\
	\hline
	\end{array}
	\]
\end{lem}

\begin{proof} We apply Proposition \ref{prop:admCnlinisom} repeatedly. The computations for $I = \{0\}$ and $I = \{0,1,\dots,pq-1\}$ are clear, because these index sets correspond to a trivial universe and a complete universe.

If $I = \{ 0 , 1 , \dots , \lfloor p/2 \rfloor , pq - \lfloor p/2 \rfloor , \dots , pq - 1 \}$, then the inequalities $\lfloor p/2 \rfloor < \lfloor p/2 \rfloor + 1 , p , q < pq - \lfloor p/2 \rfloor$ imply $(I \t{ mod } pq)$ has no translation invariance. We obtain $\lfloor p/2 \rfloor < \lfloor p/2 \rfloor + 1 < q- \lfloor p/2 \rfloor$ using the assumption $q>3$, and this implies $(I \t{ mod } q)$ also has no translation invariance. Finally, $(I \t{ mod } p) = \{0,1,\dots, p-1\}$ is invariant under $(-) + 1$. Thus, the only nontrivial transfer is $C_1 \to_{I} C_p$.

If $I = \{0,1,\dots, \lfloor q/2 \rfloor , pq - \lfloor q/2 \rfloor , \dots , pq - 1\}$, then the inequalities $\lfloor q/2 \rfloor < \lfloor q/2 \rfloor + 1 , \lfloor q/2 \rfloor + p , q < pq - \lfloor q/2 \rfloor$ imply that $I$ has no translation invariance. We have $(I \t{ mod } p) = \{0,1,\dots,p-1\}$ and $(I \t{ mod } q) = \{0,1,\dots, q-1\}$, which both are invariant under $(-)+1$. Thus, the transfers are $C_1 \to_I C_p$ and $C_1 \to_I C_q$.

If $I = \{0,p,2p,\dots,p(q-1)\}$, then $0 < 1 < p$ and $q \notin I$. Therefore $I$ is only invariant under $(-) + p $. Next, $(I \t{ mod } p) = \{0\}$, so it has no translation invariance. Finally, $(I \t{ mod } q) = \{0,1,\dots,q-1\}$, which is invariant under $(-)+1$. Thus, the transfers are $C_p \to_I C_{pq}$ and $C_1 \to_I C_q$. A similar argument works for $I = \{ 0,q,2q,\dots, (p-1)q \}$.
\end{proof}

\begin{lem}\label{lem:Cpqlinisom2} Suppose that $p$ and $q$ are prime and $p < q$. If $p=2$ or $3$, then $\indq$ is not realized by any $C_{pq}$-linear isometries operad. If $p>3$, then it is realized by the $C_{pq}$-linear isometries operad over $U \big( \! \pm \! 1, 0 , p , 2p , \dots , p(q-1) \big) $.
\end{lem}

\begin{proof} Suppose first that $p > 3$, and let $I = \{0,1,p,2p,\dots,p(q-1),pq-1\}$. Then $I \subset \bb{Z}/pq$ has no translation invariance because $p < p+1 < 2p$ and $q \notin I$. Next, $(I \t{ mod } p) = \{0,1,p-1\}$ also has no translation invariance because $1 < 2 < p-1$. Finally, $(I \t{ mod } q) = \{0,1,\dots,q-1\}$, which is invariant under $(-)+1$. Therefore $C_1 \to_I C_q$ is the only nontrivial transfer.

Now suppose that $p = 2$ or $3$. We shall prove that $\indq$ cannot be realized by a linear isometries operad. Suppose $I \subset \bb{Z}/pq$ is such that $C_1 \to_I C_q$ but $C_1 \not\to_I C_p$. Then $I \subset p ( \bb{Z}/pq)$, because if $(I \t{ mod } p) \neq \{0\}$, then $\t{res}^{pq}_p U(I)$ is complete. The reduction map $\pi : \bb{Z}/pq \to \bb{Z}/q$ induces a bijection $\pi : p(\bb{Z}/pq) \to \bb{Z}/q$, and since $C_1 \to_I C_q$, we must have $\pi(I) = \bb{Z}/q$. Therefore $I = p(\bb{Z}/pq)$, and therefore $C_p \to_I C_{pq}$. Thus, no $C_{pq}$-linear isometries operad $\c{L}(U(I))$ can realize $\indq$.
\end{proof}

\appendix

\section{Filtering lattices of transfer systems}\label{sec:filtTr}

This appendix describes method for filtering $\b{Tr}(G)$ by simpler sublattices. We find that these filtrations clarify the structure of $\b{Tr}(G)$, and we believe they may be useful in inductive arguments. To be precise, we construct a chain of maximal length in $\b{Tr}(G)$, and then pass to the corresponding sequence of under-lattices. We begin with examples, and then explain the general procedure (Construction \ref{constr:filTr}).

Recall the diagrams for $\b{Tr}(C_{p^3})$ and $\b{Tr}(K_4)$ from Figures \ref{fig:Cpn} and \ref{fig:CpqK4}.

\begin{ex}\label{ex:filTrCp3} The sequence
	\[
	\cpppa \,\, \cpppb \,\, \cpppc \,\, \cpppe \,\, \cpppg \,\, \cpppi \,\, \cpppn
	\]
is a maximal-length chain in $\b{Tr}(C_{p^3})$. The corresponding under-lattices are
	\[
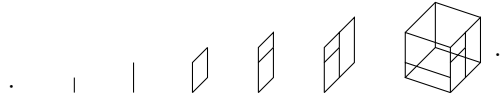
	\begin{array}{ccccccc}
		\begin{tikzpicture}[scale=0.2]
			\node at (0,0) {$$};
			\node at (-1,6) {$$};
			\node at (0,0) {$\cdot$};
		\end{tikzpicture}
		&
		\begin{tikzpicture}[scale=0.2]
			\node at (0,0) {$$};
			\node at (-1,6) {$$};
			\draw (0,0) -- (0,1) ;
		\end{tikzpicture}
		&
		\begin{tikzpicture}[scale=0.2]
			\node at (0,0) {$$};
			\node at (-1,6) {$$};
			\draw (0,0) -- (0,1) -- (0,2) ;
		\end{tikzpicture}
		&
		\begin{tikzpicture}[scale=0.2]
			\node at (0,0) {$$};
			\node at (-1,6) {$$};
			\draw (0,0) -- (0,2) -- (1,3) ;
			\draw (0,0) -- (1,1) -- (1,3) ;
		\end{tikzpicture}
		&
		\begin{tikzpicture}[scale=0.2]
			\node at (0,0) {$$};
			\node at (-1,6) {$$};
			\draw (0,0) -- (0,2) -- (1,3) ;
			\draw (0,0) -- (1,1) -- (1,3) ;
			\draw (0,2) -- (0,3) -- (1,4) -- (1,3) ;
		\end{tikzpicture}
		&
		\begin{tikzpicture}[scale=0.2]
			\node at (0,0) {$$};
			\node at (-1,6) {$$};
			\draw (0,0) -- (0,2) -- (1,3) ;
			\draw (0,0) -- (1,1) -- (1,3) ;
			\draw (0,2) -- (0,3) -- (1,4) -- (1,3) ;
			\draw (1,1) -- (2,2) -- (2,5) -- (1,4) ;
		\end{tikzpicture}
		&
		\begin{tikzpicture}[scale=0.2]
			\node at (-4,0) {$$};
			\node at (0,0) {$$};
			\node at (-1,6) {$$};
			\draw (0,0) -- (0,2) -- (1,3) ;
			\draw (0,0) -- (1,1) -- (1,3) ;
			\draw (0,2) -- (0,3) -- (1,4) -- (1,3) ;
			\draw (1,1) -- (2,2) -- (2,5) -- (1,4) ;
			\draw (0,0) -- (-3,1) -- (-1,3) -- (2,2) ;
			\draw (-3,1) -- (-3,2) -- (0,1) ;
			\draw (-3,2) -- (-3,4) -- (0,3) ;
			\draw (-3,4) -- (-1,6) -- (2,5) ;
			\draw (-1,3) -- (-1,6) ;
		\end{tikzpicture}
	\end{array} .
	\]
\end{ex}

\begin{ex}\label{ex:filTrK4} The sequence
	\[
	\kindt \,\, \kinda \,\, \kindab \,\, \kindabc \,\, \kindkt \,\, \kindkat \,\, \kindkab \,\, \kinds
	\]
is a maximal-length chain in $\b{Tr}(K_4)$. The corresponding under-lattices are
	\[
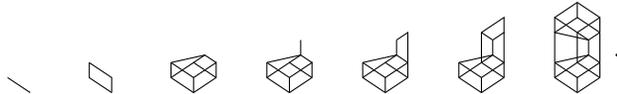
	\begin{array}{cccccccc}
		\begin{tikzpicture}[scale=0.2]
			\node at (0,0) {$\cdot$};
			\node at (0,6) {$$};
		\end{tikzpicture}
		&
		\begin{tikzpicture}[scale=0.2]
			\node at (-2.5,0) {$$};
			\node at (0,0) {$$};
			\node at (0,6) {$$};
			\draw (0,0) -- (-1.5,1) ;
		\end{tikzpicture}
		&
		\begin{tikzpicture}[scale=0.2]
			\node at (-2.5,0) {$$};
			\node at (0,0) {$$};
			\node at (0,6) {$$};
			\draw (0,0) -- (-1.5,1) -- (-1.5,2) ;
			\draw (0,0) -- (0,1) -- (-1.5,2) ;
		\end{tikzpicture}
		&
		\begin{tikzpicture}[scale=0.2]
			\node at (-2.5,0) {$$};
			\node at (0,0) {$$};
			\node at (0,6) {$$};
			\draw (0,0) -- (-1.5,1) -- (-1.5,2) ;
			\draw (0,0) -- (0,1) -- (-1.5,2) ;
			\draw (0,0) -- (1.5,1) -- (1.5,2) -- (0,1) ;
			\draw (-1.5,1) -- (0,2) -- (1.5,1) ;
			\draw (-1.5,2) -- (0.75,2.5) -- (1.5,2) ;
			\draw (0,2) -- (0.75,2.5) ;
		\end{tikzpicture}
		&
		\begin{tikzpicture}[scale=0.2]
			\node at (-2.5,0) {$$};
			\node at (0,0) {$$};
			\node at (0,6) {$$};
			\draw (0,0) -- (-1.5,1) -- (-1.5,2) ;
			\draw (0,0) -- (0,1) -- (-1.5,2) ;
			\draw (0,0) -- (1.5,1) -- (1.5,2) -- (0,1) ;
			\draw (-1.5,1) -- (0,2) -- (1.5,1) ;
			\draw (-1.5,2) -- (0.75,2.5) -- (1.5,2) ;
			\draw (0,2) -- (0.75,2.5) ;
			\draw (0.75,2.5) -- (0.75,3.5) ;
		\end{tikzpicture}
		&
		\begin{tikzpicture}[scale=0.2]
			\node at (-2.5,0) {$$};
			\node at (0,0) {$$};
			\node at (0,6) {$$};
			\draw (0,0) -- (-1.5,1) -- (-1.5,2) ;
			\draw (0,0) -- (0,1) -- (-1.5,2) ;
			\draw (0,0) -- (1.5,1) -- (1.5,2) -- (0,1) ;
			\draw (-1.5,1) -- (0,2) -- (1.5,1) ;
			\draw (-1.5,2) -- (0.75,2.5) -- (1.5,2) ;
			\draw (0,2) -- (0.75,2.5) ;
			\draw (0.75,2.5) -- (0.75,3.5) ;
			\draw (0.75,3.5) -- (1.5,4) -- (1.5,2) ;
		\end{tikzpicture}
		&
		\begin{tikzpicture}[scale=0.2]
			\node at (-2.5,0) {$$};
			\node at (0,0) {$$};
			\node at (0,6) {$$};
			\draw (0,0) -- (-1.5,1) -- (-1.5,2) ;
			\draw (0,0) -- (0,1) -- (-1.5,2) ;
			\draw (0,0) -- (1.5,1) -- (1.5,2) -- (0,1) ;
			\draw (-1.5,1) -- (0,2) -- (1.5,1) ;
			\draw (-1.5,2) -- (0.75,2.5) -- (1.5,2) ;
			\draw (0,2) -- (0.75,2.5) ;
			\draw (0.75,2.5) -- (0.75,3.5) ;
			\draw (0.75,3.5) -- (1.5,4) -- (1.5,2) ;
			\draw (0,2) -- (0,4) -- (1.5,5) -- (1.5,4) ;
			\draw (0,4) -- (0.75,3.5) ;
		\end{tikzpicture}
		&
		\begin{tikzpicture}[scale=0.2]
			\node at (-2.5,0) {$$};
			\node at (0,0) {$$};
			\node at (0,6) {$$};
			\draw (0,0) -- (-1.5,1) -- (-1.5,2) ;
			\draw (0,0) -- (0,1) -- (-1.5,2) ;
			\draw (0,0) -- (1.5,1) -- (1.5,2) -- (0,1) ;
			\draw (-1.5,1) -- (0,2) -- (1.5,1) ;
			\draw (-1.5,2) -- (0.75,2.5) -- (1.5,2) ;
			\draw (0,2) -- (0.75,2.5) ;
			\draw (0.75,2.5) -- (0.75,3.5) ;
			\draw (0.75,3.5) -- (1.5,4) -- (1.5,2) ;
			\draw (0,2) -- (0,4) -- (1.5,5) -- (1.5,4) ;
			\draw (0,4) -- (0.75,3.5) ;
			\draw (-1.5,2) -- (-1.5,4) -- (0.75,3.5) ;
			\draw (-1.5,4) -- (-1.5,5) -- (0,4) ;
			\draw (-1.5,4) -- (0,5) -- (1.5,4) ;
			\draw (-1.5,5) -- (0,6) -- (1.5,5) ;
			\draw (0,5) -- (0,6) ;
		\end{tikzpicture}
	\end{array} .
	\]
\end{ex}
The chain of $\Sigma_3$-transfer systems
	\[
	\dsia \,\, \dsib \,\, \dsid \,\, \dsie \,\, \dsif \,\, \dsii
	\]
and the chain of $Q_8$-transfer systems
	\[
	\qeia \,\, \qeib \,\, \qeic \,\, \qeig \,\, \qeih \,\, \qeik \,\, \qeil \,\, \qeir \,\, \qeis \,\, \qeit \,\, \qeio \,\, \qeip \,\, \qeiq
	\]
both have maximal length. We leave it to the interested reader to determine the associated filtrations on $\b{Tr}(\Sigma_3)$ and $\b{Tr}(Q_8)$.

Here is a method for producing such chains. Regard a $G$-transfer system $\to$ as a subset of $\b{Sub}(G)^{\times 2}$. The diagonal $\Delta\b{Sub}(G) = \{ (H,H) \, | \, H \subset G \}$ is the initial system, and $\subset_G \,\, = \{ (K,H) \in \b{Sub}(G)^{\times 2} \, | \, K \subset H\}$ is the terminal system. In general, a $G$-transfer system $\to$ corresponds to a $G$-set intermediate to $\Delta\b{Sub}(G)$ and $\subset_G$, because $\to$ is reflexive, refines inclusion, and is closed under conjugation. 

\begin{constr}\label{constr:filTr} We build an increasing sequence of transfer systems connecting $\Delta\b{Sub}(G)$ to $\subset_G$, one orbit at a time.
	
To start, we filter $S = \b{Sub}(G)$. Let $S_{-1} = \varnothing$. Then, assuming that $S_k \subsetneq \b{Sub}(G)$ has been defined, choose a minimal subgroup $H_{k+1} \in \b{Sub}(G) \setminus S_k$ and let $S_{k+1} = S_k \cup \{ gH_{k+1}g^{-1} \, | \, g \in G \}$. Continuing like this, we obtain a finite chain $\varnothing = S_{-1} \subsetneq S_0 \subsetneq \cdots \subsetneq S_N = \b{Sub}(G)$. Let $D_i = S_i \setminus S_{i-1} = [H_i]$ for $0 \leq i \leq N$. This partitions $\b{Sub}(G)$.

Next, we partition the inclusion relation on $\b{Sub}(G)$. For all pairs of integers $0 \leq i < j \leq N$ define
	\[
	\subset_{ij} \,\, = \{ (K,H) \in \b{Sub}(G)^{\times 2} \, | \, K \in D_i ,\, H \in D_j ,\, \t{and } K \subset H \} ,
	\]
and choose an orbit decomposition $\subset_{ij} \,\, = \coprod_{k=1}^{n(i,j)} O(i,j)_k$ with respect to the diagonal conjugation action. The $G$-set $\subset_{ij}$ could be empty, it could be transitive, or it could contain more than one orbit (cf. Example \ref{ex:conjS4}).

Finally, we order the orbits in $\coprod_{i<j} \subset_{ij}$ as
	\begin{align*}
	&O(0,1)_1 ,  \dots , O(0,1)_{n(0,1)} , 
	O(0,2)_1 ,  \dots , O(0,2)_{n(0,2)} ,
	O(1,2)_1 ,  \dots , O(1,2)_{n(1,2)} ,	\\
	&O(0,3)_1 , \dots , O(0,3)_{n(0,3)} ,
	O(1,3)_1 ,  \dots , O(1,3)_{n(1,3)} , O(2,3)_1 , \dots
	\end{align*}
and define a corresponding chain of relations by $\to_0 \,\, = \Delta \b{Sub}(G)$, and
	\[
	\to \!\! (i,j)_k = \Delta\b{Sub}(G) \,\, \sqcup \!\!\!\! \coprod_{O \leq O(i,j)_k} \!\!\!\! O 
	\]
for any $0 \leq i < j \leq N$ and $1 \leq k \leq n(i,j)$. We shall prove that
	\begin{align*}
		\to_0 \,\, &< \,\, \to \!\! (0,1)_1 < \dots < \,\, \to \!\! (0,1)_{n(0,1)} < \,\, \to \!\! (0,2)_1 < \cdots < \,\, \to \!\! (0,2)_{n(0,2)} \\
		&< \,\, \to \!\! (1,2)_1 \,\, < \cdots < \,\, \to \!\! (1,2)_{n(1,2)} < \,\, \to \!\! (0,3)_1 < \cdots < \,\, \to \!\! (N-1,N)_{n(N-1,N)}
	\end{align*}
is a chain of $G$-transfer systems, which connects the minimum system to the maximum system, and which has maximal length among all chains in $\b{Tr}(G)$.
\end{constr}

The chain in Example \ref{ex:filTrCp3} arises from the layers $D_i = [C_{p^i}]$, where $0 \leq i \leq 3$. In this case, we have $\subset_{ij} \,\, = \{(C_{p^i},C_{p^j})\}$ for every $0 \leq i < j \leq 3$, and the order on the orbits of $\coprod_{i < j} \subset_{ij}$ is
	\[
	\{(C_1,C_p)\} ,\, \{(C_{1},C_{p^2})\} ,\, \{(C_{p},C_{p^2})\} ,\, \{(C_{1},C_{p^3})\} ,\, \{(C_{p},C_{p^3})\} ,\, \{(C_{p^2},C_{p^3})\}.
	\]
The chain in Example \ref{ex:filTrK4} arises from the layers
	\[
	D_0 = [1] , \, D_1 = [ \la a \ra ] , \, D_2 = [ \la b \ra ] , \, D_3 = [ \la c \ra ] , \, D_4 = [ K_4 ] .
	\]

Now for the general case. Keep notation as in Construction \ref{constr:filTr}. We begin by analyzing the filtration and partition of $\b{Sub}(G)$.

\begin{lem}\label{lem:filsubG} The subset $S_k \subset \b{Sub}(G)$ is downward-closed and conjugation-invariant for any integer $-1 \leq k \leq N$.
\end{lem}

\begin{proof} There is nothing to prove for $S_{-1} = \varnothing$. Now suppose that $S_k$ is downward-closed and conjugation-invariant, and consider $S_{k+1} = S_{k} \cup [H_{k+1}]$. The set $S_{k+1}$ is still conjugation-invariant because we have attached an entire conjugacy class. Now suppose $L \in S_{k+1}$ and $M \subsetneq L$. We must prove that $M \in S_{k+1}$. If $L \in S_k$, then this follows because $S_k$ is downward-closed. If $L \in H_{k+1}$, then $L = gH_{k+1} g^{-1}$ for some $g \in G$, and hence $g^{-1}M g \subsetneq H_{k+1}$. Therefore $g^{-1} M g \in S_k$, because we chose $H_{k+1}$ as a minimal element of $\b{Sub}(G) \setminus S_k$, and therefore $M \in S_k \subset S_{k+1}$ because $S_k$ is conjugation-invariant.
\end{proof}

\begin{lem}\label{lem:layersubG} Suppose that $K \subset H$ and $H \in D_j$. Then $K \in D_i$ for some $i \leq j$, and $i=j$ if and only if $K = H$.
\end{lem}

\begin{proof} We have $H \in D_j \subset S_j$ and $K \subset H$. Therefore $K \in S_j$ by Lemma \ref{lem:filsubG}, and therefore $K \in D_i$ for some $i \leq j$ because $S_j = \coprod_{i=0}^j D_i$. If $i = j$, then $K$ and $H$ are conjugate. Therefore $K = H$ because $\abs{K} = \abs{H}$, $K \subset H$, and both sides are finite. If $K = H$, then $K \in D_i \cap D_j$, and hence $i=j$ because the sets $D_i$ are disjoint.
\end{proof}

Next, we consider the induced partition of the inclusion relation on $S_m$.

\begin{lem}\label{lem:partitioninc} For any $0 \leq m \leq N$, the set $\subset_m \,\, = \{ (K,H) \in S_m^{\times 2} \, | \, K \subset H \}$ is partitioned into the disjoint union
	\[
	\Delta S_m \,\, \sqcup \!\!\!\! \coprod_{0 \leq i < j \leq m} \!\!\!\! \subset_{ij} .
	\]
\end{lem}

\begin{proof} The sets $\Delta S_m$, $\subset_{01}$, $\subset_{02}$, $\subset_{12}$, \dots are pairwise disjoint because $\subset_{ij} \,\, \subset D_i \times D_j$ and the sets $D_0$, $D_1$, \dots are pairwise disjoint.

The containment $\Delta S_m \subset \,\, \subset_m$ follows from the reflexivity of inclusion, and for every $0 \leq i < j \leq m$, the containment $\subset_{ij} \,\, \subset \,\, \subset_m$ follows from the inclusions $D_k \subset S_k \subset S_m$ for $k=i,j$. Therefore $\Delta S_m \sqcup \coprod_{i<j} \subset_{ij}$ is contained in $\subset_m$.

Now suppose $(K,H) \in \,\, \subset_m$. Then $K \in D_i$ and $H \in D_j$ for some $i,j \leq m$, because $S_m = \coprod_{i=0}^m D_i$. Lemma \ref{lem:layersubG} implies $i \leq j$, and if $i < j$, then $(K,H) \in \,\, \subset_{ij}$. If $i = j$, then $K=H$ by Lemma \ref{lem:layersubG}, and therefore $(K,H) \in \Delta S_m$.
\end{proof}

Finally, we analyze the relations $\to \!\! (i,j)_k$. Note that there are refinements
	\[
	\Delta \b{Sub}(G) \, \sqcup \!\!\!\!\!\! \coprod_{0 \leq a < b < j} \!\!\!\!\!\!\! \subset_{ab} \, \sqcup \!\! \coprod_{0 \leq a < i} \!\!\! \subset_{aj} 
	\,\,\, \leq \,\,\, \to \!\! (i,j)_k \,\,\, \leq \,\,\,
	\Delta \b{Sub}(G) \, \sqcup \!\!\!\!\!\! \coprod_{0 \leq a < b < j} \!\!\!\!\!\!\! \subset_{ab} \, \sqcup \!\! \coprod_{0 \leq a \leq i} \!\!\! \subset_{aj} .
	\]

\begin{prop}\label{prop:filTr} The initial $G$-transfer system is $\to_0$, the terminal $G$-transfer system is $\to \!\! (N-1,N)_{n(N-1,N)}$, and the binary relation $\to \!\! (i,j)_k$ is a $G$-transfer system for every  $0 \leq i < j \leq N$ and $1 \leq k \leq n(i,j)$. Moreover, the chain
	\begin{align*}
		\to_0 \,\, &< \,\, \to \!\! (0,1)_1 < \dots < \,\, \to \!\! (0,1)_{n(0,1)} < \,\, \to \!\! (0,2)_1 < \cdots < \,\, \to \!\! (0,2)_{n(0,2)} \\
		&< \,\, \to \!\! (1,2)_1 \,\, < \cdots < \,\, \to \!\! (1,2)_{n(1,2)} < \,\, \to \!\! (0,3)_1 < \cdots < \,\, \to \!\! (N-1,N)_{n(N-1,N)}
	\end{align*}
has maximal length among all chains in $\b{Tr}(G)$.
\end{prop}

\begin{proof} It is clear that $\to_0 \,\, = \Delta \b{Sub}(G)$ is the initial transfer system, and Lemma \ref{lem:partitioninc} implies that $\to \!\! (N-1,N)_{n(N-1,N)}$ is the terminal transfer system. Now suppose that $0 \leq i < j \leq N$ and $1 \leq k \leq n(i,j)$, and write $\to \!\! (i,j)_k = \,\, \rightsquigarrow \sqcup \,O(i,j)_k$. Assume inductively that $\rightsquigarrow$ is a $G$-transfer system. We shall show that $\to \!\! (i,j)_k$ is also a $G$-transfer system.

The $G$-refinements $\Delta \b{Sub}(G) \leq \,\, \to \!\! (i,j)_k \leq \,\, \subset_G$ imply that $\to \!\! (i,j)_k$ is reflexive, is closed under conjugation, and refines inclusion. Antisymmetry follows.

Now for restriction. Suppose $(K,H) \in \,\, \to \!\! (i,j)_k$ and $L \subset H$. Then $K \subset H$ as well. If $(K,H) \in \,\, \rightsquigarrow$, then  $(K \cap L, L) \in \,\, \rightsquigarrow \,\, \subset \,\, \to \!\! (i,j)_k$ by induction. If $L = H$, then $(K \cap L,L) = (K,H) \in \,\, \to \!\! (i,j)_k$ by assumption. Thus, we may assume $(K,H) \in O(i,j)_k$ and $L \subsetneq H$. We have $H \in D_j$, and since $K \cap L \subset L \subsetneq H$, it follows $K \cap L , L \in S_{j-1}$ by Lemma \ref{lem:layersubG}. By Lemma \ref{lem:partitioninc}, $(K \cap L , L) \in \,\, \subset_{j-1} \,\, \subset \,\, \to \!\! (i,j)_k$.

Next, we consider transitivity. Suppose $(J,K) , (K,H) \in \,\, \to \!\! (i,j)_k$. Then $J \subset K \subset H$, and if either $J = K$ or $K = H$, then $(J,H) \in \,\, \to \!\! (i,j)_k$ is immediate. Thus, we may assume $J \subsetneq K \subsetneq H$. If $(J,K) , (K,H) \in \,\, \rightsquigarrow$, then $(J,H) \in \,\, \rightsquigarrow \,\, \subset \,\, \to \!\! (i,j)_k$ by induction. Thus, we may also assume that one of $(J,K)$ and $(K,H)$ is in $O(i,j)_k$. Since $(K,H) \in \,\, \to \!\! (i,j)_k$, there are $b < c \leq j$ such that $K \in D_b$ and $H \in D_c$, by Lemma \ref{lem:layersubG}. Therefore $(J,K) \notin O(i,j)_k$, which implies $(K,H) \in O(i,j)_k$. Thus, $K \in D_i$, $H \in D_j$, and there is $a < i$ such that $J \in D_a$, by Lemma \ref{lem:layersubG} again. It follows $(J,H) \in \,\, \subset_{aj} \,\, \subset \,\, \to \!\! (i,j)_k$.

Finally, note that every $G$-transfer system is the disjoint union of $\Delta \b{Sub}(G)$ with at most $M = \sum_{i<j} n(i,j)$ orbits in $\subset_G \!\! \setminus \Delta \b{Sub}(G) = \coprod_{ i < j} \subset_{ij}$. Moreover, each proper refinement $\to \,\, < \,\, \to'$ increases the number of $G$-orbits. Thus, a chain in $\b{Tr}(G)$ can have length at most $M+1$, and $\to_0 \,\, < \cdots < \,\, \to \!\! (N-1,N)_{n(N-1,N)}$ attains this bound.
\end{proof}

\section{Generating transfer systems}\label{sec:genindsys}

This appendix explains how to generate a transfer system from a prescribed set of relations. We describe the basic technique (Construction \ref{constr:transysgen}), calculate a few general cases (Propositions \ref{prop:multitsKnorm} and \ref{prop:indGorbab}), and then reinterpret our construction in terms of indexing systems and indexing categories (Propositions \ref{thm:indgenO} and \ref{prop:indgenG}).

\begin{constr}\label{constr:transysgen}Suppose $G$ is a finite group, and $R$ is binary relation on $\b{Sub}(G)$ that refines inclusion, i.e. if $K R H$, then $K \subset H$. Define
	\begin{align*}
		R_0	&:=	R ,	\\
		R_1	&:=	\bigcup_{(K,H) \in R_0} \{ (gKg^{-1} , gHg^{-1}) \, | \, g \in G \}	\\
		R_2	&:=	\bigcup_{(K,H) \in R_1} \{ (L \cap K,L) \, | \, L \subset H \}	\\
		R_3	&:= \Bigg\{ (K,H) \, \Bigg| \, \begin{array}{c}
			\t{there is $n \geq 0$ and subgroups $H_0 , H_1 , \cdots , H_n \subset G$} \\
			\t{such that $K = H_0 R_2 H_1 R_2 \cdots R_2 H_n = H$}
			\end{array}
		\Bigg\} .
	\end{align*}
Thus, we close $R$ under conjugation to get $R_1$, we close $R_1$ under restriction to get $R_2$, and we take the reflexive and transitive closure of $R_2$ to get $R_3$.
\end{constr}

\begin{thm}\label{thm:transysgen}Suppose $R$ is a binary relation on $\b{Sub}(G)$ that refines inclusion. Then $\la R \ra := R_3$ is the transfer system generated by $R$, i.e. $R_3$ the smallest transfer system that contains $R$.
\end{thm}

\begin{proof} Let $R$ be a binary relation on $\b{Sub}(G)$ that refines inclusion. Then $R = R_0 \subset R_1 \subset R_2 \subset R_3$, and if $S$ is any $G$-transfer system that contains $R$, then its closure properties imply that it must also contain $R_3$. Thus, the argument will be complete once we prove that $R_3$ is a transfer system.

To start, observe that $R_2$ is closed under conjugation and restriction, and that it refines inclusion. Now consider $R_3$. It is a preorder by construction, and it refines inclusion because $R_2$ does. Therefore $R_3$ is also antisymmetric. Conjugating $R_2$-chains proves that $R_3$ is closed under conjugation. To see that $R_3$ is closed under restriction, suppose that the chain $K = H_0 R_2 H_1 R_2 \cdots R_2 H_n = H$ witnesses the relation $K R_3 H$, and that $L \subset H$. Let $L_i = L \cap H_i$. Restricting the relation $H_i R_2 H_{i+1}$ to $L_{i+1}$ yields $L_i = (L_{i+1} \cap H_i) R_2 L_{i+1}$ for $0 \leq i < n$. We obtain a chain $(L \cap K) = L_0 R_2 L_1 R_2 \cdots R_2 L_n = L$ that witnesses $(L \cap K) R_3 L$.
\end{proof}

Here is how Construction \ref{constr:transysgen} works in practice.

\begin{ex}\label{ex:genS3ts} We compute the $\Sigma_3$-transfer system generated by $C_2 \to \Sigma_3$, where $C_2 = \la (12) \ra$. Recall the notation from Figure \ref{fig:S3}.
	\[
	\begin{array}{|c|c|c|c|}
	\hline
	R_0	&	R_1	&	R_2	&	R_3	\\
	\hline
	\begin{tikzpicture}[scale=0.4,baseline=0.2mm]
		\node(S) at (0,0) {$\cdot$};
		\node(W) at (-1.5,1) {$\cdot$};
		\node(C) at (0,1) {$\cdot$};
		\node(E) at (1.5,1) {$\cdot$};
		\node(N) at (0,2) {$\cdot$};
		\node(3) at (0.6,1.2) {$\cdot$};
		
%
%
		\draw (-1.5,1) -- (0,2) ;
	\end{tikzpicture}
	&
	\begin{tikzpicture}[scale=0.4,baseline=0.2mm]
		\node(S) at (0,0) {$\cdot$};
		\node(W) at (-1.5,1) {$\cdot$};
		\node(C) at (0,1) {$\cdot$};
		\node(E) at (1.5,1) {$\cdot$};
		\node(N) at (0,2) {$\cdot$};
		\node(3) at (0.6,1.2) {$\cdot$};
		
%
%
		\draw (-1.5,1) -- (0,2) ;
		\draw (0,1) -- (0,2) ;
		\draw (1.5,1) -- (0,2) ;
	\end{tikzpicture}
	&\begin{tikzpicture}[scale=0.4,baseline=0.2mm]
		\node(S) at (0,0) {$\cdot$};
		\node(W) at (-1.5,1) {$\cdot$};
		\node(C) at (0,1) {$\cdot$};
		\node(E) at (1.5,1) {$\cdot$};
		\node(N) at (0,2) {$\cdot$};
		\node(3) at (0.6,1.2) {$\cdot$};
		
		\draw (0,0) -- (-1.5,1) ;
		\draw (0,0) -- (0,1) ;
		\draw (0,0) -- (1.5,1) ;
		
		\draw (0,0) -- (0.6,1.2) ;
		
		\draw (-1.5,1) -- (0,2) ;
		\draw (0,1) -- (0,2) ;
		\draw (1.5,1) -- (0,2) ;
	\end{tikzpicture}
	&
	\begin{tikzpicture}[scale=0.4,baseline=0.2mm]
		\node(S) at (0,0) {$\cdot$};
		\node(W) at (-1.5,1) {$\cdot$};
		\node(C) at (0,1) {$\cdot$};
		\node(E) at (1.5,1) {$\cdot$};
		\node(N) at (0,2) {$\cdot$};
		\node(3) at (0.6,1.2) {$\cdot$};
		
		\draw (0,0) -- (-1.5,1) ;
		\draw (0,0) -- (0,1) ;
		\draw (0,0) -- (1.5,1) ;
		
		\draw (0,0) -- (0.6,1.2) ;
		
		\draw plot [smooth, tension=1.5] coordinates {(0,0) (-0.5,1) (0,2)};
		\draw (-1.5,1) -- (0,2) ;
		\draw (0,1) -- (0,2) ;
		\draw (1.5,1) -- (0,2) ;
	\end{tikzpicture}
	\\
	\hline
	\end{array}
	\]
Strictly speaking, each dot $\cdot$ above represents a relation $H \to H$, and
	\[
	R_0 = \{ (H,H) \, | \, H \subset \Sigma_3 \} \cup \{ (C_2 , \Sigma_3) \}.
	\]
This distinction is irrelevant because $\la R_0 \ra = \la ( C_2 , \Sigma_3 ) \ra$. We produced Figures \ref{fig:Cpn}--\ref{fig:S3} by performing calculations like these ad nauseum, and then analyzing the results.
\end{ex}

There are a few things we can say about the transfer system $\la R \ra$ on general grounds. To start, Theorem \ref{thm:transysgen} implies the following rough bounds. We call a relation $K \to H$ \emph{nontrivial} if $K \neq H$.

\begin{prop}\label{prop:boundtransysgen}Let $R$ be a binary relation on $\b{Sub}(G)$ that refines inclusion, and let $N \subset G$ be a normal subgroup.
	\begin{enumerate}
		\item{}Suppose that for every relation $KRH$, we have $H \subset N$. Then $H \subset N$ for every nontrivial relation $(K,H) \in \la R \ra$.
		\item{}Suppose that for every relation $KRH$, we have $N \subset K$. Then $H \not\subset N$ for every nontrivial relation $(K,H) \in \la R \ra$.
	\end{enumerate}
\end{prop}

\begin{proof}We start with $(1)$. Assume that $K R_0 H$ implies $H \subset N$. Then $K R_1 H$ implies $H \subset N$, because $N$ is normal, and $K R_2 H$ implies $H \subset N$ from the transitivity of $\subset$. Finally, if $(K,H) \in R_3$ is nontrivial, then there is a chain $K = H_0 R_2 H_1 R_2 \cdots R_2 H_n = H$ with $n > 0$, and $H_{n-1} R_2 H_n$ implies $H = H_n \subset N$.

Now consider $(2)$. Assume that $K R_0 H$ implies $N \subset K$. Then $K R_1 H$ implies $N \subset K$ because $N$ is normal. Now suppose that $K R_2 H$. We shall prove that if $H \subset N$, then $K = H$. For in this case, there are subgroups $K' \subset H' \supset L'$ such that $K' R_1 H'$ and $(K,H) = (L' \cap K' , L')$. If $H \subset N$, then $L' = H \subset N \subset K'$ and therefore $K = L' \cap K' = L' = H$. Finally, we prove that for every $(K,H) \in R_3$, if $H \subset N$, then $K = H$. For suppose $K = H_0 R_2 H_1 R_2 \cdots R_2 H_n = H \subset N$ for $n \geq 0$. If $n = 0$, there is nothing to check. If $n > 0$, then since $R_2$ refines inclusion, we have $H_{i+1} \subset N$ and $H_i R_2 H_{i+1}$ for every $0 \leq i < n$. It follows from the above that $K = H_0 = H_1 = \cdots = H_n = H$.
\end{proof}

We now identify $\la R \ra$ in a two simple cases. We assume that all relations in $R$ have a shared, normal source or a shared, normal target.

\begin{prop}\label{prop:multitsKnorm} Suppose $G$ is a finite group, $K \subset G$ is a normal subgroup, and $K \subset H_1,\dots,H_n \subset G$ are subgroups such that the set $\{H_1,\dots,H_n\}$ is closed under conjugation by elements of $G$. Then $\la (K,H_i) \, | \, 1 \leq i \leq n \ra$ is equal to the relation
	\[
	\to \,\, = \{ (M,M) \, | \, M \subset G \} \cup \bigcup_{i=1}^n \{ (M \cap K , M) \, | \, M \subset H_i \}.
	\]
\end{prop}

\begin{proof} Let $R = \{ (K,H_i) \, | \, 1 \leq i \leq n \}$ and keep notation as in Construction \ref{constr:transysgen}. Then $R = R_0 = R_1$, and $R_2 = \bigcup_{i=1}^n \{ (L \cap K , L) \, | \, L \subset H_i \}$.

Suppose that $L R_3 M$. Then either $L = M$, or there is a chain of relations $L = L_0 R_2 L_1 R_2 \cdots R_2 L_m = M$ for some $m > 0$. The relation $L \to M$ is trivial in the former case, so assume the latter is true. Then for each $1 \leq j \leq m$, we have $L_{j-1} = L_j \cap K$ and $L_j \subset H_{i_j}$ for some $1 \leq i_j \leq n$. Therefore $L_0 = L_1 \cap K = L_2 \cap K = \cdots = L_m \cap K$, so that $(L,M) = (M \cap K,M)$ and $M \subset H_i$ for some $1 \leq i \leq n$. Therefore $R_3$ refines $\to$.

Conversely, suppose $L \to M$ and write $\rightsquigarrow \,\, = \la R \ra = R_3$. If $L = M$, then $L \rightsquigarrow M$ by reflexivity. Now suppose $L = M \cap K$, where $M \subset H_i$ for some $i$. The relation $K \rightsquigarrow H_i$ holds by definition, and hence $L \rightsquigarrow M$ holds by restriction. Therefore $\to$ refines $\rightsquigarrow \,\, = R_3$.
\end{proof}

If the set $\{H_1,\dots, H_n\}$ is not closed under conjugation, we close up and then apply Proposition \ref{prop:multitsKnorm}. This computes $\la (K,H) \ra$ for any normal subgroup $K \subset G$.

The next observation is useful in the dual computation, and in Proposition \ref{prop:discind}.

\begin{lem}\label{lem:transysint} Suppose $\to$ is a $G$-transfer system and $K_1, \dots, K_n \subset H \subset G$ are subgroups such that $K_i \to H$ for every $i=1,\dots,n$. Then $K_1 \cap \cdots \cap K_n \to H$.
\end{lem}

\begin{proof} We have $K_1 \to H$, and for any $i = 1 , \dots , n-1$, restricting $K_{i+1} \to H$ along $\bigcap_{j=1}^i K_j \subset H$ gives $\bigcap_{j=1}^{i+1} K_j \to \bigcap_{j=1}^i K_j$. Therefore there is a chain $\bigcap_{j=1}^n K_j \to \bigcap_{j=1}^{n-1} \to \cdots \to K_1 \to H$, and $\bigcap_{j=1}^n K_j \to H$ follows by transitivity.
\end{proof}

\begin{prop}\label{prop:indGorbab} Suppose $G$ is finite group, $H \subset G$ is a normal subgroup, and $K_1 , \dots, K_n \subset H$ are subgroups such that the set $\{K_1,\dots,K_n\}$ is closed under conjugation by elements of $G$. Then $\la (K_i,H) \, | \, 1 \leq i \leq n \ra$ is equal to the relation
	\[
	\to \,\, = \Bigg\{ (M,M) \, \Bigg| \, M \subset G \Bigg\} \cup \Bigg\{ (M \cap K_{i_1} \cap \cdots \cap K_{i_m} , M ) \, \Bigg| \, 
		\begin{array}{c}
			M \subset H \t{ and }	\\
			1 \leq i_1,\dots,i_m \leq n
		\end{array}
	\Bigg\} .
	\]
\end{prop}

\begin{proof} Let $R = \{ (K_i,H) \, | \, 1 \leq i \leq n\}$ and keep notation as in Construction \ref{constr:transysgen}. Then $R = R_0 = R_1$, and $R_2 = \bigcup_{i=1}^n \{ (L \cap K_i , L) \, | \, L \subset H \}$.

Suppose that $L R_3 M$. As in Proposition \ref{prop:multitsKnorm}, we may assume there is a chain of relations $L = L_0 R_2 L_1 R_2 \cdots R_2 L_m = M$ for some $m > 0$. For every $1 \leq j \leq m$, we have $L_{j-1} = L_j \cap K_{i_j}$ for some $1 \leq i_j \leq n$ and $L_j \subset H$. Therefore $L_0 = L_1 \cap K_{i_1} = L_2 \cap K_{i_1} \cap K_{i_2} = \cdots = L_m \cap \bigcap_{j=1}^m K_{i_j}$, so that $(L,M) = (M \cap \bigcap_{j=1}^m K_{i_j} , M)$ for some $M \subset H$. Therefore $R_3$ refines $\to$.

Conversely, suppose $L \to M$ and write $\rightsquigarrow \,\, = \la R \ra = R_3$. The relation $L \rightsquigarrow M$ is trivial if $L=M$, so assume $L = M \cap K_{i_1} \cap \cdots \cap K_{i_m}$ for some $M \subset H$ and indices $1 \leq i_1, \dots , i_m \leq n$. The relation $K_{i_j} \rightsquigarrow H$ holds for all $j$ by definition, hence $\bigcap_{j=1}^m K_{i_j} \rightsquigarrow H$ by Lemma \ref{lem:transysint}, and hence $L = M \cap \bigcap_{j=1}^m K_{i_j} \rightsquigarrow M$ by restriction. Therefore $\to$ refines $\rightsquigarrow \,\ = R_3$.
\end{proof}

If the set $\{K_1,\dots,K_n\}$ is not closed under conjugation, we close up and then apply Proposition \ref{prop:indGorbab}. This computes $\la (K,H) \ra$ for any normal subgroup $H \subset G$.

If $K,H \subset G$ are both normal, then Propositions \ref{prop:multitsKnorm} and \ref{prop:indGorbab} have the following common specialization.

\begin{cor}\label{cor:indgen1orb} Suppose $G$ is a finite group, $H,K \subset G$ are normal subgroups of $G$, and $K \subset H$. Then $\la (K,H) \ra = \{ (M,M) \, | \, M \subset G \} \cup \{ (M \cap K,M) \, | \, M \subset H \}$.
\end{cor}

The transfer system $\to \,\, = \la (K,H) \ra$ can be quite complicated when neither $K$ nor $H$ is normal in $G$, but we can say the following for certain. Recall that the normal closure of $H$ is the join of all conjugates of $H$ in $G$, and dually, the normal core of $K$ is the intersection of all conjugates of $K$ in $G$. Proposition \ref{prop:boundtransysgen} bounds $\la (K,H) \ra$ above and below by these subgroups. Additionally, Lemma \ref{lem:transysint} implies that for any $g_1 , \dots, g_n \in NH$, we have $\bigcap_{i=1}^n g_i K g_i^{-1} \to H$.

We conclude by recasting Construction \ref{constr:transysgen} in terms of indexing systems and indexing categories. We start with indexing systems. Suppose that $\b{O}$ is a set of orbits $H/K$, for some subgroups $H \subset G$. Define the graph $\to_{\b{O}}$ of $\b{O}$ exactly as in Definition \ref{defn:graphindsys}:
	\[
	K \to_{\b{O}} H \quad\t{if and only if}\quad K \subset H \t{ and } H/K \in \b{O}.
	\]
Thus $\to_{\b{O}}$ is a binary relation on $\b{Sub}(G)$ that refines inclusion, and the transfer system $\la \to_{\b{O}} \ra$ is well-defined. Recall the isomorphism $\to_\bullet \, : \b{Ind}(G) \rightleftarrows \b{Tr}(G) : \c{I}_{\bullet}$ of Theorem \ref{thm:transys}.

\begin{prop}\label{thm:indgenO} Suppose that $\b{O}$ is a set of orbits. Then $\c{I}_{\la \to_{\b{O}} \ra}$ is the indexing system generated by $\b{O}$. Equivalently, $\to_{\la \b{O} \ra} = \la \to_{\b{O}} \ra$.
\end{prop}

\begin{proof}For any indexing system $\c{I}$, we have:
	\[
	\b{O} \subset \c{I}
	\,\, \iff \,\,
	\t{$\to_{ \b{O}}$ refines $\to_{ \c{I}}$}
	\,\, \iff \,\,
	\t{$\la \to_{ \b{O}} \ra$ refines $\to_{ \c{I}}$}
	\,\, \iff \,\,
	\t{$\c{I}_{\la \to_{ \b{O}} \ra} \subset \c{I}$.}
	\]
Taking $\c{I} = \c{I}_{\la \to_{ \b{O}} \ra}$ proves that $\b{O}$ is contained in the indexing system  $\c{I}_{\la \to_{ \b{O}} \ra}$, and the equivalences above prove that $\c{I}_{\la \to_{ \b{O}} \ra}$ is the least such indexing system.
\end{proof}

\begin{cor}\label{cor:indgenO}Suppose that $\b{O}$ is a set of orbits, and let $\la \b{O} \ra$ be the indexing system that it generates. Then $H/K \in \la \b{O} \ra$ if and only if $(K,H) \in \la \to_{ \b{O}} \ra$. 
\end{cor}

Now for indexing categories. Let $\s{O}_G^{\pi}$ be the wide subcategory of $\s{O}_G$ that consists of all projection maps of the form $\pi(gK) = gH : G/K \to G/H$, for some subgroups $K \subset H \subset G$. Suppose $\s{G} \subset \s{O}_G^{\pi}$ is a wide subgraph, by which we mean a sub-directed graph of $\s{O}_G^{\pi}$ that contains all objects of $\s{O}_G^{\pi}$. We define a relation $\to_{\s{G}}$ on $\b{Sub}(G)$ by
	\[
	K \to_{\s{G}} H \quad\t{if and only if}\quad K \subset H \t{ and } (\pi : G/K \to G/H) \in \s{G}.
	\]
Thus, $\to_{\s{G}}$ is a binary relation on $\b{Sub}(G)$ that refines inclusion, and the transfer system $\la \to_{\s{G}} \ra$ is well-defined. Recall the isomorphism $\b{Set}^G_\bullet : \b{Tr}(G) \rightleftarrows \b{IndCat}(G) : \, \to_\bullet$ of Corollary \ref{cor:transysBH2}. The next result is proven the same way as Proposition \ref{thm:indgenO}.

\begin{prop}\label{prop:indgenG} Suppose $\s{G}$ is a wide subgraph of $\s{O}_G^{\pi}$. Then $\b{Set}^G_{\la \to_{\s{G}} \ra}$ is the indexing category generated by $\s{G}$. Equivalently, $\to_{\la \s{G} \ra} = \la \to_{\s{G}} \ra$.
\end{prop}

\begin{cor} Suppose $\s{G}$ is a wide subgraph of $\s{O}_G^{\pi}$ and let $\la \s{G} \ra$ be the indexing category that it generates. Then a morphism $f : S \to T$ is in $\la \s{G} \ra$ if and only if $(G_s , G_{f(s)}) \in \la \to_{\s{G}} \ra$ for every $s \in S$.
\end{cor}


\begin{thebibliography}{99}

\bibitem{BBR}
S. Balchin, D. Barnes, and C. Roitzheim.
$N_\infty$-operads and associahedra.
Preprint. arXiv:1905.03797.

\bibitem{BH}
A.J. Blumberg and M.A. Hill.
Operadic multiplications in equivariant spectra, norms, and transfers.
Adv. Math. 285 (2015), 658--708.

\bibitem{BH2}
A.J. Blumberg and M.A. Hill.
Incomplete Tambara functors.
Algebr. Geom. Topol. 18 (2018), no. 2, 723--766.

\bibitem{BonPer}
P. Bonventre and L.A. Pereira.
Genuine equivariant operads.
Preprint. arXiv:1707.02226.

\bibitem{GM}
B.J. Guillou and J.P. May.
Equivariant iterated loop space theory and permutative G-categories.
Algebr. Geom. Topol. 17 (2017), no. 6, 3259--3339.

\bibitem{GutWhite}
J.J. Guti\'{e}rrez and D. White.
Encoding equivariant commutativity via operads. 
Algebr. Geom. Topol. 18 (2018), no. 5, 2919--2962. 

\bibitem{MayGILS}
J.P. May.
The geometry of iterated loop spaces.
Lectures Notes in Mathematics, Vol. 271. Springer-Verlag, Berlin-New York, 1972. viii+175 pp.

\bibitem{Rubin}
J. Rubin.
Combinatorial $N_\infty$ operads.
Preprint. arXiv:1705.03585.

\end{thebibliography}
\end{document}